\documentclass{amsart}

\usepackage{amsfonts, amssymb, amsmath, amsthm}                               % AMS symbols

\newtheorem{theorem}{Theorem} 	      	      	                              % Theorem environment
\newtheorem{corollary}[theorem]{Corollary}     	      	      	      	      % Corollary environment
\newtheorem{lemma}[theorem]{Lemma}     	       	      	      	      	      % Lemma environment
\newtheorem{proposition}[theorem]{Proposition} 	      	      	      	      % Proposition environment
\newtheorem*{remark}{Remark}                                                  % Remark environment
\numberwithin{equation}{section}                                              % Equation numbering
\numberwithin{theorem}{section}                                               % Theorem numbering

\newcommand{\ul}[1]{\underline{#1}}                                           % Underline
\newcommand{\mf}[1]{\mathfrak{#1}}                                            % Mathfrak
\newcommand{\mc}[1]{\mathcal{#1}}                                             % Mathcal
\newcommand{\Z}{\mathbb{Z}}                                                   % Integers
\newcommand{\R}{\mathbb{R}}                                                   % Real numbers
\newcommand{\C}{\mathbb{C}}                                                   % Complex numbers
\newcommand{\Sph}{\mathbb{S}}                                                 % Unit sphere

\newcommand{\paren}[1]{\left(#1\right)}                                       % Resized parentheses
\newcommand{\brak}[1]{\left[#1\right]}                                        % Resized brackets
\newcommand{\abs}[1]{\left|#1\right|}                                         % Resized absolute value
\newcommand{\nabs}[1]{\left\|#1\right\|}                                      % Resized norm

\DeclareMathOperator{\real}{Re}                                               % Real component
\DeclareMathOperator{\imag}{Im}                                               % Imaginary component

\DeclareMathOperator{\trace}{tr}                                              % Trace

\newcommand{\lapl}{\Delta}                                                    % Laplacian
\newcommand{\cint}{{\begingroup\textstyle\smallint\endgroup}}                 % Covariant integral

\newcommand{\ass}[2]{{\bf (#1)}${}_{#2}$}                                     % Assumption condition

\allowdisplaybreaks

\begin{document}

\title[New Tensorial Estimates]{New tensorial estimates in Besov spaces for time-dependent $(2 + 1)$-dimensional problems}
\author{Arick Shao}

\address{Department of Mathematics, University of Toronto, Toronto, ON, Canada M5S 2E4}
\email{ashao@math.utoronto.ca}


\begin{abstract}
In this paper, we consider various tensorial estimates in geometric Besov-type norms on a one-parameter foliation of surfaces with evolving geometries.
Moreover, we wish to accomplish this with only very weak control on these geometries.
Several of these estimates were proved in \cite{kl_rod:cg, kl_rod:stt}, but in very specific settings.
A primary objective of this paper is to significantly simplify and make more robust the proofs of the estimates.
Another goal is to generalize these estimates to more abstract settings.
In \cite{alex_shao:nc_inf}, we will apply these estimates in order to consider a variant of the problem in \cite{kl_rod:cg}, that of a truncated null cone in an Einstein-vacuum spacetime extending to infinity.
This analysis will then be used in \cite{alex_shao:bondi} to study and to control the Bondi mass and the angular momentum under minimal conditions.
\end{abstract}

\maketitle

\section{Introduction} \label{sec.intro}

In this paper, we investigate a number of tensorial estimates in general $(1 + 2)$-dimensional geometric settings.
For example, in the Euclidean setting $I \times \R^2$, where $I = [0, 1]$, a known integrated product estimate is
\begin{align}
\label{eqi.est_trace} \nabs{ \int_0^1 \partial_t \phi (s, \cdot) \psi (s, \cdot) ds }_{ B^0_{2, 1} ( \R^2 ) } \lesssim \| \phi \|_{ H^1 ( I \times \R^2 ) } \| \psi \|_{ H^1 ( I \times \R^2 ) } \text{,}
\end{align}
where $\phi, \psi$ are sufficiently nice functions on $I \times \R^2$.
This estimate was proved in \cite[Sect. 3]{kl_rod:stt} using classical Littlewood-Paley decompositions.
A corresponding non-integrated product estimate, again on $I \times \R^2$, is
\begin{align}
\label{eqi.est_prod} \| \phi (t, \cdot) \psi (t, \cdot) \|_{ B^0_{2, 1} ( \R^2 ) } \lesssim \| \phi \|_{ H^1 ( I \times \R^2 ) } \| \psi \|_{ H^1 ( I \times \R^2 ) } \text{,} \qquad t \in I \text{.}
\end{align}
In addition, on the Euclidean cylinder $I \times \Sph^2$, one has the elliptic estimate
\begin{equation} \label{eqi.est_ell} \| \nabla^2 \phi (t, \cdot) \|_{ B^0_{2, 1} (\Sph^2) } \lesssim \| \lapl \phi (t, \cdot) \|_{ B^0_{2, 1} (\Sph^2) } \text{,} \qquad t \in I \text{,} \end{equation}
for appropriately defined Besov spaces on $\Sph^2$.

The general questions we wish to pose are the following:
\begin{enumerate}
\item Do analogues of \eqref{eqi.est_trace}-\eqref{eqi.est_ell}, as well as other related estimates, hold when $\R^2$ and $\Sph^2$ are replaced by another surface $\mc{S}$, and when the geometry of $\mc{S}$ is allowed to change with the time variable $t$?

\item Can we establish these estimates with only very weak assumptions on how the geometries of $\mc{S}$ evolve with time, and, for the elliptic estimates, with only similarly weak assumptions on the curvatures of these $\mc{S}$?

\item Can we retain these estimates even when the fields under consideration are no longer scalar but are tensorial in nature?
\end{enumerate}
In this paper, we will examine these questions for rather general $\mc{S}$.
The main results are stated in Theorems \ref{thm.est_prod_elem}, \ref{thm.est_prod}, \ref{thm.est_trace_sh}, \ref{thm.sharp_trace}, \ref{thm.besov_weak}, and \ref{thm.besov_impr}.

\subsection{Motivations} \label{sec.intro_motive}

The main motivation for this problem comes from mathematical general relativity, in the analysis of regular null cones in spacetimes.
More specifically, one wishes to control the geometry of the null cone given the assumption of bounded or small curvature flux.
Here, ``controlling the geometry" refers to quantitative control for various connection coefficients of the cone, as well as a lower bound on the null conjugacy radius of the cone.

The \emph{curvature flux} refers to an $L^2$-norm along a null hypersurface of certain components of the spacetime curvature tensor.
It is a fundamental quantity for dealing with local energy estimates involving the curvature.
In particular, the curvature flux arises naturally from the Bel-Robinson tensor, which one can think of as an analogue of the stress-energy tensor for the spacetime Weyl curvature.

A number of results involving mathematical relativity and the Einstein equations have relied heavily on variations of the curvature flux.
Examples include the stability of Minkowski spacetimes (e.g., \cite{bie_zip:stb_mink, chr_kl:stb_mink, kl_nic:stb_mink}), improved breakdown and continuation criteria for the Einstein equations (e.g., \cite{kl_rod:bdc, parl:bdc, shao:bdc_nv, shao:bdc_nvp, wang:ibdc}), and the formation of black holes and trapped surfaces (e.g., \cite{chr:gr_bh, kl_rod:fbh}), among several others.
In particular, regarding the breakdown criteria for the Einstein equations, a major component of the proofs of this family of results was precisely that of controlling the geometry of null cones by the curvature flux.
\footnote{More accurately, the geometry of the null cone is controlled by the curvature flux, properties of the time foliation, and properties of the matter field (in non-vacuum situations).}
This was by far the most technically demanding portion of the argument, involving an elaborate bootstrap procedure and the construction of a geometric tensorial Littlewood-Paley theory, cf. \cite{kl_rod:glp}.
Furthermore, variants of these same ideas have been applied to recent work on the bounded $L^2$-curvature conjecture (\cite{kl_rod_szf:blc, szf:par1, szf:par2, szf:par3, szf:par4}).

Control of the geometry of null cones by its curvature flux was first proved in \cite{kl_rod:cg}, for geodesically foliated truncated null cones beginning from a $2$-sphere in an Einstein-vacuum spacetime; technical aspects of the argument were also included in \cite{kl_rod:glp, kl_rod:stt, kl_rod:rin}.
The result was extended to null cones beginning from a point in \cite{wang:cg, wang:cgp}, which was needed in the breakdown criteria problem.
Other versions of this result include \cite{parl:bdc, shao:bdc_nv, wang:ibdc}, which dealt with time foliated null cones.
In particular, \cite{shao:bdc_nv} extended the result to Einstein-scalar and Einstein-Maxwell spacetimes.

There are many reasons why this problem in general is technically demanding.
The first is that the assumptions regarding the curvature flux, an $L^2$-norm of the curvature \emph{along the entire null cone}, grants only very weak regularity for the geometries of the spheres foliating the cone.
In particular, standard coordinate-based methods fail to work due to the inherent lack of regularity.
Since the differential equations under consideration were tensorial in nature, one often had to devise methods to deal with these quantities in an invariant manner.

Another difficulty is that one required several bilinear product estimates in order to prove the required bounds.
In the Euclidean setting, many of these estimates, in particular \eqref{eqi.est_trace}, were proved in \cite[Sect. 3]{kl_rod:stt}.
Next, the main results of \cite{kl_rod:stt} extended these bounds to the specific setting of geodesically foliated truncated null cones beginning from a $2$-sphere in an Einstein-vacuum spacetime.
In \cite{wang:cg}, similar estimates were proved for corresponding null cones beginning from a point.

Furthermore, regarding the spheres foliating the null cone, one does not have even $L^2$-control on their curvatures.
This made proving even the simplest of elliptic estimates rather laborious.
Even worse, one also required more elaborate tensorial elliptic estimates in Besov spaces.
In fact, a significant part of \cite{wang:cg} was dedicated to the derivation of these Besov-elliptic bounds.

These estimates in \cite{kl_rod:stt, wang:cg} form the starting point for the discussions of this paper.
In relation to \cite{kl_rod:stt, wang:cg}, our paper will accomplish two primary goals:
\begin{itemize}
\item We extend several estimates in \cite{kl_rod:stt, wang:cg} to more abstract settings.

\item We give proofs for our main estimates that are significantly shorter and simpler than the corresponding proofs in \cite{kl_rod:stt, wang:cg}.
\end{itemize}
Most of the main estimates we establish here are analogues or generalizations of product estimates in \cite{kl_rod:stt, wang:cg}, though in a more abstract setting.
In simplifying the proof, we adopt different methods for establishing the desired bounds.

We note that the estimates proved in this paper are retroactively applicable to previous results pertaining to controlling the geometry of null cones by the curvature flux, e.g., \cite{kl_rod:cg, parl:bdc, shao:bdc_nv, wang:cg, wang:cgp, wang:ibdc}.
In all the above works, their settings are special cases of the abstract conditions considered here.

Regarding current applications, the immediate goal is to apply the estimates of this paper to another variant of the ``null cone with finite curvature flux" problem: that of an outgoing truncated null cone extending indefinitely toward infinity.
Using the same ideas as in \cite{kl_rod:cg}, one hopes to obtain corresponding control for the null geometry.
This is accomplished in \cite{alex_shao:nc_inf}, under the assumption that the null cone is sufficiently close to a Schwarzschild null cone, in the sense of a weighted curvature flux.
An additional goal is to control the associated Bondi mass of this null cone by this weighted flux; this is discussed in \cite{alex_shao:bondi}.

All the problems mentioned above will require corresponding bilinear product and elliptic estimates, for precisely the same reasons as in \cite{kl_rod:cg, parl:bdc, shao:bdc_nv, wang:cg}.
Thus, an important motivation behind the work in this paper is to establish these estimates in a sufficiently generalized setting so that they can be applied with relatively little effort to all potential variant problems.
In this paper, we establish these general estimates, while \cite{alex_shao:bondi, alex_shao:nc_inf} demonstrate in detail how the framework developed here is applied to the ``null cone with finite curvature flux" class of problems.

\subsection{The Abstract Formalism} \label{sec.intro_abs}

As mentioned before, the bilinear product estimates in \cite{kl_rod:stt} and the Besov-elliptic estimates found in \cite{kl_rod:cg} applied to an extremely specific setting: geodesically foliated truncated null cones in vacuum spacetimes beginning from a sphere.
The corresponding estimates in \cite{wang:cg} are equally specific: same null cones, but beginning from a point.
Suppose one wishes to work instead with double-null-foliated null cones or in non-vacuum backgrounds, for example.
One certainly expects that the general proof templates from \cite{kl_rod:stt, wang:cg} can be adapted to these cases.
However, since the arguments presented in \cite{kl_rod:stt, wang:cg} were highly dependent on the specific setting, all the bilinear product estimates and the elliptic estimates would in principle have to be redone.
This becomes rather unsavory, since the arguments were quite lengthy and technically elaborate.

As a result, we wish to frame the hypotheses and conclusions of these estimates in a more abstract framework.
In particular, we wish to express our results in a sufficiently general manner so that the null cone problems in \cite{kl_rod:cg, kl_rod:stt, wang:cg}, as well as variations of these problems, can be expressed in terms of this abstract formalism.
Indeed, the settings of \cite{kl_rod:cg, kl_rod:stt, parl:bdc, shao:bdc_nv, wang:cg} can be shown to satisfy the abstract assumptions we impose for our main estimates in this paper.

Throughout the paper, we will work with a one-parameter foliation
\[ \mc{N} = [0, \delta] \times \mc{S} \text{,} \]
where $\delta > 0$ is fixed, and where $\mc{S}$ is a two-dimensional manifold.
To describe the geometric information of this system, we impose a family of evolving Riemannian metrics on $\mc{S}$, parametrized by the interval $[0, \delta]$.
In other words, for each $0 \leq \tau \leq \delta$, we let $\gamma [\tau]$ denote a Riemannian metric on the cross-section
\[ \mc{S}_\tau = \{ \tau \} \times \mc{S} \simeq \mc{S} \text{,} \]
such that these $\gamma [\tau]$'s vary smoothly with respect to $\tau$.
From this, we can construct other geometric objects on $\mc{N}$ (e.g., covariant derivatives with respect to the $\gamma [\tau]$'s), which amount to smoothly varying aggregations of objects on the $\mc{S}_\tau$'s.

The quantities we will analyze throughout the paper are represented as \emph{horizontal tensor fields} on $\mc{N}$.
These can be interpreted as smooth tensor fields on $\mc{N}$ which are everywhere tangent to the $\mc{S}_\tau$'s.
For instance, by aggregating the metrics $\gamma [\tau]$ mentioned above into a single object $\gamma$ on $\mc{N}$, we obtain such a horizontal tensor field.
By measuring how these $\gamma [\tau]$'s evolve with respect to $\tau$, we obtain another horizontal field, which we refer to as the ``second fundamental form".

This notion of horizontal tensor fields has been used implicitly in various works in mathematical general relativity.
For example, in \cite{chr:gr_bh, chr_kl:stb_mink, kl_nic:stb_mink, kl_rod:cg}, among several others, the authors deal with various tensorial quantities on null cones that are horizontal, i.e., tangent to the spheres which foliate the cone.
Furthermore, in \cite{chr_kl:stb_mink, kl_rod:bdc}, for instance, one foliates a spacetime into Riemannian timeslices and analyzes tensorial objects on this spacetime which are tangent to these timeslices.
The formalisms we adopt for notating, describing, and analyzing these horizontal fields derive from \cite{shao:bdc_nv, shao:ksp, shao:bdc_nvp}, which applied these notions for similar purposes.

One major component of the formalism is our definition of a \emph{covariant evolutionary derivative} (i.e., with respect to parameter $\tau \in [0, \delta]$) of horizontal tensor fields.
In comparison to previous results on null cones, e.g., \cite{chr:gr_bh, chr_kl:stb_mink, kl_nic:stb_mink, kl_rod:cg}, one can show this coincides with the derivative operators $\nabla_L$ and $\nabla_{\ul{L}}$, i.e., the appropriate horizontal projections of the corresponding \emph{spacetime} covariant derivatives.
In this sense, our construction generalizes those used in previous works on null cones.

\begin{remark}
We also note that some other unrelated topics can be connected to the abstract concepts used here.
One well-known example involves the Ricci flow, in particular with a process known as ``Uhlenbeck's trick", used to derive a covariant evolutionary equation for the Riemann curvature.
In \cite[Sect. 6.3]{andr_hop:ric_sph}, it was shown that this trick could be naturally expressed in terms of an abstract formulation equivalent to the one discussed here.
For further details, see also \cite{brend:ric_sph, ham:ric_flow_four}.
\end{remark}

Another component is the ``inverse" to the above operation: a \emph{covariant integral} along the vertical direction.
Such operations were used implicitly in \cite{kl_rod:fbh} in order to estimate trace norms of various quantities.
\footnote{This is in fact quite similar to the intended applications of this paper.
The main difference with \cite{kl_rod:fbh} is that here, we postulate far less regularity, making the analysis much more difficult.}
Here, we shall make much more explicit mention and use of these integral operators.
These will play a major role in demonstrating the improved regularity one obtains from covariant derivatives and parallel frames, both in the abstract setting discussed in this paper and in our specific motivating problem of null cones with bounded curvature flux.

The main observation behind our generalization of the bilinear product estimates of \cite{kl_rod:stt, wang:cg} is that the required assumptions can be stated in terms of objects defined from our abstract formalism.
In particular, much of the assumptions are given by weak integral bounds on the second fundamental form mentioned above.
The remaining assumptions involve relatively trivial quantitative control on the geometry of a single initial leaf of the foliation $\mc{N}$.
The exact conditions required for the main estimates of this paper will be described in detail in later sections.

We also mention that the process we use to prove our main estimates can be almost directly generalized to some abstract vector bundles.
This is due to the observation that our covariant methods in this paper depend only on the presence of a bundle metric and a compatible bundle connection.
One potential application of this would be for working with corresponding null cones in Einstein-Yang-Mills spacetimes.
In this setting, one has precisely this bundle-metric-connection system when describing the Lie algebra-valued Yang-Mills curvature.

Finally, for the Besov-elliptic estimates, one also requires that $\mc{S}$ is sufficiently close to $\Sph^2$ in some weak sense, as positive curvature plays a fundamental role.
In this case, we must impose additional ``weakly spherical" assumptions: that $\mc{S}$ is diffeomorphic to $\Sph^2$, and that the curvature of $\mc{K}$ is sufficiently close to $1$.
However, we require that this closeness is in an exceedingly weak sense; we postulate $\mc{K} - 1$ to be small in a manner that is even weaker than $L^2$.
This lack of regularity is responsible for much of the difficulty behind these elliptic estimates.

\subsection{Simplification of Proofs I} \label{sec.intro_sprod}

In \cite[Sect. 3]{kl_rod:stt}, it was shown that the Euclidean analogues of these bilinear product estimates could be established with relative ease via classical Littlewood-Paley decomposition methods.
Consequently, in the much less trivial setting of null cones with bounded curvature flux, i.e., of \cite{kl_rod:cg, kl_rod:stt}, a reasonable first approach to similar estimates could be the following:
\begin{itemize}
\item Choose appropriate local coordinate systems on the cone.

\item Apply the Euclidean estimates to these coordinate systems.

\item Extract a global estimate on the null cone from these coordinate estimates.
\end{itemize}
However, such a method fails for this setting.

The fundamental reason for this failure is the extremely weak assumption of bounded curvature flux.
In particular, one has only $L^2$-integrability of the spacetime curvature in the spherical component.
A cursory examination of the geometric equations involving the null cones shows that the spatial gradient of the second fundamental form, as defined in the preceding discussion, is of the same level as the spacetime curvature.
Therefore, one expects the same spherical $L^2$-integrability for this gradient as for the spacetime curvature.

Now, if we choose natural coordinate systems as dictated by the above outline, then one can relate the Christoffel symbols associated with these coordinates to the gradient of the second fundamental form.
After some analysis, one sees that these Christoffel symbols can have at best $L^2$-integrability in the spherical component.
It turns out that this $L^2$-integrability for these connection quantities makes it such that the above reduction to the Euclidean estimates just barely fails.
In fact, if one has $L^q$-integrability for some $q > 2$, then this reduction can be recovered.
However, this regularity is unattainable, due to the stringent curvature assumptions.

As a result, \cite{kl_rod:cg, kl_rod:stt, wang:cg} resorted largely to geometric and intrinsically tensorial methods.
In \cite{kl_rod:glp}, the authors constructed a fully geometric and tensorial Littlewood-Paley theory via the heat flow.
In particular, all the required Besov-type norms and estimates were defined in terms of this theory.
The geometric nature of these constructions ultimately circumvented the need for local coordinate systems and their associated Christoffel symbols.
The price to be paid, though, was a significant amount of added technical baggage to the process, in the forms of numerous elaborate heat flow, Besov, and commutator estimates.

As a result, one is interested in simplifying, both lengthwise and technically, the proofs in \cite{kl_rod:stt} of the bilinear product estimates.
In \cite{wang:cg}, it was shown that a significant portion of the argument could in fact be made using local coordinate methods.
Although this somewhat shortened the proof, due to the difficulties mentioned above, not all of the proofs could be reduced to coordinate analyses.
Unfortunately, the difficult cases included the null cone analogues of \eqref{eqi.est_trace} and \eqref{eqi.est_prod}, which were the toughest estimates to establish in the first place.

In this paper, we claim that \emph{all} of these bilinear product estimates in \cite{kl_rod:stt, wang:cg}, including the analogues of \eqref{eqi.est_trace} and \eqref{eqi.est_prod}, do in fact reduce to the corresponding Euclidean versions.
Although this could not be done using coordinate considerations, as previously mentioned, we demonstrate here that this can be accomplished, on the other hand, using specially chosen frames.
More specifically, we wish to systematically decompose horizontal tensor fields into local scalar quantities via a collection of parallel transported horizontal frames.

In the motivating null cone settings (e.g., \cite{kl_rod:stt, wang:cg}), these parallel frames have additional regularities that were not previously exploited.
In fact, one can derive $L^4$-integrability in the spherical component for the connection coefficients associated with these frames, which is a strict improvement over frames constructed from transported coordinate systems.
This is now sufficient for reducing the geometric bilinear product estimates to their Euclidean analogues.

This additional regularity for parallel frames is closely related to the evolutionary covariant derivative.
The commutation formula between this derivative and other covariant derivatives has nicer algebraic properties than the corresponding commutator using a non-covariant evolutionary derivative.
Indeed, in the former commutator, one sees only the curl of the second fundamental form, while in the latter, one sees instead the full gradient.
While this gain seems negligible at first, it is vastly important in the case of null cones with bounded curvature flux.

In this specific setting, the curl of the second fundamental form is further related to the geometry of the ambient spacetime via the well-known Codazzi equations.
By taking this into account, one can establish slightly better estimates for the curl of the second fundamental form than for the full gradient.
This leads to the improved $L^4$-estimates for the parallel frame connection coefficients.

\begin{remark}
In the abstract formulation that we utilize in the paper, we will achieve this improved regularity for parallel frames by postulating corresponding regularity assumptions for this curl of the second fundamental form.
In particular, settings involving null cones with curvature flux will satisfy these regularity conditions.
\end{remark}

With the above process, one transforms tensorial quantities to a family of scalar quantities localized to coordinate systems.
One can then define Besov-type spaces and norms in the standard fashion with respect to these coordinate systems and an associated partition of unity.
From here, one can now apply the analogous estimates in Euclidean spaces and piece together the various scalar components into the desired tensorial estimate in our nontrivial geometric setting.

The final piece of the puzzle is to relate the coordinate-based Besov-type norms mentioned above to the corresponding geometric, invariant norms.
For this task, we provide estimates showing that, under the same assumptions as needed before, these coordinate-based Besov norms are in fact equivalent to the analogous geometric norms.
These are inspired largely by a similar argument found in \cite{wang:cg}.

Finally, we mention that a major source of difficulty in \cite{kl_rod:stt, wang:cg} is the lack of regularity for the Gauss curvatures of the level spheres of null cones.
Much of the machinery in \cite{kl_rod:stt} is related to bypassing this obstacle.
Our analysis here, on the other hand, does not involve these Gauss curvatures at all.
Moreover, our methods also avoid many of the heat flow and commutator estimates needed in \cite{kl_rod:stt}.

\subsection{Simplification of Proofs II} \label{sec.intro_sell}

While this scalar decomposition using parallel frames suffices to reduce the bilinear product estimates to their Euclidean counterparts, the elliptic estimates will need to be treated with a completely different strategy.
For our motivating setting (e.g., \cite{kl_rod:cg, wang:cg}), while the curvature cannot be controlled in $L^2$, the part that is not $L^2$-controlled has special structure.
In fact, this term can be expressed as a (horizontal) divergence of a term that has $H^{1/2}$-control.
We shall show in this paper how this structure can be exploited in order to both simplify proofs and derive improved estimates.

\begin{remark}
Again, in the abstract setup of this paper, we will assume that the curvatures have the structure that is satisfied in \cite{kl_rod:cg, wang:cg} and their variants.
\end{remark}

In \cite{kl_rod:cg, parl:bdc, shao:bdc_nv, wang:cg}, the authors deal with the curvature issue by showing that it has $H^{-1/2}$-control.
This suffices to prove the basic $L^2$-elliptic estimates that are required.
However, due to the presence of this $H^{-1/2}$-space, the process is quite complex, as it already involves both product estimates and commutator estimates involving heat flows.
On the other hand, here we can give a short and elementary proof of these same estimates by leveraging the observation that the part of the curvature that is not $L^2$-controlled has this divergence form.

In \cite{kl_rod:cg, wang:cg}, as well as any potential variant of this setting, one requires not only the aforementioned $L^2$-elliptic estimates, but also some analogues of these bounds in geometric Besov-type norms.
More specifically, one requires the property that the operators $\nabla \mc{D}^{-1}$ are bounded with respect to a geometric $B^0_{2, 1}$-type norm.
Here, $\nabla$ denotes the covariant derivative with respect to the spheres, while $\mc{D}$ denotes one of the symmetric Hodge operators used, e.g., in \cite{chr_kl:stb_mink, kl_rod:cg, wang:cg}.

This boundedness property for the operators $\nabla \mc{D}^{-1}$ was stated in \cite{kl_rod:cg} without proof.
In \cite{wang:cg}, where these estimates were first treated in detail, it was shown that one needed additional lower-order error terms on the right-hand sides of these estimates.
In fact, a significant part of \cite{wang:cg} was dedicated to addressing many of the subtleties behind proving these bounds.

In this paper, we claim, once again, that these error terms are in fact not necessary.
Furthermore, we give a far shorter proof of these Besov-elliptic estimates.
The key behind our improvements is again the divergence form of the worst terms comprising the curvatures of the spheres.
More specifically, because of this structure, we can solve for a conformal renormalization that regularizes these curvatures.

This conformal transformation can be constructed precisely to absorb this worst divergence term that is not $L^2$-bounded.
In other words, the conformally transformed curvature will be $L^2$-controlled.
Having this curvature bounded in $L^2$ serves to greatly simplify all the proofs of elliptic estimates.
Furthermore, one can show, with relatively little effort, that our desired Besov-elliptic estimates in fact hold with respect to this conformally regularized setting.

What we ultimately require, however, is that these estimates hold with respect to the \emph{original} metric.
To this end, we show that the corresponding estimates with respect to the conformally regularized metric can be pulled back to the original metric.
This is dependent on the following observations:
\begin{itemize}
\item The conformal factor can be shown to have sufficient control.

\item Most of the Hodge operators are conformally invariant.
\end{itemize}
This completes the process that proves the desired Besov-elliptic estimates.
Although the conformal transformation adds some new technology, the resulting proof is much shorter and simpler.
This also avoids the lower-order error terms found in \cite{wang:cg}, which greatly complicates much of the remaining analysis.

Finally, by using both the abstracted bilinear product and elliptic estimates established in this paper, the proofs of the main results in \cite{kl_rod:cg, parl:bdc, shao:bdc_nv, wang:cg, wang:ibdc} (i.e., control of the null geometry by the curvature flux) become far simpler.
We will demonstrate this simplification explicitly in \cite{alex_shao:nc_inf}, where we consider truncating null cones extending to infinity, under similar assumptions on curvature flux.

\subsection{Outline of Paper} \label{sec.intro_outline}

We now briefly outline the remainder of this paper.

\begin{itemize}
\item In Section \ref{sec.geom}, we discuss the analysis of tensor fields on a single Riemannian surface.
We construct in this setting many of the tools we will use throughout the paper.
We also define various regularity assumptions on the surface, and we discuss the impact these have on estimates.

\item In Section \ref{sec.fol}, we move to our main setting: a one-parameter foliation of surfaces discussed in Section \ref{sec.geom}.
We consider general situations in which the regularity conditions defined in Section \ref{sec.geom} hold uniformly on each level surface.
From these assumptions, we devise a scalar decomposition scheme that can be used to reduce geometric tensorial estimates to ones in Euclidean space.
Finally, we apply these tools to prove some basic estimates.

\item In Section \ref{sec.cfol}, we build further upon the setting of Section \ref{sec.fol} by defining abstract notions of covariant evolution.
We define covariant evolutionary derivative and integral operators, and we discuss their basic properties.
Next, we consider a class of assumptions that control how the geometries of the level surfaces can evolve.
We show that these conditions suffice to propagate any regularity on the initial surface (in the sense of Section \ref{sec.geom}) uniformly to all the remaining surfaces (in the sense of Section \ref{sec.fol}).

\item In Section \ref{sec.thm}, we revisit the scalar decomposition scheme devised in Section \ref{sec.fol}, but in terms of the propagated regularity discussed in Section \ref{sec.cfol}.
We use this construction to obtain simple proofs of the main bilinear product estimates of this paper, in particular the analogues of \eqref{eqi.est_trace} and \eqref{eqi.est_prod}.

\item In Section \ref{sec.curv}, we consider the case in which the level surfaces are ``weakly spherical", and we obtain elliptic estimates in this setting.
We define precisely the meaning of weakly spherical, and we proceed to use the structure present in this definition in order to prove various $L^2$-elliptic estimates in a completely straightfoward and elementary manner.
In the remainder of this section, we use a conformal renormalization argument in order to prove Besov-type refinements of these elliptic estimates.

\item In Section \ref{sec.nc}, we give a brief and informal discussion of how the abstract formalisms and the estimates developed in this paper relate to geodesically foliated null cones in Einstein-vacuum spacetimes, i.e., the setting of \cite{kl_rod:cg}.

\item In Appendix \ref{sec.eucl}, we very briefly review some aspects of (dyadic) Littlewood-Paley theory in Euclidean spaces.
Afterwards, we proceed to prove the Euclidean analogues of the main bilinear product estimates of this paper.

\item In Appendix \ref{sec.bcomp}, we provide the proof of a norm comparison estimate found in Section \ref{sec.fol}.
This comparison shows that, under the appropriate regularity assumptions, the geometric tensorial Sobolev and Besov norms are equivalent to their (specially chosen) coordinate-based counterparts.

\item In Appendix \ref{sec.cell}, we show that the conformally renormalized metric constructed in Section \ref{sec.curv} satisfies better elliptic estimates than the original metric.
In particular, we show that the desired Besov-elliptic estimates hold with respect to this conformally renormalized metric.
\end{itemize}

\subsection{Notations} \label{sec.intro_not}

Consider a general manifold $M$ of arbitrary dimension.
We let $\mc{C}^\infty M$ denote the space of all smooth real-valued functions on $M$.
Next, let $\mc{V}$ be a vector bundle over $M$.
Given $z \in M$, we let $\mc{V}_z$ denote the fiber of $\mc{V}$ at $z$.
Moreover, we let $\mc{C}^\infty \mc{V}$ denote the space of all smooth sections of $\mc{V}$.

Throughout the paper, we will always let $r, r_1, r_2$ and $l, l_1, l_2$ denote non-negative integers.
The prototypical examples of vector bundles over $M$ are the tensor bundles $T^r_l M$ on $M$ of rank $(r, l)$, for which the associated fiber $( T^r_l M )_z$ of $T^r_l M$ at any $z \in M$ is the tensor space on $M$ of rank $(r, l)$ at $z$.
\footnote{Here, $r$ is the contravariant rank, and $l$ is the covariant rank.}
Therefore, the elements of $\mc{C}^\infty T^r_l M$ are the smooth tensor fields of rank $(r, l)$ on $M$.

We will often use standard index notation to describe tensor and tensor fields.
Indices, given by lowercase Latin letters, will be with respect to fixed frames and coframes.
In accordance with Einstein summation notation, repeated indices indicate summations over all allowable index values.

Given nonnegative real numbers $c_1, \dots, c_m$, we write
\[ X \lesssim_{c_1, \dots, c_m} Y \]
to mean that $X \leq C Y$ for some constant $C > 0$ depending on $c_1, \dots, c_m$.
If no $c_i$'s are given, then the constant $C$ is universal.
Similarly, we write
\[ X \simeq_{c_1, \dots, c_m} Y \]
to mean that both of the following statements hold:
\[ X \lesssim_{c_1, \dots, c_m} Y \text{,} \qquad Y \lesssim_{c_1, \dots, c_m} X \text{.} \]
To shorten notations, \emph{we will generally omit the dependence of constants (i.e., the $c_i$'s in the above) in inequalities within proofs of statements.}
The precise dependence of inequalities within proofs will be implicit from references to previous propositions.

Finally, from now on, we fix the following values:
\begin{itemize}
\item Let $C > 1$ and $B > 0$ denote real constants.

\item Let $N > 0$ denote an integer constant.
\end{itemize}
These symbols will denote constants that control the regularity of our settings.

\subsection{Acknowledgements} \label{sec.intro_ack}

The author would like to thank both Sergiu Klainerman and Spyros Alexakis for their support and their mathematical insights.

\section{Geometric Analysis} \label{sec.geom}

In our analysis, we will deal with two separate settings.
\begin{enumerate}
\item A fixed $2$-dimensional Riemannian manifold.

\item A one-parameter family of copies of this $2$-dimensional manifold, equipped with a family of evolving Riemannian metrics.
\end{enumerate}
In this section, we discuss aspects of the first setting.
Throughout, we let $\mc{S}$ denote an arbitrary oriented $2$-dimensional manifold.

\subsection{Riemannian Structures} \label{sec.geom_riem}

Let $h \in \mc{C}^\infty T^0_2 \mc{S}$ be a Riemannian metric on $\mc{S}$, and let $h^{-1} \in \mc{C}^\infty T^2_0 \mc{S}$ denote the metric dual of $h$.
As usual, within index notation, $h^{-1}$ is written as simply $h$, but with superscript indices.
\footnote{In other words, the components $h^{ab}$ comprise the inverse matrix of the $h_{ab}$'s.}
Moreover, $h$ and the chosen orientation of $\mc{S}$ induce a volume form $\omega \in \mc{C}^\infty T^0_2 \mc{S}$ on $\mc{S}$.

Recall $h$ and $h^{-1}$ define pointwise tensorial inner products and norms on $\mc{S}$.
More specifically, for any $F, G \in \mc{C}^\infty T^r_l \mc{S}$, we define
\[ \langle F, G \rangle = h_{a_1 b_1} \dots h_{a_r b_r} h^{c_1 d_1} \dots h^{c_l d_l} F^{a_1 \dots a_r}{}_{c_1 \dots c_l} G^{b_1 \dots b_r}{}_{d_1 \dots d_l} \in C^\infty \mc{S} \text{,} \]
i.e., the bundle metric on $T^r_l \mc{S}$ induced by $h$.
We also define the pointwise norm:
\[ | F | = \langle F, F \rangle^\frac{1}{2} \text{.} \]
We can now use $h$ and $\omega$ to define standard integral norms:
\[ \| F \|_{ L^q_x }^q = \int_\mc{S} | F |^q d \omega \text{,} \qquad \| F \|_{ L^\infty_x } = \sup_{ x \in \mc{S} } | F | |_x \text{,} \qquad q \in [1, \infty) \text{.} \]

\begin{remark}
In the scalar case $r = l = 0$, the inner product $\langle \cdot, \cdot \rangle$ is simply multiplication of two functions, and the norm $| \cdot |$ is the absolute value.
In particular, both the inner product and the norm are independent of $h$.
\end{remark}

Furthermore, following standard conventions, we let
\[ \nabla: \mc{C}^\infty T^r_l \mc{S} \rightarrow \mc{C}^\infty T^r_{l+1} \mc{S} \text{,} \qquad \lapl: \mc{C}^\infty T^r_l \mc{S} \rightarrow \mc{C}^\infty T^r_l \mc{S} \]
denote the Levi-Civita connection and the Bochner Laplacian for $(\mc{S}, h)$, respectively.
Higher-order covariant differentials are defined iteratively:
\[ \nabla^k: \mc{C}^\infty T^r_l \mc{S} \rightarrow \mc{C}^\infty T^r_{l+k} \mc{S} \text{,} \qquad \nabla^k = \nabla \nabla^{k-1} \text{,} \qquad k > 1 \text{.} \]
In addition, we let $\mc{K} \in \mc{C}^\infty \mc{S}$ denote the Gauss curvature of $(\mc{S}, h)$.

Next, we recall the symmetric Hodge operators, as defined in \cite{chr_kl:stb_mink, kl_rod:cg}.
For completeness, we begin by defining the vector bundles on which these operators act.
The rank-$0$ and rank-$1$ bundles are defined as
\[ H_0 \mc{S} = \mc{C}^\infty \mc{S} \otimes \C \text{,} \qquad H_1 \mc{S} = \mc{C}^\infty T^0_1 \mc{S} \text{.} \]
The sections of $H_0 \mc{S}$ and $H_1 \mc{S}$ are the complex-valued smooth scalar functions and the $1$-forms on $\mc{S}$, respectively.
In addition, we define $H_2 \mc{S}$ over $\mc{S}$ to be vector bundle over $\mc{S}$ all covariant symmetric $h$-traceless horizontal $2$-tensors on $\mc{S}$.
\footnote{In other words, $A \in \mc{C}^\infty T^0_2 \mc{S}$ is in $\mc{C}^\infty H_2 \mc{S}$ iff $A_{ba} = A_{ab}$ and $h^{ab} A_{ab} \equiv 0$.}

\begin{remark}
Note that $H_0 \mc{S}$ and $H_1 \mc{S}$ are independent of $h$, while $H_2 \mc{S}$ is not.
\end{remark}

We can now define the symmetric Hodge operators as follows:
\begin{alignat*}{3}
\mc{D}_1 &: \mc{C}^\infty H_1 \mc{S} \rightarrow \mc{C}^\infty H_0 \mc{S} \text{,} &\qquad \mc{D}_1 X &= h^{ab} \nabla_a X_b - i \cdot \omega^{ab} \nabla_a X_b \text{,} \\
\mc{D}_2 &: \mc{C}^\infty H_2 \mc{S} \rightarrow \mc{C}^\infty H_1 \mc{S} \text{,} &\qquad \mc{D}_2 X_a &= h^{bc} \nabla_b X_{ac} \text{,} \\
\mc{D}_1^\ast &: \mc{C}^\infty H_0 \mc{S} \rightarrow \mc{C}^\infty H_1 \mc{S} \text{,} &\qquad \mc{D}_1^\ast X_a &= - \nabla_a ( \real X ) - \omega_a{}^c \nabla_c ( \imag X ) \text{,} \\
\mc{D}_2^\ast &: \mc{C}^\infty H_1 \mc{S} \rightarrow \mc{C}^\infty H_2 \mc{S} \text{,} &\qquad -2 \mc{D}_2^\ast X_{ab} &= \nabla_a X_b + \nabla_b X_a - h_{ab} h^{cd} \nabla_c X_d \text{.}
\end{alignat*}
Direct computations show that the $\mc{D}_i^\ast$'s are the $L^2$-adjoints of the $\mc{D}_i$'s (with respect to $h$).
Additional calculations yield the following identities:
\begin{alignat}{3}
\label{eq.hodge_sq} \mc{D}_1 \mc{D}_1^\ast &= - \lapl \text{,} &\qquad \mc{D}_1^\ast \mc{D}_1 &= - \lapl + \mc{K} \text{,} \\
\notag \mc{D}_2 \mc{D}_2^\ast &= - \frac{1}{2} \lapl - \frac{1}{2} \mc{K} \text{,} &\qquad \mc{D}_2^\ast \mc{D}_2 &= - \frac{1}{2} \lapl + \mc{K} \text{.}
\end{alignat}
Integrating by parts and applying \eqref{eq.hodge_sq} results in the following integral identities:
\begin{alignat}{3}
\label{eq.hodge_id} \int_{ \mc{S} } ( | \nabla X |^2 + \mc{K} | X |^2 ) d \omega &= \int_{ \mc{S} } | \mc{D}_1 X |^2 d \omega \text{,} &\qquad X &\in \mc{C}^\infty H_1 \mc{S} \text{,} \\
\notag \int_{ \mc{S} } ( | \nabla X |^2 + 2 \mc{K} | X |^2 ) d \omega &= 2 \int_{ \mc{S} } | \mc{D}_2 X |^2 d \omega \text{,} &\qquad X &\in \mc{C}^\infty H_2 \mc{S} \text{,} \\
\notag \int_{ \mc{S} } | \nabla X |^2 d \omega &= \int_{ \mc{S} } | \mc{D}_1^\ast X |^2 d \omega \text{,} &\qquad X &\in \mc{C}^\infty H_0 \mc{S} \text{,} \\
\notag \int_{ \mc{S} } ( | \nabla X |^2 - \mc{K} | X |^2 ) d \omega &= 2 \int_{ \mc{S} } | \mc{D}_2^\ast X |^2 d \omega \text{,} &\qquad X &\in \mc{C}^\infty H_1 \mc{S} \text{.}
\end{alignat}

\subsection{Geometric Littlewood-Paley Theory} \label{sec.geom_glp}

Next, we construct a fully geometric and tensorial Littewood-Paley (abbreviated L-P) theory, using smoothed spectral decompositions of the (Bochner) Laplacian.
We adopt the same ideas as in \cite{bq_ger_tzv:strichartz}, for example, but we also consider tensorial quantities.
\footnote{An alternative approach is to use the geometric L-P theory of \cite{kl_rod:glp}, based on the heat flow.
This was done in previous works in this direction, e.g., \cite{kl_rod:stt, wang:cg}.
In fact, for our purposes, we can attain our desired estimates using either theory, with mostly the same proofs.
Since the spectral version is much easier to rigorously construct and utilize, we opt for this route here.}

For technical purposes, we consider the Hilbert space $L^2 T^r_l \mc{S}$, defined as the completion of $\mc{C}^\infty T^r_l \mc{S}$ with respect to the above $L^2_x$-norm on $\mc{S}$.
Consider the negative Laplacian $-\lapl$, interpreted as a positive self-adjoint unbounded operator on $L^2 T^r_l \mc{S}$, which then has a spectral decomposition
\footnote{This spectral decomposition is in fact discrete; see \cite{gil:index}.}
\[ -\lapl = \int_0^\infty \lambda \cdot d E_\lambda \text{.} \]

The spectral L-P operators can now be constructed in the following fashion:
\begin{itemize}
\item Fix a function $\varsigma \in \mc{C}^\infty \R$, supported in the region $1/2 \leq | \xi | \leq 2$, satisfying
\[ \sum_{k \in \Z} \varsigma ( 2^{-2k} \xi ) = 1 \text{,} \qquad \xi \in \R \setminus \{ 0 \} \text{.} \]

\item For each $k \in \Z$, we define the L-P operators on $L^2 T^r_l \mc{S}$ by
\[ P_k = \varsigma ( - 2^{-2k} \lapl ) \text{,} \qquad P_- = \chi_{ \{ 0 \} } ( -\lapl ) \text{.} \]
In particular, $P_-$ is precisely the $L^2$-projection onto the kernel of $\lapl$.

\item Given any $k \in \Z$, we can define the aggregated operators
\[ P_{< k} = P_- + \sum_{l < k} P_l \text{,} \]
where the summation is in the strong operator topology.
In particular,
\[ P_{< 0} + \sum_{k \geq 0} P_k \text{,} \]
is the identity operator on $L^2 T^r_l \mc{S}$.
\end{itemize}

Note that the $P_k$'s are almost pairwise orthogonal, that is,
\begin{equation} \label{eq.glp_almost_ortho} P_k P_l \equiv 0 \text{,} \qquad k, l \in \Z \text{,} \quad | k - l | > 1 \text{.} \end{equation}
As a result, given $k \in \Z$, we will often use the notations
\[ P_{\sim k} = P_{k - 1} + P_k + P_{k + 1} \text{,} \qquad P_{\lesssim k} = P_{< k + 1} \text{.} \]
The property \eqref{eq.glp_almost_ortho} implies the handy self-replication identities
\[ P_k = P_k P_{\sim k} \text{,} \qquad P_{< k} = P_{< k} P_{\lesssim k} \text{,} \qquad P_{\geq k} = P_{\geq k} P_{\gtrsim k} \text{.} \]

Another important trick is the identity
\[ 2^{2k} P_k = - \lapl P^\prime_k \text{,} \qquad k \in \Z \text{,} \]
where $P^\prime_k = \tilde{\varsigma} ( - 2^{-2k} \lapl )$ is another smoothed spectral projection.
$P^\prime_k$ has the same support as $P_k$, and hence has mostly the same estimates as $P_k$.
Note that a similar trick applies to other related operators, such as powers of $-\lapl$ or $I - \lapl$.

We also mention that since $\lapl$ commutes with contractions, metric contractions, and volume form contractions, then all the geometric L-P operators (e.g., the $P_k$'s and $P_{< 0}$) must also commute with these contractions.

\subsection{Sobolev and Besov Norms} \label{sec.geom_norms}

We now list some basic $L^2_x$-estimates satisfied by the geometric L-P operators $P_k$ and $P_{< 0}$.

\begin{proposition} \label{thm.glp}
Let $k \in \Z$, and let $F \in \mc{C}^\infty T^r_l \mc{S}$.
\begin{itemize}
\item $P_k$ and $P_{< k}$ are bounded operators, i.e.,
\begin{equation} \label{eq.glp_bdd} \| P_k F \|_{ L^2_x } + \| P_{< k} F \|_{ L^2_x } \lesssim \| F \|_{ L^2_x } \text{.} \end{equation}

\item The following ``finite band" estimates hold:
\begin{align}
\label{eq.glp_fbl} \| \lapl P_k F \|_{ L^2_x } \lesssim 2^{2k} \| F \|_{ L^2_x } \text{,} \qquad \| \lapl P_{< k} F \|_{ L^2_x } \lesssim 2^{2k} \| F \|_{ L^2_x } \text{.}
\end{align}

\item The following ``finite band" estimates hold:
\footnote{In the expressions $P_k \nabla F$ and $P_{< k} \nabla F$, the L-P operators are of course those on $L^2 T^r_{l+1} \mc{S}$.}
\begin{alignat}{3}
\label{eq.glp_fb} \| \nabla P_k F \|_{ L^2_x } &\lesssim 2^k \| F \|_{ L^2_x } \text{,} &\qquad \| P_k \nabla F \|_{ L^2_x } &\lesssim 2^k \| F \|_{ L^2_x } \text{,} \\
\notag \| \nabla P_{< k} F \|_{ L^2_x } &\lesssim 2^k \| F \|_{ L^2_x } \text{,} &\qquad \| P_{< k} \nabla F \|_{ L^2_x } &\lesssim 2^k \| F \|_{ L^2_x } \text{.}
\end{alignat}

\item The following ``reverse finite band" estimates hold:
\begin{equation} \label{eq.glp_fbr} \| P_k F \|_{ L^2_x } \lesssim 2^{-2k} \| \lapl F \|_{ L^2_x } \text{,} \qquad \| P_k F \|_{ L^2_x } \lesssim 2^{-k} \| \nabla F \|_{ L^2_x } \text{.} \end{equation}
\end{itemize}
\end{proposition}

\begin{proof}
First, \eqref{eq.glp_bdd} and \eqref{eq.glp_fbl} are direct consequences of the definitions of the $P_k$'s and $P_-$.
The first estimate in \eqref{eq.glp_fb} follows from \eqref{eq.glp_fbl} via an integration by parts:
\[ \| \nabla P_k F \|_{ L^2_x }^2 \leq \| ( -\lapl ) P_k F \|_{ L^2_x } \| P_k F \|_{ L^2_x } \lesssim 2^{2 k} \| F \|_{ L^2_x }^2 \text{.} \]
That $P_k \nabla$ is similarly bounded follows from the previous estimate by a standard duality argument.
The remaining estimates in \eqref{eq.glp_fb} can be obtained by summing the already completed estimates.
The first part of \eqref{eq.glp_fbr} follows immediately from the definition of $P_k$.
For the second part of \eqref{eq.glp_fbr}, we can apply \eqref{eq.glp_fb}:
\[ \| P_k F \|_{ L^2_x } = 2^{-2 k} \| P^\prime_k ( -\lapl ) F \|_{ L^2_x } \lesssim 2^k \| \nabla F \|_{ L^2_x } \text{.} \qedhere \]
\end{proof}

In general, one can define for any $F \in \mc{C}^\infty T^r_l \mc{S}$ the $L^2$-based Sobolev norms
\[ \| F \|_{ H^s_x } = \| ( I - \lapl )^\frac{s}{2} F \|_{ L^2_x } \text{,} \qquad s \in \R \text{.} \]
Fractional powers of $I - \lapl$ can be defined in the usual manner using functional analytic means.
In particular, in the case $s = 1$, we have
\[ \| F \|_{ H^1_x } \simeq \| \nabla F \|_{ L^2_x } + \| F \|_{ L^2_x } \text{.} \]

Using our geometric L-P theory, we can define a more general class of \emph{geometric tensorial} Besov-type norms.
Indeed, for any $a \in [1, \infty)$ and $s \in \R$, we can define
\begin{align*}
\| F \|_{ B^{a, s}_{\ell, x} }^a &= \sum_{k \geq 0} 2^{ask} \| P_k F \|_{ L^2_x }^a + \| P_{< 0} F \|_{ L^2_x }^a \text{,} \\
\| F \|_{ B^{\infty, s}_{\ell, x} } &= \max \left( \sup_{k \geq 0} 2^{s k} \| P_k F \|_{ L^2_x }, \| P_{< 0} F \|_{ L^2_x } \right) \text{.}
\end{align*}
These are the direct analogues of the standard $B^s_{2, a}$-norms in Euclidean space.

\begin{remark}
We adopt the somewhat peculiar notational conventions above in order to maintain consistency with conventions developed in later sections.
\end{remark}

The following properties are immediate from the definitions of the $B^{a, s}_{\ell, x}$-norms.

\begin{proposition} \label{thm.besov}
Let $a, a^\prime \in [1, \infty]$, let $s, s^\prime \in \R$, and let $F \in \mc{C}^\infty T^r_l \mc{N}$.
\begin{itemize}
\item If $a^\prime \leq a$ and $s \leq s^\prime$, then
\begin{equation} \label{eq.besov_triv_1} \| F \|_{ B^{a, s}_{\ell, x} } \leq \| F \|_{ B^{a^\prime, s^\prime}_{\ell, x} } \text{.} \end{equation}

\item If $a \leq a^\prime$ and $s < s^\prime$, then
\begin{equation} \label{eq.besov_triv_2} \| F \|_{ B^{a, s}_{\ell, x} } \lesssim_{ a^\prime - a, s^\prime - s } \| F \|_{ B^{a^\prime, s^\prime}_{\ell, x} } \text{.} \end{equation}
\end{itemize}
\end{proposition}

The cases of interest in this paper and in the ensuing applications are $a = 2$ and $a = 1$.
As expected, the former case coincides with the above $H^s_x$-norms.

\begin{proposition} \label{thm.sobolev}
If $s \in \R$ and $F \in \mc{C}^\infty T^r_l \mc{N}$, then
\begin{equation} \label{eq.sobolev} \| F \|_{ B^{2, s}_{\ell, x} } \simeq_s \| F \|_{ H^s_x } \text{.} \end{equation}
\end{proposition}

\begin{proof}
This follows from the spectral properties of the operators $P_k$, $P_{< 0}$, and $( I - \lapl )^{s/2}$, along with the almost orthogonality property \eqref{eq.glp_almost_ortho}.
\end{proof}

\subsection{Regularity Conditions} \label{sec.geom_reg}

On a fixed Riemannian surface $(\mc{S}, h)$, one has many standard estimates, e.g., Sobolev inequalities.
In this section, we will review many of these results.
Our main focus here, however, is not the estimates themselves, but rather the control of the \emph{constants} of these estimates.
In later sections, we will apply such estimates uniformly to a family of manifolds.

In order to control these constants, we devise appropriate \emph{regularity assumptions}.
Here, we state the assumptions for $(\mc{S}, h)$ that we will encounter.
The eventual objective is to show that the constants of various estimates on $(\mc{S}, h)$ can indeed be controlled by the parameters of these regularity assumptions.

First, we say that $(\mc{S}, h)$ satisfies \ass{r0}{C, N}, with data $\{ U_i, \varphi_i, \eta_i \}_{i = 1}^N$, iff:
\begin{itemize}
\item The area $| \mc{S} |$ of $(\mc{S}, h)$ satisfies
\[ C^{-1} \leq | \mc{S} | \leq C \text{.} \]

\item The $(U_i, \varphi_i)$'s, where $1 \leq i \leq N$, are local coordinate systems on $\mc{S}$ that cover $\mc{S}$.
Moreover, each $\varphi_i (U_i)$ is a bounded neighborhood in $\R^2$.

\item The $\eta_i$'s form a partition of unity of $\mc{S}$, subordinate to the $U_i$'s, such that
\[ 0 \leq \eta_i \leq 1 \text{,} \qquad | \partial^i_a \eta_i | \leq C \text{,} \qquad a, b \in \{ 1, 2 \} \text{,} \]
for each $1 \leq i \leq N$, where $\partial^i_1, \partial^i_2$ denote the $\varphi_i$-coordinate vector fields.

\item For each $1 \leq i \leq N$, we have on $U_i$ the uniform positivity property
\[ C^{-1} | \xi |^2 \leq \sum_{a, b = 1}^2 h_{ab} \xi^a \xi^b \leq C | \xi |^2 \text{,} \qquad \xi \in \R^2 \text{,} \]
where we have indexed with respect to the $\varphi_i$-coordinate system on $U_i$.
\end{itemize}

On such a coordinate system $(U_i, \varphi_i)$, as expressed in the data for the \ass{r0}{} condition, we can define the associated \emph{area density} $\vartheta_i \in \mc{C}^\infty U_i$ by
\[ \vartheta_i = \sqrt{ h_{11} h_{22} - h_{12}^2 } \text{,} \]
where we have indexed with respect to the $\varphi_i$-coordinates.
The \ass{r0}{} condition, in particular the uniform positivity property, implies uniform bounds for the $\vartheta_i$'s.

\begin{proposition} \label{thmr.R_extra}
Assume $(\mc{S}, h)$ satisfies \ass{r0}{C, N}, with data $\{ U_i, \varphi_i, \eta_i \}_{i = 1}^N$.
\begin{itemize}
\item For any $1 \leq i \leq N$, the area density $\vartheta_i$ satisfies
\begin{equation} \label{eqr.vd_unif} \vartheta_i \simeq_C 1 \text{.} \end{equation}

\item If $q \in [1, \infty]$, $1 \leq i \leq N$, and $f \in \mc{C}^\infty U_i$, then
\footnote{Here, the right-hand side refers to the $L^q$-norm for functions on $\R^2$.}
\begin{equation} \label{eqr.change_of_coord} \| f \|_{ L^q_x } \simeq_{C, q} \| f \circ \varphi_i^{-1} \|_{ L^q_x } \text{.} \end{equation}
\end{itemize}
\end{proposition}

\begin{proof}
The uniform positivity property in the \ass{r0}{C, N} condition immediately implies \eqref{eqr.vd_unif}.
The integral comparison \eqref{eqr.change_of_coord} follows immediately from \eqref{eqr.vd_unif}.
\end{proof}

Next, we say $(\mc{S}, h)$ satisfies \ass{r1}{C, N}, with data $\{ U_i, \varphi_i, \eta_i, \tilde{\eta}_i, e^i \}_{i = 1}^N$, iff:
\begin{itemize}
\item $(\mc{S}, h)$ satisfies \ass{r0}{C, N}, with data $\{ U_i, \varphi_i, \eta_i \}_{i = 1}^N$.

\item For any $1 \leq i \leq N$, we have that $e^i = ( e^i_1, e^i_2 ) \in \mc{C}^\infty T^1_0 U_i \times \mc{C}^\infty T^1_0 U_i$ forms an orthonormal frame on $U_i$ and satisfies the estimates
\[ \| \nabla e^i_a \|_{ L^4_x } \leq C \text{,} \qquad a \in \{ 1, 2 \} \text{.} \]

\item For any $1 \leq i \leq N$, we have that $\tilde{\eta}_i \in \mc{C}^\infty \mc{S}$ is supported within $U_i$, is identically $1$ on the support of $\eta_i$, and satisfies the estimates
\[ 0 \leq \tilde{\eta}_i \leq 1 \text{,} \qquad | \partial^i_a \tilde{\eta}_i | \leq C \text{,} \qquad a \in \{ 1, 2 \} \text{.} \]

\item For each $1 \leq i \leq N$, the $\varphi_i$-area density $\vartheta_i \in \mc{C}^\infty U_i$ satisfies
\[ \| \nabla \vartheta_i \|_{ L^2_x } \leq C \text{.} \]
\end{itemize}
The \ass{r1}{} condition will be essential for our upcoming analysis, in particular with regards to all the geometric Besov-type estimates.

\begin{remark}
In fact, all the analysis in the paper will still remain valid if the $L^4_x$-norm for the orthonormal frame elements are replaced by an $L^q_x$-norm, for any $q \in (2, \infty]$.
However, we will not need this flexibility in any of our intended applications, hence we consider only the case $q = 4$ here for simplicity.
\end{remark}

For future reference, we note the following:

\begin{proposition} \label{thmr.vdd_inv}
Assume $(\mc{S}, h)$ satisfies \ass{r1}{C, N}, with data $\{ U_i, \varphi_i, \eta_i, \tilde{\eta}_i, e^i \}_{i = 1}^N$.
Then, for any $1 \leq i \leq N$, the following estimate holds:
\begin{align}
\label{eqr.vdd_inv} \| \nabla \vartheta_i^{-1} \|_{ L^2_x } \lesssim_C 1 \text{.}
\end{align}
\end{proposition}

\begin{proof}
This follows immediately from \ass{r1}{C, N} and \eqref{eqr.vd_unif}.
\end{proof}

Finally, we say $(\mc{S}, h)$ satisfies \ass{r2}{C, N}, with data $\{ U_i, \varphi_i, \eta_i, \tilde{\eta}_i, e^i \}_{i = 1}^N$, iff:
\begin{itemize}
\item $(\mc{S}, h)$ satisfies \ass{r1}{C, N}, with data $\{ U_i, \varphi_i, \eta_i, \tilde{\eta}_i, e^i \}_{i = 1}^N$.

\item For each $1 \leq i \leq N$, the $\varphi_i$-coordinate vector fields $\partial^i_1, \partial^i_2$ satisfy
\[ \| \nabla \partial^i_a \|_{ L^2_x } \leq C \text{,} \qquad a \in \{ 1, 2 \} \text{.} \]

\item For each $1 \leq i \leq N$, the second \emph{coordinate} derivatives of $\eta_i$ satisfy
\[ | \partial^i_a \partial^i_b \eta_i | \leq C \text{,} \qquad a, b \in \{ 1, 2 \} \text{.} \]
\end{itemize}
The \ass{r2}{} condition will only be used in the proof of Theorem \ref{thm.sharp_trace}.

\subsection{Sobolev Inequalities} \label{sec.geom_sob}

One application of the \ass{r0}{} condition is deriving first-order scalar Sobolev estimates on $(\mc{S}, h)$.
This can be done by applying the corresponding Euclidean estimates on each coordinate system $(U_i, \varphi_i)$ in the data.
For example, using this process, we can derive the following standard estimates:

\begin{proposition} \label{thm.gns_pre}
Suppose $(\mc{S}, h)$ satisfies \ass{r0}{C, N} holds, and let $f \in \mc{C}^\infty \mc{S}$.
\begin{itemize}
\item The following Gagliardo-Nirenberg-Sobolev inequality holds:
\begin{equation} \label{eq.gns_1_pre} \| f \|_{ L^2_x } \lesssim_{C, N} \| \nabla f \|_{ L^1_x } + \| f \|_{ L^1_x } \text{.} \end{equation}

\item The following Gagliardo-Nirenberg inequality holds for any $q \in (2, \infty)$:
\begin{equation} \label{eq.gns_1s_pre} \| f \|_{ L^\infty_x } \lesssim_{C, N, q} \| \nabla f \|_{ L^q_x }^\frac{2}{q} \| f \|_{ L^q_x }^{ 1 - \frac{2}{q} } + \| f \|_{ L^q_x } \text{.} \end{equation}
\end{itemize}
\end{proposition}

\begin{proof}
See, e.g., \cite{kl_rod:glp, shao:bdc_nv}.
\end{proof}

Next, we apply Proposition \ref{thm.gns_pre} in order to prove similar Sobolev estimates for \emph{tensorial} quantities.
This is trickier, since the connection $\nabla$ now depends on $h$.
However, we can sidestep this issue by dealing with the square norm of the tensor field, which is scalar, and recalling that $h$ and $\nabla$ are compatible.

\begin{proposition} \label{thm.gns_ineq}
Suppose $(\mc{S}, h)$ satisfies \ass{r0}{C, N}, and let $F \in \mc{C}^\infty T^r_l \mc{S}$.
\begin{itemize}
\item If $q \in (2, \infty)$, then
\begin{align}
\label{eq.gns_1} \| F \|_{ L^q_x } &\lesssim_{C, N, q} \| \nabla F \|_{ L^2_x }^{ 1 - \frac{2}{q} } \| F \|_{ L^2_x }^\frac{2}{q} + \| F \|_{ L^2_x } \text{,} \\
\label{eq.gns_1s} \| F \|_{ L^\infty_x } &\lesssim_{C, N, q} \| \nabla F \|_{ L^q_x }^\frac{2}{q} \| F \|_{ L^q_x }^{ 1 - \frac{2}{q} } + \| F \|_{ L^q_x } \text{.}
\end{align}

\item Moreover, the following estimate holds:
\begin{equation} \label{eq.gns_2} \| F \|_{ L^\infty_x } \lesssim_{C, N} \| \nabla^2 F \|_{ L^2_x }^\frac{1}{2} \| F \|_{ L^2_x }^\frac{1}{2} + \| F \|_{ L^2_x } \text{.} \end{equation}
\end{itemize}
\end{proposition}

\begin{proof}
\footnote{The proof of \eqref{eq.gns_1} is identical to that of \cite[Cor. 2.4]{kl_rod:glp}; we reproduce it here for convenience.}
First, we make use of \eqref{eq.gns_1_pre} as follows:
\begin{align*}
\| F \|_{ L^q_x }^\frac{q}{2} &\lesssim \| \nabla | F |^\frac{q}{2} \|_{ L^1_x } + \| | F |^\frac{q}{2} \|_{ L^1_x } \lesssim ( \| \nabla F \|_{ L^2_x } + \| F \|_{ L^2_x } ) \| F \|_{ L^{q - 2}_x }^\frac{q - 2}{2} \text{.}
\end{align*}
Applying the above with $q = 2k$, with $k > 1$ an integer, yields
\[ \| F \|_{ L^{2k}_x }^k \lesssim ( \| \nabla F \|_{ L^2_x } + \| F \|_{ L^2_x } ) \| F \|_{ L^{2k - 2}_x }^{k - 1} \text{.} \]
Taking $k = 2, 3, 4, \dots$ and applying an induction argument, we obtain
\[ \| F \|_{ L^{2k}_x } \lesssim ( \| \nabla F \|_{ L^2_x } + \| F \|_{ L^2_x } )^\frac{k - 1}{k} \| F \|_{ L^2_x }^\frac{1}{k} \text{.} \]
This is \eqref{eq.gns_1} when $q$ is an even integer; the general case follows via interpolation.

The idea for proving \eqref{eq.gns_1s} is similar.
We apply \eqref{eq.gns_1s_pre} as follows:
\begin{align*}
\| F \|_{ L^\infty_x }^2 &\lesssim \| \nabla | F |^2 \|_{ L^q_x }^\frac{2}{q} \| | F |^2 \|_{ L^q_x }^{ 1 - \frac{2}{q} } + \| | F |^2 \|_{ L^q_x } \\
&\lesssim \| F \|_{ L^\infty_x } ( \| \nabla F \|_{ L^q_x }^\frac{2}{q} \| F \|_{ L^q_x }^{ 1 - \frac{2}{q} } + \| F \|_{ L^q_x } ) \text{.}
\end{align*}
The inequality \eqref{eq.gns_1s} follows immediately.
Finally, for \eqref{eq.gns_2}, we apply \eqref{eq.gns_1s} and \eqref{eq.gns_1} in succession and integrate by parts to eliminate the $L^2$-norms of $\nabla F$.
\end{proof}

We can combine Proposition \ref{thm.gns_ineq} with our geometric L-P theory in order to derive weak Bernstein estimates for our L-P operators.

\begin{proposition} \label{thm.glp_bernstein}
Suppose $(\mc{S}, h)$ satisfies \ass{r0}{C, N}.
If $F \in \mc{C}^\infty T^r_l \mc{S}$, and if $k \geq 0$, $q \in (2, \infty)$, and $q^\prime \in (1, 2)$, then the following estimates hold:
\begin{alignat}{3}
\label{eq.glp_bernstein} \| P_k F \|_{ L^q_x } &\lesssim_{C, N, q} 2^{ (1 - \frac{2}{q}) k } \| F \|_{ L^2_x } \text{,} &\qquad \| P_{< 0} F \|_{ L^q_x } &\lesssim_{C, N, q} \| F \|_{ L^2_x } \text{,} \\
\notag \| P_k F \|_{ L^2_x } &\lesssim_{C, N, q^\prime} 2^{ (\frac{2}{q^\prime} - 1) k } \| F \|_{ L^{q^\prime}_x } \text{,} &\qquad \| P_{< 0} F \|_{ L^2_x } &\lesssim_{C, N, q^\prime} \| F \|_{ L^{q^\prime}_x } \text{.}
\end{alignat}
\end{proposition}

\begin{proof}
The first two inequalities are proved using \eqref{eq.gns_1} in conjunction with \eqref{eq.glp_fb}.
The remaining two inequalities follow by duality.
\end{proof}

\subsection{Conformal Transformations} \label{sec.geom_conf}

We conclude this section by reviewing some basic properties of conformal transformations of $h$.
Let $\bar{h} \in \mc{C}^\infty T^0_2 \mc{S}$ be another Riemannian metric that is a conformal transform of $h$,
\[ \bar{h} = e^{2 u} h \text{,} \qquad u \in \mc{C}^\infty \mc{S} \text{.} \]

We will use the following notational conventions: geometric objects and norms defined with respect to $\bar{h}$ will be denoted with a ``bar" over the symbol.
For example, $\bar{\omega} = e^{2 u} \omega$ denotes the volume form associated with $\bar{h}$ (with respect to the same orientation).
In index notation, we have the identities
\[ \bar{h}_{ab} = e^{2 u} h_{ab} \text{,} \qquad \bar{\omega}_{ab} = e^{2 u} \omega_{ab} \text{,} \qquad \bar{h}^{ab} = e^{-2 u} h^{ab} \text{,} \qquad \bar{\omega}^{ab} = e^{-2 u} \omega^{ab} \text{.} \]

The relations between $\nabla$ and $\bar{\nabla}$ and between $\mc{K}$ and $\bar{\mc{K}}$ are standard formulas in differential geometry.
First, for any $F \in \mc{C}^\infty T^r_l \mc{S}$,
\begin{align}
\label{eq.conf_cd} \bar{\nabla}_a F_{c_1 \dots c_l}^{d_1 \dots d_r} &= \nabla_a F_{c_1 \dots c_l}^{d_1 \dots d_r} + (r - l) \cdot \nabla_a u \cdot F_{c_1 \dots c_l}^{d_1 \dots d_r} \\
\notag &\qquad - \sum_{i = 1}^l ( \nabla_{c_i} u \cdot F_{c_1 \hat{a}_i c_l}^{d_1 \dots d_r} - h_{a c_i} h^{b c} \nabla_c u \cdot F_{c_1 \hat{b}_i c_l}^{d_1 \dots d_r} ) \\
\notag &\qquad + \sum_{j = 1}^l ( h_a{}^{d_j} \nabla_b u \cdot F_{c_1 \dots c_l}^{d_1 \hat{b}_j d_r} - h^{d_j b} \gamma_{a e} \nabla_b u \cdot F_{c_1 \dots c_l}^{d_1 \hat{e}_j d_r} ) \text{.}
\end{align}
In the above, the notation $c_1 \hat{a}_i c_l$ means $c_1 \dots c_l$, but with $c_i$ replaced by $a$.
By a similar brute force computation, one can also derive
\begin{equation} \label{eq.conf_gc} \bar{\mc{K}} = e^{-2u} ( \mc{K} - \lapl u ) \text{.} \end{equation}

\begin{remark}
Note that $\nabla$ and $\bar{\nabla}$ coincide for scalar fields, such as $u$.
\end{remark}

Next, we recall that the symmetric Hodge operators $\mc{D}_1$, $\mc{D}_2$, and $\mc{D}_1^\ast$ are conformally invariant.
Indeed, using \eqref{eq.conf_cd}, we obtain the relations
\footnote{Since the bundles $H_2 \mc{S}$ and $\bar{H}_2 \mc{S}$ coincide, the above relation for $\mc{D}_2$ and $\bar{\mc{D}}_2$ makes sense.}
\begin{equation} \label{eq.conf_hodge} \bar{\mc{D}}_1 = e^{-2u} \mc{D}_1 \text{,} \qquad \bar{\mc{D}}_2 = e^{-2u} \mc{D}_2 \text{,} \qquad \bar{\mc{D}}_1^\ast = \mc{D}_1^\ast \text{.} \end{equation}
On the other hand, $\mc{D}_2^\ast$ fails to be conformally invariant.

Finally, we show that a sufficiently nice conformal factor will preserve the regularity conditions that were defined in Section \ref{sec.geom_reg}.

\begin{proposition} \label{thm.conf_regf}
Let $u \in \mc{C}^\infty \mc{S}$, and let $\bar{h} = e^{2 u} h$ be a conformal transform of $h$.
Assume $(\mc{S}, h)$ satisfies \ass{r0}{C, N}, with data $\{ U_i, \varphi_i, \eta_i \}_{i = 1}^N$, and assume
\[ D = \| u \|_{ L^\infty_x } + \| \nabla u \|_{ L^4_x } \ll 1 \]
is sufficiently small.
Then, the following properties hold:
\begin{itemize}
\item For any $m \in \Z$,
\begin{equation} \label{eq.conf_unif} \| e^{m u} \|_{ L^\infty_x } \lesssim_m 1 \text{.} \end{equation}

\item For any $q \in [1, \infty]$ and $F \in \mc{C}^\infty T^r_l \mc{S}$,
\begin{equation} \label{eq.conf_comp} \| F \|_{ \bar{L}^q_x } \simeq_{q, r, l} \| F \|_{ L^q_x } \text{,} \qquad \| F \|_{ \bar{H}^1_x } \simeq_{r, l} \| F \|_{ H^1_x } \text{.} \end{equation}

\item There exists a constant $C^\prime$, depending on $C$ and $N$, such that $(\mc{S}, \bar{h})$ satisfies \ass{r0}{C^\prime, N}, with the same data $\{ U_i, \varphi_i, \eta_i \}_{i = 1}^N$.

\item Suppose $(\mc{S}, h)$ also satisfies \ass{r1}{C, N}, with data $\{ U_i, \varphi_i, \eta_i, \tilde{\eta}_i, e^i \}_{i = 1}^N$, with $e^i = (e^i_1, e^i_2)$ for each $1 \leq i \leq N$.
If we define $\bar{e}^i = ( e^{-u} e^i_1, e^{-u} e^i_2 )$, then there is some constant $C^\prime$, depending on $C$ and $N$, such that $(\mc{S}, \bar{h})$ satisfies \ass{r1}{C^\prime, N}, with data $\{ U_i, \varphi_i, \eta_i, \tilde{\eta}_i, \bar{e}^i \}_{i = 1}^N$.
\end{itemize}
\end{proposition}

\begin{proof}
First of all, since $D$ is small, \eqref{eq.conf_unif} follows trivially from the assumptions, and the first estimate in \eqref{eq.conf_comp} is an immediate consequence of \eqref{eq.conf_unif}.

Since $\bar{h} = e^{2 u} h$, then \eqref{eq.conf_unif} implies the area $| \bar{\mc{S}} |$ of $( \mc{S}, \bar{h} )$ satisfies $| \bar{\mc{S}} | \simeq | \mc{S} | \simeq 1$.
Similarly, on each $U_i$, indexing with respect to the $\varphi_i$-coordinates, we have
\[ \sum_{a, b = 1}^2 \bar{h}_{ab} \xi^a \xi^b \simeq \sum_{a, b = 1}^2 h_{ab} \xi^a \xi^b \simeq | \xi |^2 \text{,} \qquad \xi \in \R^2 \text{.} \]
It follows that $(\mc{S}, \bar{h})$ satisfies \ass{r0}{C^\prime, N} for some appropriately chosen $C^\prime$.

Recalling \eqref{eq.conf_cd} and applying \eqref{eq.gns_1} and the first part of \eqref{eq.conf_comp}, we can estimate
\[ \| \bar{\nabla} F \|_{ \bar{L}^2_x } \lesssim \| \nabla F \|_{ L^2_x } + \| \nabla u \|_{ L^4_x } \| F \|_{ L^4_x } \lesssim \| \nabla F \|_{ L^2_x } + \| F \|_{ L^2_x } \text{.} \]
Since $(\mc{S}, \bar{h})$ satisfies \ass{r0}{C^\prime, N}, as shown above, we also have
\[ \| \nabla F \|_{ L^2_x } \lesssim \| \bar{\nabla} F \|_{ \bar{L}^2_x } + \| F \|_{ \bar{L}^2_x } \text{.} \]
This completes the proof of \eqref{eq.conf_comp}.

For the last property, assume $(\mc{S}, h)$ satisfies \ass{r1}{C, N}, with the aforementioned data.
By definition, $\bar{e}^i_1 = e^{-u} e^i_1$ and $\bar{e}^i_2 = e^{-u} e^i_2$ form an $\bar{h}$-orthonormal frame on $U_i$, for each $1 \leq i \leq N$.
Moreover, using \eqref{eq.conf_cd} and \eqref{eq.conf_comp}, we can estimate
\begin{align*}
\| \bar{\nabla} \bar{e}_a \|_{ \bar{L}^4_x } \lesssim \| \nabla u \|_{ L^4_x } + \| \nabla e_a \|_{ L^4_x } \lesssim 1 \text{,} \qquad a \in \{ 1, 2 \} \text{.}
\end{align*}
Now, if $\vartheta_i \in \mc{C}^\infty U_i$ is the $h$-volume density on $U_i$ with respect to the $\varphi_i$-coordinates (see Section \ref{sec.geom_reg}), then the corresponding quantity for $\bar{h}$ is $\bar{\vartheta}_i = e^{2 u} \vartheta_i$, and
\[ \| \bar{\nabla} \bar{\vartheta}_i \|_{ \bar{L}^2_x } \lesssim \| \nabla u \|_{ L^2_x } + \| \nabla \vartheta_i \|_{ L^2_x } \lesssim 1 \text{.} \]
Thus, $(\mc{S}, \bar{h})$ satisfies \ass{r1}{C^\prime, N} for some $C^\prime$, with data $\{ U_i, \varphi_i, \eta_i, \tilde{\eta}_i, \bar{e}^i \}_{i = 1}^N$.
\end{proof}

\begin{remark}
In the setting of Proposition \ref{thm.conf_regf}, if $(\mc{S}, h)$ satisfies \ass{r2}{C, N}, then one can also show that $(\mc{S}, \bar{h})$ satisfies \ass{r2}{C^\prime, N}, with appropriately related data, for some $C^\prime$ depending on $C$ and $N$.
However, we will not need this result in this paper.
\end{remark}

\section{Foliations} \label{sec.fol}

In Section \ref{sec.geom}, we discussed the analysis of tensor fields on a two-dimensional Riemannian manifold.
In this section, we consider our second general setting: that of a one-parameter family of two-dimensional Riemannian manifolds.

Consider the $1$-parameter foliation
\[ \mc{N} = [0, \delta] \times \mc{S} \text{,} \qquad \delta > 0 \text{.} \]
We define $t$ to be the natural projection onto the first component:
\[ t: \mc{N} \rightarrow [0, \delta] \text{,} \qquad t(\tau, x) = \tau \text{.} \]
Throughout, we will let $\tau$ denote an arbitrary element on the interval $[0, \delta]$.
Given such a $\tau$, we let $\mc{S}_\tau$ denote the associated level set of $t$:
\[ \mc{S}_\tau = t^{-1} (\tau) = \{ \tau \} \times \mc{S} \text{.} \]

Although we will work only on the $3$-manifold with boundary $\mc{N}$, we will always implicitly assume that all our objects can be smoothly extended beyond the boundaries $\mc{S}_0$ and $\mc{S}_\delta$.
In other words, our full setting, on which all our objects of analysis are defined, is the extended foliation $\mc{N}^\prime = (-\varepsilon, \delta + \varepsilon) \times \mc{S}$, for some $\varepsilon > 0$.

\subsection{Horizontal Fields} \label{sec.fol_hor}

Consider the trivial diffeomorphisms
\[ \Xi_\tau: \mc{S}_\tau \leftrightarrow \mc{S} \text{,} \qquad \Xi_\tau (\tau, x) = x \text{,} \]
identifying $\mc{S}_\tau$ with $\mc{S}$.
From $\Xi_\tau$, we can construct natural identifications
\[ \Xi_\tau^\ast: \mc{C}^\infty T^r_l \mc{S}_\tau \leftrightarrow \mc{C}^\infty T^r_l \mc{S} \text{.} \]
We will use these identifications repeatedly in our basic constructions.

The first task is to define objects that represent, roughly, a family of corresponding objects on the $\mc{S}_\tau$'s varying smoothly with respect to $\tau$.
For this purpose, we can naturally construct a vector bundle $\ul{T}^r_l \mc{N}$ over $\mc{N}$, with fibers
\[ ( \ul{T}^r_l \mc{N} )_z = ( T^r_l \mc{S}_\tau )_z \text{,} \qquad z = (\tau, x) \in \mc{N} \text{.} \]
We call $\ul{T}^r_l \mc{N}$ the \emph{horizontal tensor bundle} of rank $(r, l)$.
In particular, $\ul{T}^0_0 \mc{N}$ can be naturally identified with the space $\mc{C}^\infty \mc{N}$ of smooth functions on $\mc{N}$.

An element $A \in \mc{C}^\infty \ul{T}^r_l \mc{N}$ is called a \emph{horizontal tensor field}.
Given $\tau$, we let
\[ A [\tau] = \Xi_\tau^\ast ( A | \mc{S}_\tau ) \in \mc{C}^\infty T^r_l \mc{S} \text{,} \]
i.e., the tensor field on $\mc{S}$ that is the restriction of $A$ to $\mc{S}_\tau$.
Alternately, we can think of $A$ as a family $A [\tau]$ of fields on $\mc{S}$ that varies smoothly with respect to $\tau$.

The next structure to impose on $\mc{N}$ is a \emph{horizontal metric} $\gamma \in \mc{C}^\infty \ul{T}^0_2 \mc{N}$.
More specifically, we stipulate that $\gamma [\tau]$ is a Riemannian metric on $\mc{S}$ for each $\tau$.
We also let $\gamma^{-1} \in \mc{C}^\infty \ul{T}^2_0 \mc{N}$ denote the dual to $\gamma$, that is, $\gamma^{-1} [\tau]$ is the dual $( \gamma [\tau] )^{-1}$ to $\gamma [\tau]$ for each $\tau$.
The above induces a \emph{horizontal volume form} $\epsilon \in \mc{C}^\infty \ul{T}^0_2 \mc{N}$, such that $\epsilon [\tau]$ represents the volume form of $\mc{S}$ associated with $\gamma [\tau]$ and the orientation of $\mc{S}$.

In general, families of objects defined on $\mc{S}$ that are parametrized by $\tau$ can be aggregated into corresponding ``horizontal" objects on $\mc{N}$.
We already saw three examples of this in the definitions of $\gamma$, $\gamma^{-1}$, and $\epsilon$.
We now list the remaining common examples we will reference throughout the paper.

\begin{itemize}
\item The pointwise inner products with respect to the $\gamma [\tau]$'s lift to a corresponding inner product on horizontal tensor fields, with respect to $\gamma$:
\[ \langle \cdot, \cdot \rangle : \mc{C}^\infty \ul{T}^r_l \mc{N} \times \mc{C}^\infty \ul{T}^r_l \mc{N} \rightarrow \mc{C}^\infty \mc{N} \text{,} \qquad \langle \Psi, \Phi \rangle [\tau] = \langle \Psi [\tau], \Phi [\tau] \rangle \text{.} \]
The horizontal tensor norm $| \cdot |$ is similarly defined from the $\gamma [\tau]$-norms.

\item Given $\Psi_i \in \mc{C}^\infty \ul{T}^{r_i}_{l_i} \mc{N}$, where $i \in \{ 1, 2 \}$, we define the tensor product
\[ \Psi_1 \otimes \Psi_2 \in \mc{C}^\infty \ul{T}^{r_1 + r_2}_{l_1 + l_2} \mc{N} \text{,} \qquad ( \Psi_1 \otimes \Psi_2 ) [\tau] = \Psi_1 [\tau] \otimes \Psi_2 [\tau] \text{.} \]

\item The Levi-Civita connections $\nabla$ on the $( \mc{S}, \gamma [\tau] )$'s can be aggregated into a single \emph{horizontal covariant differential} operator
\[ \nabla: \mc{C}^\infty \ul{T}^r_l \mc{N} \rightarrow \mc{C}^\infty \ul{T}^r_{l+1} \mc{N} \text{.} \]
More specifically, if $A \in \mc{C}^\infty \ul{T}^r_l \mc{N}$ and $X \in \mc{C}^\infty \ul{T}^1_0 \mc{N}$, then
\[ ( \nabla_X A ) [\tau] = \nabla_{ X [\tau] } ( A [\tau] ) \text{.} \]
Higher-order differentials $\nabla^k$, $k > 1$, are defined similarly.

\item We can also define the horizontal (Bochner) Laplacian with respect to $\gamma$:
\[ \lapl : \mc{C}^\infty \ul{T}^r_l \mc{N} \rightarrow \mc{C}^\infty \ul{T}^r_l \mc{N} \text{,} \qquad ( \lapl \Psi ) [\tau] = \lapl ( \Psi [\tau] ) \text{.} \]

\item The horizontal \emph{Gauss curvature} is the map $\mc{K} \in \mc{C}^\infty \mc{N}$ such that each $\mc{K} [\tau]$ is precisely the Gauss curvature associated with $\gamma [\tau]$.

\item We can aggregate the geometric L-P operators $P_k$, $P_{< k}$, $P_-$, so that they act on $\mc{C}^\infty \ul{T}^r_l \mc{N}$, with respect to $\gamma [\tau]$ on each $\mc{S}_\tau$.
\end{itemize}

Finally, the Hodge bundles defined in Section \ref{sec.geom_riem} can be lifted to horizontal objects on $\mc{N}$.
Let $\ul{H}_i \mc{N}$, where $i \in \{ 0, 1, 2 \}$, denote the natural vector bundles over $\mc{N}$, for which the fibers at each $(\tau, x) \in \mc{N}$ is $( \ul{H}_i \mc{N} )_{ (\tau, x) } = ( H_i \mc{S}_\tau )_{ (\tau, x) }$.
Furthermore, for $j \in \{ 1, 2 \}$, we define the aggregated Hodge operators
\[ \mc{D}_j : \mc{C}^\infty \ul{H}_j \mc{N} \rightarrow \mc{C}^\infty \ul{H}_{j-1} \mc{N} \text{,} \qquad \mc{D}_j^\ast : \mc{C}^\infty \ul{H}_{j-1} \mc{N} \rightarrow \mc{C}^\infty \ul{H}_j \mc{N} \text{,} \]
to behave like the corresponding Hodge operators on each $( \mc{S}, \gamma [\tau] )$.

Given $\Psi \in \mc{C}^\infty \ul{T}^r_l \mc{N}$, we will generally assume that any norm of $\Psi [\tau]$ is taken with respect to $\gamma [\tau]$.
In particular, noting this convention, if $p \in [1, \infty)$ and $q \in [1, \infty]$, then we can define the iterated integral norms
\begin{align*}
\| \Psi \|_{ L^{p, q}_{t, x} } = \paren{ \int_0^\delta \| \Psi [\tau] \|_{ L^q_x }^p d \tau }^\frac{1}{p} \text{,} \qquad \| \Psi \|_{ L^{\infty, q}_{t, x} } = \sup_{ 0 \leq \tau \leq \delta } \| \Psi [\tau] \|_{ L^q_x } \text{.}
\end{align*}
Moreover, given $a \in [1, \infty)$, $s \in \R$, and $p \in [1, \infty]$, we can define the following geometric, fully tensorial, and \emph{time-integrated} Besov-type norms:
\begin{align*}
\| \Psi \|_{ B^{a, p, s}_{\ell, t, x} }^a &= \sum_{k \geq 0} 2^{ask} \| P_k \Psi \|_{ L^{p, 2}_{t, x} }^a + \| P_{< 0} \Psi \|_{ L^{p, 2}_{t, x} }^a \text{,} \\
\| \Psi \|_{ B^{\infty, p, s}_{\ell, t, x} } &= \max \left( \sup_{k \geq 0} 2^{sk} \| P_k \Psi \|_{ L^{p, 2}_{t, x} }, \| P_{< 0} \Psi \|_{ L^{p, 2}_{t, x} } \right) \text{.}
\end{align*}

\begin{remark}
Heuristically speaking, for the $B^{a, p, s}_{\ell, t, x}$-norm, the parameters $a$, $p$, $s$ refer to the summability of the L-P components, the integrability of the $t$-component, and the differentiability of the spatial components, respectively.
The order ``$\ell, t, x$" refers to the relative order of integration and summation.
\end{remark}

\subsection{Localization} \label{sec.fol_loc}

If $U \subseteq \mc{S}$ is open, then we can consider the localized foliation
\[ \mc{N}_U = [0, \delta] \times U \text{.} \]
Since $U$ is a submanifold of $\mc{S}$ of the same dimension, we can treat $\mc{N}_U$ using the same formalisms as we did for $\mc{N}$.
All the geometric objects we have defined on $\mc{N}$ have direct analogues on $\mc{N}_U$, obtained via restriction.
Localized foliations will be generally useful for dealing with local coordinate systems and local frames.

Relations involving horizontal tensors are often more easily described using index notations.
We will use the same indexing conventions as one would use for tensors on $\mc{S}$ or the $\mc{S}_\tau$'s.
We will use lowercase Latin indices to denote components of a horizontal tensor field, with repeated indices indicating summations.

To be a bit more specific, suppose $U$ is as before, and let
\[ e_1, e_2 \in \mc{C}^\infty \ul{T}^1_0 \mc{N}_U \text{,} \qquad e^1_\ast, e^2_\ast \in \mc{C}^\infty \ul{T}^0_1 \mc{N}_U \]
denote a local horizontal frame and associated coframe.
We can then index horizontal fields with respect to these frames.
For example, if $\Psi \in \mc{C}^\infty \ul{T}^1_1$, then
\[ \Psi^a_b = \Psi ( e^a_\ast, e_b ) \in \mc{C}^\infty \mc{N}_U \text{.} \]
Note that the chosen frame and coframe is allowed to change as $\tau$ changes.

\begin{remark}
In particular, if we index on $\mc{N}$ as above, and if we index on each $\mc{S}_\tau \simeq \mc{S}$ with respect to the $e_a [\tau]$'s and $e^b_\ast [\tau]$'s, then for any $\Psi \in \mc{C}^\infty \ul{T}^1_1 \mc{N}$, we have
\[ \Psi^a_b [\tau] = ( \Psi [\tau] )^a_b \text{.} \]
This extends directly to horizontal tensor fields of any rank.
\end{remark}

\begin{remark}
To remain consistent with standard index notation conventions, the components of $\gamma^{-1}$ will be denoted $\gamma^{ab}$ rather than $( \gamma^{-1} )^{ab}$.
\end{remark}

$\Psi \in \mc{C}^\infty \ul{T}^r_l \mc{N}_U$ is called \emph{equivariant} iff the restrictions $\Psi [\tau] \in \mc{C}^\infty T^r_l U$ are constant over all $\tau$.
Note that given $F \in \mc{C}^\infty T^r_l U$, there is a unique equivariant field $\mf{e} F \in \mc{C}^\infty \ul{T}^r_l \mc{N}_U$, the \emph{equivariant transport} of $F$, satisfying $\mf{e} F [\tau] = F$ for every $\tau$.

For example, if $(U, \tilde{\varphi}) = (U; \tilde{x}^1, \tilde{x}^2)$ is a local coordinate system on $\mc{S}$, then we can transport each coordinate $\tilde{x}^a$ to $\mc{N}_U$ equivariantly:
\[ x^a = \mf{e} \tilde{x}^a \in \mc{C}^\infty \mc{N}_U \text{,} \qquad a \in \{ 1, 2 \} \text{.} \]
This defines a pair of equivariant functions on $\mc{N}_U$ which form coordinate systems on each timeslice $U_\tau = \{ \tau \} \times U$.
Furthermore, if $\tilde{\partial}_1$ and $\tilde{\partial}_2$ are the coordinate vector fields associated to $\tilde{x}^1$ and $\tilde{x}^2$, then their equivariant transports
\[ \partial_a = \mf{e} \tilde{\partial}_a \in \mc{C}^\infty \ul{T}^1_0 \mc{N}_U \text{,} \qquad a \in \{ 1, 2 \} \]
form the coordinate vector fields for the $x^a$'s on each $U_\tau$.

\subsection{Regularity Conditions} \label{sec.fol_reg}

The next task is to port the regularity conditions in Section \ref{sec.geom_reg} for a single surface to the current setting.
Essentially, we wish to assume these conditions hold \emph{uniformly} on every $(\mc{S}, \gamma [\tau])$.
However, for our analysis, we need a bit more---that the data associated with the conditions on the $\mc{S}_\tau$'s vary smoothly with $\tau$.
The precise conditions on $(\mc{N}, \gamma)$ that we will use are below.

We say that $(\mc{N}, \gamma)$ satisfies \ass{R0}{C, N}, with data $\{ U_i, \varphi_i, \eta_i \}_{i = 1}^N$, iff:
\begin{itemize}
\item For any $\tau \in [0, \delta]$, the area $| \mc{S}_\tau |$ of $(\mc{S}, \gamma [\tau])$ satisfies
\[ C^{-1} \leq | \mc{S}_\tau | \leq C \text{.} \]

\item The $(U_i, \varphi_i)$'s, where $1 \leq i \leq N$, are local coordinate systems on $\mc{S}$ that cover $\mc{S}$.
Moreover, each $\varphi_i (U_i)$ is a bounded neighborhood in $\R^2$.

\item The $\eta_i$'s form a partition of unity of $\mc{S}$, subordinate to the $U_i$'s, such that
\[ 0 \leq \eta_i \leq 1 \text{,} \qquad | \partial^i_a \eta_i | \leq C \text{,} \qquad a, b \in \{ 1, 2 \} \text{,} \]
for each $1 \leq i \leq N$, where $\partial^i_1, \partial^i_2$ denote the $\varphi_i$-coordinate vector fields.

\item For each $1 \leq i \leq N$ and $\tau \in [0, \delta]$, we have on $U_i$ the positivity property
\[ C^{-1} | \xi |^2 \leq \sum_{a, b = 1}^2 ( \gamma [\tau] )_{ab} \xi^a \xi^b \leq C | \xi |^2 \text{,} \qquad \xi \in \R^2 \text{,} \]
where we have indexed with respect to the $\varphi_i$-coordinate system on $U_i$.
\end{itemize}

Note that if $(\mc{N}, \gamma)$ satisfies \ass{R0}{C, N}, then every level surface $(\mc{S}, \gamma [\tau])$ trivially satisfies \ass{r0}{C, N}, with the same data.
Similar to Section \ref{sec.geom_reg}, given data for the \ass{R0}{} condition, we can define the associated \emph{area density} $\vartheta_i \in \mc{C}^\infty \mc{N}_{U_i}$ by
\[ \vartheta_i [\tau] = \sqrt{ ( \gamma [\tau] )_{11} ( \gamma [\tau] )_{22} - ( \gamma [\tau] )_{12}^2 } \text{,} \]
where we have again indexed with respect to the $\varphi_i$-coordinates.

Next, we say $(\mc{N}, \gamma)$ satisfies \ass{R1}{C, N}, with data $\{ U_i, \varphi_i, \eta_i, \tilde{\eta}_i, e^i \}_{i = 1}^N$, iff:
\begin{itemize}
\item $(\mc{N}, \gamma)$ satisfies \ass{R0}{C, N}, with data $\{ U_i, \varphi_i, \eta_i \}_{i = 1}^N$.

\item For any $1 \leq i \leq N$, we have that $e^i = ( e^i_1, e^i_2 ) \in \mc{C}^\infty \ul{T}^1_0 \mc{N}_{U_i} \times \mc{C}^\infty \ul{T}^1_0 \mc{N}_{U_i}$ forms a ($\gamma$-)orthonormal frame on each $( U_i )_\tau$ and satisfies
\[ \| \nabla e^i_a \|_{ L^{\infty, 4}_{t, x} } \leq C \text{,} \qquad a \in \{ 1, 2 \} \text{.} \]

\item For any $1 \leq i \leq N$, we have that $\tilde{\eta}_i \in \mc{C}^\infty \mc{S}$ is supported within $U_i$, is identically $1$ on the support of $\eta_i$, and satisfies the estimates
\[ 0 \leq \tilde{\eta}_i \leq 1 \text{,} \qquad | \partial^i_a \tilde{\eta}_i | \leq C \text{,} \qquad a \in \{ 1, 2 \} \text{.} \]

\item For each $1 \leq i \leq N$, the $\varphi_i$-area density $\vartheta_i \in \mc{C}^\infty \mc{N}_{U_i}$ satisfies
\[ \| \nabla \vartheta_i \|_{ L^{\infty, 2}_{t, x} } \leq C \text{.} \]
\end{itemize}
In particular, if $(\mc{N}, \gamma)$ satisfies \ass{R1}{C, N}, then every level surface $(\mc{S}, \gamma [\tau])$ trivially satisfies \ass{r1}{C, N}, with the corresponding data.
\footnote{In particular, one restricts the frame elements $e^i_a$ to each $\mc{S}_\tau$.}

\begin{remark}
Again, the $L^4_x$-norm in the \ass{R1}{} condition is only a matter of convenience, as this exponent ``$4$" throughout this paper can be replaced by any $q \in (2, \infty]$.
\end{remark}

Finally, we say $(\mc{N}, \gamma)$ satisfies \ass{R2}{C, N}, with data $\{ U_i, \varphi_i, \eta_i, \tilde{\eta}_i, e^i \}_{i = 1}^N$, iff:
\begin{itemize}
\item $(\mc{N}, \gamma)$ satisfies \ass{R1}{C, N}, with data $\{ U_i, \varphi_i, \eta_i, \tilde{\eta}_i, e^i \}_{i = 1}^N$.

\item For each $1 \leq i \leq N$, the $\varphi_i$-coordinate vector fields $\partial^i_1, \partial^i_2$ satisfy
\[ \| \nabla ( \mf{e} \partial^i_a ) \|_{ L^{\infty, 2}_{t, x} } \leq C \text{,} \qquad a \in \{ 1, 2 \} \text{.} \]

\item For each $1 \leq i \leq N$, the second \emph{coordinate} derivatives of $\eta_i$ satisfy
\[ | \partial^i_a \partial^i_b \eta_i | \leq C \text{,} \qquad a, b \in \{ 1, 2 \} \text{.} \]
\end{itemize}
Once again, if $(\mc{N}, \gamma)$ satisfies \ass{R2}{C, N}, then every level surface $(\mc{S}, \gamma [\tau])$ trivially satisfies \ass{r2}{C, N}, with the corresponding data.

\subsection{Scalar Reductions} \label{sec.fol_scal}

In the Sobolev-type inequalities of Proposition \ref{thm.gns_ineq}, we reduced estimates for a horizontal tensorial quantity $\Psi$ to one for a scalar quantity by dealing with $| \Psi |^2$ instead.
This allowed us to take advantage of the compatibility of $\nabla$ with $\gamma$ and avoid dealing with connection quantities resulting from a choice of frames.
However, this method will not suffice in some situations, including some comparisons of derivative norms and certain Besov-type estimates for tensor fields.

In this section, we discuss in general the systematic reduction of tensorial quantities to localized scalar analogues.
Using this process, we can define coordinate-based Sobolev- and Besov-type norms on $(\mc{N}, \gamma)$.
With respect to these coordinate norms, the main bilinear product estimates of this paper can be reduced to their corresponding (much easier to prove) Euclidean counterparts.
Moreover, using the regularity conditions defined in Section \ref{sec.fol_reg}, we show that the coordinate-based norms are in fact comparable to the geometric Besov norms.

Assume for now $(\mc{N}, \gamma)$ satisfies \ass{R1}{C, N}, with data $\{ U_i, \varphi_i, \eta_i, \tilde{\eta}_i, e_i \}_{i = 1}^N$.
For convenience, we will also use $\eta_i$ and $\tilde{\eta}_i$ to denote their equivariant transports, $\mf{e} \eta_i$ and $\mf{e} \tilde{\eta}_i$, respectively.
Similarly, we let $\varphi_i$ denote the function $\varphi_i: \mc{N}_{U_i} \rightarrow \R^2$, whose components are the equivariant transports of the components of $\varphi_i$.

Moreover, for $1 \leq i \leq N$, we let $e_{\ast i}^1, e_{\ast i}^2 \in \mc{C}^\infty T^0_1 \mc{N}_{U_i}$ denote the dual coframe to $e^i_1, e^i_2$.
Define $i \mc{X}^r_l$ to be the collection of all local horizontal fields of the form
\[ e^i_{a_1} \otimes \dots \otimes e^i_{a_l} \otimes e_{\ast i}^{b_1} \otimes \dots \otimes e_{\ast i}^{b_r} \in \mc{C}^\infty \ul{T}^l_r \mc{N}_{U_i} \text{,} \]
where $a_1, \dots, a_l, b_1, \dots, b_r \in \{ 1, 2 \}$.
This family $i \mc{X}^r_l$ consists of exactly $2^{r + l}$ elements and forms a local orthonormal frame on $\mc{N}_{U_i}$ for $\mc{C}^\infty \ul{T}^r_l \mc{N}$.
Moreover, by the \ass{R1}{C, N} assumption, each $X \in i \mc{X}^r_l$ satisfies
\begin{equation} \label{eq.frame_basis} \| X \|_{ L^{\infty, \infty}_{t, x} } \leq 1 \text{,} \qquad \| \nabla X \|_{ L^{\infty, 4}_{t, x} } \lesssim_C r + l \text{.} \end{equation}

Given any $\Psi \in \mc{C}^\infty \ul{T}^r_l \mc{N}$ and $X \in i \mc{X}^r_l$, its full contraction $\Psi ( X )$ defines an element of $\mc{C}^\infty \mc{N}_{U_i}$.
Moreover, the restriction $\Psi | \mc{N}_{U_i}$ can be entirely reconstructed from all of its contractions with elements of $i \mc{X}^r_l$.
Observe that if $\Phi \in \mc{C}^\infty \ul{T}^r_l \mc{N}$ as well, then the orthonormality of the $e^i_a$'s and $e_{\ast i}^b$'s implies that
\begin{equation} \label{eq.scalar_red_ip} \langle \Psi, \Phi \rangle | U_i = \sum_{ X \in i \mc{X}^r_l } [ \Psi (X) \cdot \Phi (X) ] \text{,} \qquad | \Psi |^2 | U_i = \sum_{ X \in i \mc{X}^r_l } [ \Psi (X) ]^2 \text{.} \end{equation}

\begin{remark}
As a special case, we can define $i \mc{X}^0_0$ to be the set containing only the constant function with value $1$ on $U_i$.
Then, the the scalar reduction theory described here also applies as written to the trivial case $r = l = 0$.
\end{remark} 

We now prove some basic properties for our scalar reduction scheme, based on the assumptions in Section \ref{sec.geom_reg}.
First is the following bound for first derivatives.

\begin{proposition} \label{thm.scalar_red_D}
Assume $(\mc{N}, \gamma)$ satisfies \ass{R1}{C, N}.
Then,
\begin{equation} \label{eq.scalar_red_D} \sum_{ X \in i \mc{X}^r_l } \| \nabla [ \Psi (X) ] \|_{ L^{p, 2}_{t, x} } \lesssim_{C, N, r, l} \| \nabla \Psi \|_{ L^{p, 2}_{t, x} } + \| \Psi \|_{ L^{p, 2}_{t, x} } \text{,} \end{equation}
for any $1 \leq i \leq N$, $p \in [1, \infty]$, and $\Psi \in \mc{C}^\infty \ul{T}^r_l \mc{N}$.
\end{proposition}

\begin{proof}
Fix $X \in i \mc{X}^r_l$, and define $\nabla \Psi (X)$ and $\Psi ( \nabla X )$ to be the local horizontal $1$-forms mapping any $Y \in \mc{C}^\infty \ul{T}^1_0 \mc{N}_{U_i}$ to $( \nabla_Y \Psi ) (X)$ and $\Psi ( \nabla_Y X )$, respectively.
Applying the Leibniz rule along with H\"older's inequality yields
\begin{align*}
\| \nabla [ \Psi (X) ] \|_{ L^{p, 2}_{t, x} } &\leq \| \nabla \Psi (X) \|_{ L^{p, 2}_{t, x} } + \| \Psi ( \nabla X ) \|_{ L^{p, 2}_{t, x} } \\
&\leq \| \nabla \Psi (X) \|_{ L^{p, 2}_{t, x} } + \| \Psi \|_{ L^{2, 4}_{t, x} } \| \nabla X \|_{ L^{\infty, 4}_{t, x} } \text{.}
\end{align*}
By \eqref{eq.gns_1} and \eqref{eq.frame_basis},
\[ \| \Psi \|_{ L^{2, 4}_{t, x} } \| \nabla X \|_{ L^{\infty, 4}_{t, x} } \lesssim \| \nabla \Psi \|_{ L^{p, 2}_{t, x} } + \| \Psi \|_{ L^{p, 2}_{t, x} } \text{.} \]
Combining the above and summing over all $X \in \mc{X}^r_l (i)$ completes the proof.
\end{proof}

Assume again that $(\mc{N}, \gamma)$ satisfies \ass{R1}{C, N}, with data $\{ U_i, \varphi_i, \eta_i, \tilde{\eta}_i, e_i \}_{i = 1}^N$.
Given $\Psi \in \mc{C}^\infty \ul{T}^r_l \mc{N}$ and $X \in i \mc{X}^r_l$, since $\eta_i \Psi (X)$ is supported within $\mc{N}_{U_i}$, then
\[ [ \eta_i \cdot \Psi (X) ] \circ \varphi_i^{-1} \]
can be treated as a smooth function on $[0, \delta] \times \R^2$.
As a result, given $a \in [1, \infty]$, $p \in [1, \infty]$, and $s \in \R$, we can define coordinate-based Besov-type norms
\begin{align*}
\| \Psi \|_{ \mc{B}^{a, p, s}_{\ell, t, x} } &= \sum_{1 \leq i \leq N} \sum_{ X \in i \mc{X}^r_l } \| [ \eta_i \cdot \Psi ( X ) ] \circ \varphi_i^{-1} \|_{ B^{a, p, s}_{\ell, t, x} } \text{.}
\end{align*}
The norm on the right-hand side in the above refers to the corresponding standard Besov norm on $[0, \delta] \times \R^2$; see Appendix \ref{sec.eucl} for precise definitions.

\begin{remark}
We will use the same symbols to denote both the geometric Besov norms and the corresponding norms in Euclidean space.
The specific norm that is referenced in any particular instance will depend on context.
\end{remark}

Note that the above norms depend very heavily on the data associated with the \ass{R1}{} condition.
On the other hand, estimates involving these norms will depend only on the regularity constants $C$ and $N$.
Thus, we can make use of these noncanonical norms as intermediate quantities in order to ultimately reach the invariant geometric Besov norms defined in Section \ref{sec.fol_hor}.

\subsection{The Besov Comparison Property} \label{sec.fol_bcomp}

In Section \ref{sec.fol_hor}, we have defined geometric Besov norms on $(\mc{N}, \gamma)$.
Moreover, in Section \ref{sec.fol_scal}, we defined coordinate-based Besov norms, based on a localized scalar decomposition of horizontal tensor fields with respect to data associated with the \ass{R1}{} condition.
The question that remains is whether these two types of Besov norms are related to each other.

In the subsequent proposition, we show that under the \ass{R1}{} condition, corresponding geometric and coordinate-based Besov norms are in fact equivalent, as long as one considers only a low number of horizontal derivatives.
As a result, in order to obtain estimates involving such geometric Besov norms, we only need to establish estimates involving the corresponding coordinate-based norms.

\begin{proposition} \label{thm.comp_main}
Assume $(\mc{N}, \gamma)$ satisfies \ass{R1}{C, N}, and suppose that $a \in [1, \infty]$, $p \in [1, \infty]$, $s \in (-1, 1)$, and $\Psi \in \mc{C}^\infty \ul{T}^r_l \mc{N}$.
Then,
\begin{align}
\label{eq.comp_main} \| \Psi \|_{ B^{a, p, s}_{\ell, t, x} } \simeq_{C, N, s, r, l} \| \Psi \|_{ \mc{B}^{a, p, s}_{\ell, t, x} } \text{.}
\end{align}
Furthermore, if $\{ U_i, \varphi_i, \eta_i, \tilde{\eta}_i, e_i \}_{i = 1}^N$ denotes the associated data, then
\begin{equation} \label{eq.comp_main_tech} \sum_{i = 1}^N \sum_{ X \in i \mc{X}^r_l } \| [ \tilde{\eta}_i \cdot \Psi (X) ] \circ \varphi_i^{-1} \|_{ B^{a, p, s}_{\ell, t, x} } \lesssim_{C, N, s, r, l} \| \Psi \|_{ B^{a, p, s}_{\ell, t, x} } \text{.} \end{equation}
\end{proposition}

\begin{proof}
See Appendix \ref{sec.bcomp}.
\end{proof}

The proof of Proposition \ref{thm.comp_main} is rather technical, involving decompositions using both geometric and Euclidean L-P operators in tandem.
As a result, we defer this proof until Appendix \ref{sec.bcomp}.
In the meantime, we assume the conclusions of Proposition \ref{thm.comp_main}, and we discuss its various consequences.

The first consequence of Proposition \ref{thm.comp_main} is a (geometric) Besov refinement of the standard $L^4$-$H^{1/2}$-Sobolev inequality in $2$-dimensions.

\begin{proposition} \label{thm.sob_frac_sh}
Assume that $(\mc{N}, \gamma)$ satisfies \ass{R1}{C, N}.
If $\Psi \in \mc{C}^\infty \ul{T}^r_l \mc{N}$, then
\begin{equation} \label{eq.sob_frac_sh} \| \Psi \|_{ L^{\infty, 4}_{t, x} } \lesssim_{ C, N, r, l } \| \Psi \|_{ B^{2, \infty, 1/2}_{\ell, t, x} } \text{.} \end{equation}
\end{proposition}

\begin{proof}
Let $\{ U_i, \varphi_i, \eta_i, \tilde{\eta}_i, e_i \}_{i = 1}^N$ be the data associated with the \ass{R1}{C, N} condition.
By Proposition \ref{thm.comp_main}, we need only prove the bound
\[ \| \Psi \|_{ L^{\infty, 4}_{t, x} } \lesssim \| \Psi \|_{ \mc{B}^{2, \infty, 1/2}_{\ell, t, x} } \text{.} \]

By our scalar reduction scheme, we can write
\[ | \Psi | \leq \sum_{i = 1}^N \sum_{ X \in i \mc{X}^r_l } | \eta_i \Psi (X) | \text{.} \]
Thus, by \eqref{eqr.change_of_coord} and the standard $H^{1/2}$-$L^4$-Sobolev inequality on $\R^2$, we have
\begin{align*}
\| \Psi \|_{ L^{\infty, 4}_{t, x} } &\lesssim \sum_{i = 1}^N \sum_{ X \in i \mc{X}^r_l } \| \eta_i \Psi (X) \circ \varphi_i^{-1} \|_{ L^{\infty, 4}_{t, x} } \\
&\lesssim \sup_{ \tau \in [0, \delta] } \sum_{i = 1}^N \sum_{ X \in i \mc{X}^r_l } \| \eta_i \Psi (X) \circ \varphi_i^{-1} [\tau] \|_{ H^{1/2}_x } \text{,}
\end{align*}
where the right-hand side is the standard Sobolev norm on $\R^2$.
By the L-P characterization of Sobolev norms (see Section \ref{sec.eucl_lp}), we obtain \eqref{eq.sob_frac_sh}:
\[ \| \Psi \|_{ L^{\infty, 4}_{t, x} } \lesssim \sup_{ \tau \in [0, \delta] } \sum_{i = 1}^N \sum_{ X \in i \mc{X}^r_l } \| \eta_i \Psi (X) \circ \varphi_i^{-1} [\tau] \|_{ B^{2, 1/2}_{\ell, x} } \lesssim \| \Psi \|_{ \mc{B}^{2, \infty, 1/2}_{\ell, t, x} } \text{.} \qedhere \]
\end{proof}

Although our setting here is one of a foliation $(\mc{N}, \gamma)$ of copies of $\mc{S}$, one can also obtain as a special case corresponding estimates on a single surface $(\mc{S}, h)$.
Indeed, this can be done by considering a horizontal metric $\gamma$ that is equivariant.
As this represents a system with static geometry, the estimates obtained on this foliation reduce to corresponding estimates on any single level set of this foliation.

For example, \eqref{eq.sob_frac_sh} reduces to a fractional Sobolev inequality.

\begin{corollary} \label{thm.sob_frac_shf}
Assume that $(\mc{S}, h)$ satisfies \ass{r1}{C, N}.
If $F \in \mc{C}^\infty T^r_l \mc{S}$, then
\begin{equation} \label{eq.sob_frac_shf} \| F \|_{ L^4_x } \lesssim_{ C, N, r, l } \| F \|_{ H^{1/2}_x } \text{.} \end{equation}
\end{corollary}

We will also require the following single-surface variant of Proposition \ref{thm.comp_main}.

\begin{corollary} \label{thm.comp_mainf}
Assume $(\mc{N}, \gamma)$ satisfies \ass{R1}{C, N}, with data $\{ U_i, \varphi_i, \eta_i, \tilde{\eta}_i, e^i \}_{i = 1}^N$.
If $\tau \in [0, \delta]$, $a \in [1, \infty]$, $s \in (-1, 1)$, and $\Psi \in \mc{C}^\infty \ul{T}^r_l \mc{N}$, then
\begin{align}
\label{eq.comp_mainf} \sum_{i = 1}^N \sum_{ X \in i \mc{X}^r_l } \| \eta_i \Psi [\tau] (X [\tau]) \circ \varphi_i^{-1} \|_{ B^{a, s}_{\ell, x} } &\simeq_{C, N, s, r, l} \| \Psi [\tau] \|_{ B^{a, s}_{\ell, x} } \text{,} \\
\notag \sum_{i = 1}^N \sum_{ X \in i \mc{X}^r_l } \| \tilde{\eta}_i \Psi [\tau] (X [\tau]) \circ \varphi_i^{-1} \|_{ B^{a, s}_{\ell, x} } &\lesssim_{C, N, s, r, l} \| \Psi [\tau] \|_{ B^{a, s}_{\ell, x} } \text{.}
\end{align}
\end{corollary}

Corollary \ref{thm.comp_mainf} can be immediately proved by applying Proposition \ref{thm.comp_main} to the foliation $(\mc{N}, \gamma^\prime)$, where $\gamma^\prime = \mf{e} ( \gamma [\tau] )$ is the equivariant transport of $\gamma [\tau]$.

\subsection{Product Estimates} \label{sec.fol_est}

Next, we obtain some tensorial product estimates involving geometric Besov norms using our scalar reduction machinery.
The basic strategy, as always, is to use Proposition \ref{thm.comp_main} to convert geometric Besov norms to equivalent coordinate-based norms.
From this point, one can proceed by applying standard estimates in Euclidean space; these are given in Appendix \ref{sec.eucl}.

To avoid confusion, in future proofs, we will let $\partial$ denote the Euclidean gradient on $\R^2$, and we will let $\nabla$ denote the horizontal covariant differential for $(\mc{N}, \gamma)$.

\begin{theorem} \label{thm.est_prod_elem}
Assume that $(\mc{N}, \gamma)$ satisfies \ass{R1}{C, N}.
Furthermore, consider horizontal tensor fields $\Psi \in \mc{C}^\infty \ul{T}^{r_1}_{l_1} \mc{N}$ and $\Phi \in \mc{C}^\infty \ul{T}^{r_2}_{l_2} \mc{N}$.
\begin{itemize}
\item If $a \in [1, \infty]$, $p \in [1, \infty]$, and $s \in (-1, 1)$, then
\begin{align}
\label{eq.est_prod_elem} \| \Phi \otimes \Psi \|_{ B^{a, p, s}_{\ell, t, x} } &\lesssim_{ C, N, s, r_1, l_1, r_2, l_2 } ( \| \nabla \Phi \|_{ L^{\infty, 2}_{t, x} } + \| \Phi \|_{ L^{\infty, \infty}_{t, x} } ) \| \Psi \|_{ B^{a, p, s}_{\ell, t, x} } \text{.}
\end{align}

\item If $s \in [0, 1)$, and if $p, p_1, p_2 \in [1, \infty]$ satisfy $p^{-1} = p_1^{-1} + p_2^{-1}$, then
\begin{align}
\label{eq.est_prod_sob} \| \Phi \otimes \Psi \|_{ B^{1, p, s}_{\ell, t, x} } &\lesssim_{ C, N, s, r_1, l_1, r_2, l_2 } \| \Phi \|_{ B^{ 2, p_1, (1 + s) / 2 }_{\ell, t, x} } \| \Psi \|_{ B^{ 2, p_2, (1 + s) / 2 }_{\ell, t, x} } \text{.}
\end{align}
\end{itemize}
\end{theorem}

\begin{proof}
By Proposition \ref{thm.comp_main}, we need only prove \eqref{eq.est_prod_elem} and \eqref{eq.est_prod_sob}, but with all geometric Besov norms replaced by their coordinate-based analogues.
As usual, we assume data $\{ U_i, \varphi_i, \eta_i, \tilde{\eta}_i, e_i \}_{i = 1}^N$ associated with the \ass{R1}{C, N} condition.

For \eqref{eq.est_prod_elem}, we begin with the scalar reduction process:
\begin{align*}
\| \Phi \otimes \Psi \|_{ \mc{B}^{a, p, s}_{\ell, t, x} } &= \sum_{i = 1}^N \sum_{ W \in i \mc{X}^{r_1 + r_2}_{l_1 + l_2} } \| [ \eta_i \cdot ( \Phi \otimes \Psi ) (W) ] \circ \varphi_i^{-1} \|_{ B^{a, p, s}_{\ell, t, x} } \\
&= \sum_{i = 1}^N \sum_{ X \in i \mc{X}^{r_1}_{l_1} } \sum_{ Y \in i \mc{X}^{r_2}_{l_2} } \| [ \tilde{\eta}_i \Phi (X) \cdot \eta_i \Psi (Y) ] \circ \varphi_i^{-1} \|_{ B^{a, p, s}_{\ell, t, x} } \text{.}
\end{align*}
Observe that the elements of $i \mc{X}^{r_1 + r_2}_{l_1 + l_2}$ are precisely the tensor products of elements of $i \mc{X}^{r_1}_{l_1}$ and $i \mc{X}^{r_2}_{l_2}$.
Applying the Euclidean analogue \eqref{eqc.est_prod_elem} of our desired estimate along with the \ass{R1}{C, N} condition on $(\mc{N}, \gamma)$, we obtain
\begin{align*}
\| \Phi \otimes \Psi \|_{ B^{a, p, s}_{\ell, t, x} } &\lesssim \sum_{i = 1}^N \sum_{ X \in i \mc{X}^{r_1}_{l_1} } \{ \| \partial [ \tilde{\eta}_i \Phi (X) \circ \varphi_i^{-1} ] \|_{ L^{\infty, 2}_{t, x} } + \| \tilde{\eta}_i \Phi (X) \circ \varphi_i^{-1} \|_{ L^{\infty, \infty}_{t, x} } \} \\
&\qquad \cdot \sum_{ Y \in i \mc{X}^{r_2}_{l_2} } \| \eta_i \Psi (Y) \circ \varphi_i^{-1} \|_{ B^{a, p, s}_{\ell, t, x} } \\
&\lesssim \sup_{ 1 \leq i \leq N } \sum_{ X \in i \mc{X}^{r_1}_{l_1} } \{ \| \nabla [ \Phi (X) ] \|_{ L^{\infty, 2}_{t, x} } + \| \Phi (X) \|_{ L^{\infty, \infty}_{t, x} } \} \cdot \| \Psi \|_{ \mc{B}^{a, p, s}_{\ell, t, x} } \text{.}
\end{align*}
Finally, applying \eqref{eq.scalar_red_D} yields \eqref{eq.est_prod_elem}.

For \eqref{eq.est_prod_sob}, we decompose as before and apply the Euclidean analogue \eqref{eqc.est_prod_sob}:
\begin{align*}
\| \Phi \otimes \Psi \|_{ \mc{B}^{1, p, s}_{\ell, t, x} } &= \sum_{i = 1}^N \sum_{ X \in i \mc{X}^{r_1}_{l_1} } \sum_{ Y \in i \mc{X}^{r_2}_{l_2} } \| [ \tilde{\eta}_i \Phi (X) \cdot \eta_i \Psi (Y) ] \circ \varphi_i^{-1} \|_{ B^{1, p, s}_{\ell, t, x} } \\
&\lesssim \sup_{1 \leq i \leq N} \sum_{ X \in i \mc{X}^{r_1}_{l_1} } \| \tilde{\eta}_i \Phi (X) \circ \varphi_i^{-1} \|_{ B^{2, p_1, (1 + s) / 2}_{\ell, t, x} } \\
&\qquad \cdot \sum_{i = 1}^N \sum_{ Y \in i \mc{X}^{r_2}_{l_2} } \| \eta_i \Psi (Y) \circ \varphi_i^{-1} \|_{ B^{2, p_2, (1 + s) / 2}_{\ell, t, x} } \text{.}
\end{align*}
By \eqref{eq.comp_main_tech} and the definition of the coordinate Besov norms, we obtain \eqref{eq.est_prod_sob}:
\begin{align*}
\| \Phi \otimes \Psi \|_{ \mc{B}^{1, p, s}_{\ell, t, x} } &\lesssim \| \Phi \|_{ B^{2, p_1, (1 + s) / 2}_{\ell, t, x} } \| \Psi \|_{ \mc{B}^{2, p_2, (1 + s) / 2}_{\ell, t, x} } \text{.} \qedhere
\end{align*}
\end{proof}

\begin{remark}
In fact, \eqref{eq.est_prod_elem} and \eqref{eq.est_prod_sob} still hold if the products $\Phi \otimes \Psi$ on the left-hand sides are replaced by zero or more contractions, metric ($\gamma$-)contractions, and volume form ($\epsilon$-)contractions of $\Phi \otimes \Psi$.
This is due to the observation that the above contraction operations commute with all the geometric L-P operators.
\end{remark}

Finally, by taking $p = \infty$ and $\gamma$ to be equivariant in Theorem \ref{thm.est_prod_elem}, and by recalling Proposition \ref{thm.sobolev}, we obtain analogous estimates for a fixed surface.

\begin{corollary} \label{thm.est_prod_elemf}
Assume that $(\mc{S}, h)$ satisfies \ass{r1}{C, N}.
\begin{itemize}
\item If $a \in [1, \infty]$, $s \in (-1, 1)$, $F \in \mc{C}^\infty T^{r_1}_{l_1} \mc{S}$, and $G \in \mc{C}^\infty T^{r_2}_{l_2} \mc{S}$, then
\begin{align}
\label{eq.est_prod_elemf} \| F \otimes G \|_{ B^{a, s}_{\ell, x} } &\lesssim_{ C, N, s, r_1, l_1, r_2, l_2 } ( \| \nabla F \|_{ L^2_x } + \| F \|_{ L^\infty_x } ) \| G \|_{ B^{a, s}_{\ell, x} } \text{.}
\end{align}

\item If $s \in [0, 1)$, $F \in \mc{C}^\infty T^{r_1}_{l_1} \mc{S}$, and $G \in \mc{C}^\infty T^{r_2}_{l_2} \mc{S}$, then
\begin{align}
\label{eq.est_prod_sobf} \| F \otimes G \|_{ B^{1, s}_{\ell, x} } &\lesssim_{ C, N, s, r_1, l_1, r_2, l_2 } \| F \|_{ H^{ (1 + s) / 2 }_x } \| G \|_{ H^{ (1 + s) / 2 }_x } \text{.}
\end{align}
\end{itemize}
\end{corollary}

\subsection{Conformal Transformations} \label{sec.fol_conf}

We return to the topic of conformal transformations, but now in the foliation setting.
Consider another horizontal metric $\bar{\gamma} \in \mc{C}^\infty \ul{T}^0_2 \mc{N}$, related to $\gamma$ via the \emph{conformal transform}
\[ \bar{\gamma} = e^{2 u} \gamma \text{,} \qquad u \in \mc{C}^\infty \mc{N} \text{.} \]
We can view this as a family of conformal transformations for the $\gamma [\tau]$'s, such that the conformal factors $e^{2 u} [\tau]$ also vary smoothly with respect to $\tau$.
If we let $\bar{\epsilon}$ denote the volume form associated with $\bar{\gamma}$, then we have
\[ \bar{\gamma}_{ab} = e^{2 u} \gamma_{ab} \text{,} \qquad \bar{\epsilon}_{ab} = e^{2 u} \epsilon_{ab} \text{,} \qquad \bar{\gamma}^{ab} = e^{-2 u} \gamma^{ab} \text{,} \qquad \bar{\epsilon}^{ab} = e^{-2 u} \epsilon^{ab} \text{.} \]
In general, we denote objects with respect to $\bar{\gamma}$ with a ``bar" over the symbol.

We can now port Proposition \ref{thm.conf_regf} directly to the foliation setting.

\begin{proposition} \label{thm.conf_reg}
Let $u \in \mc{C}^\infty \mc{N}$, and let $\bar{\gamma} = e^{2 u} \gamma$ be a conformal transform of $\gamma$.
Furthermore, assume that the following quantity is sufficiently small:
\[ D = \| u \|_{ L^{\infty, \infty}_{t, x} } + \| \nabla u \|_{ L^{\infty, 4}_{t, x} } \ll 1 \]
Then, the following properties hold:
\begin{itemize}
\item Suppose $(\mc{N}, \gamma)$ satisfies \ass{R0}{C, N}, with data $\{ U_i, \varphi_i, \eta_i \}_{i = 1}^N$.
Then, there exists a constant $C^\prime$, depending on $C$ and $N$, such that $(\mc{N}, \bar{\gamma})$ satisfies \ass{R0}{C^\prime, N}, with the same data $\{ U_i, \varphi_i, \eta_i \}_{i = 1}^N$.

\item Suppose $(\mc{N}, \gamma)$ also satisfies \ass{R1}{C, N}, with data $\{ U_i, \varphi_i, \eta_i, \tilde{\eta}_i, e^i \}_{i = 1}^N$, with $e^i = (e^i_1, e^i_2)$ for each $1 \leq i \leq N$.
If we define $\bar{e}^i = ( e^{-u} e^i_1, e^{-u} e^i_2 )$, then there is some constant $C^\prime$, depending on $C$ and $N$, such that $(\mc{N}, \bar{\gamma})$ satisfies \ass{R1}{C^\prime, N}, with data $\{ U_i, \varphi_i, \eta_i, \tilde{\eta}_i, \bar{e}^i \}_{i = 1}^N$.
Furthermore, for any $a \in [1, \infty]$, $p \in [1, \infty]$, $s \in (-1, 1)$, and $\Psi \in \mc{C}^\infty \ul{T}^r_l \mc{N}$, we have
\begin{equation} \label{eq.conf_besov} \| \Psi \|_{ \bar{B}^{a, p, s}_{\ell, t, x} } \simeq_{ C, N, s, r, l } \| \Psi \|_{ B^{a, p, s}_{\ell, t, x} } \text{.} \end{equation}
\end{itemize}
\end{proposition}

\begin{proof}
That the \ass{R0}{C, N} condition for $(\mc{N}, \gamma)$ implies the same for $(\mc{N}, \bar{\gamma})$ follows immediately by applying Proposition \ref{thm.conf_regf} to each $(\mc{S}, \gamma [\tau])$.
A similar argument using Proposition \ref{thm.conf_regf} also shows that if $(\mc{N}, \gamma)$ satisfies \ass{R1}{C, N}, then so does $(\mc{N}, \bar{\gamma})$, for the desired data given in the statement of the proposition.

Thus, it remains only to prove the Besov comparison \eqref{eq.conf_besov}.
By the definitions of the coordinate-based Besov norms with respect to both $\bar{\gamma}$ and $\gamma$, we have
\begin{align*}
\| \Psi \|_{ \bar{\mc{B}}^{a, p, s}_{\ell, t, x} } &= \sum_{i = 1}^N \sum_{\bar{X} \in i \bar{\mc{X}}^r_l } \| \eta_i \Psi (\bar{X}) \circ \varphi_i^{-1} \|_{ B^{a, p, s}_{\ell, t, x} } \\
&= \sum_{i = 1}^N \sum_{X \in i \mc{X}^r_l } \| e^{(r - l) u} \eta_i \Psi ( X ) \circ \varphi_i^{-1} \|_{ B^{a, p, s}_{\ell, t, x} } \\
&= \| e^{ (r - l) u } \Psi \|_{ \mc{B}^{a, p, s}_{\ell, t, x} } \text{.}
\end{align*}
The relationship between $i \mc{X}^r_l$ and $i \bar{\mc{X}}^r_l$ follows from the relationship between the corresponding data for the \ass{R1}{} conditions for $(\mc{N}, \gamma)$ and $(\mc{N}, \bar{\gamma})$.
Since both $(\mc{N}, \gamma)$ and $(\mc{N}, \bar{\gamma})$ satisfy the \ass{R1}{} condition, then the comparison \eqref{eq.comp_main} yields
\[ \| \Psi \|_{ \bar{B}^{a, p, s}_{\ell, t, x} } = \| e^{ (r - l) u } \Psi \|_{ B^{a, p, s}_{\ell, t, x} } \lesssim ( \| \nabla e^{ (r - l) u } \|_{ L^{\infty, 2}_{t, x} } + \| e^{ (r - l) u } \|_{ L^{\infty, \infty}_{t, x} } ) \| \Psi \|_{ B^{a, p, s}_{\ell, t, x} } \text{.} \]
By \eqref{eq.conf_unif} and our smallness assumption for $u$, one inequality in \eqref{eq.conf_besov} follows.
The reverse inequality can be proved using similar means.
\end{proof}

\section{Covariant Evolution} \label{sec.cfol}

We maintain the general setting of Section \ref{sec.fol}, that is, the foliation $\mc{N} = [0, \delta] \times \mc{S}$, with a horizontal metric $\gamma \in \mc{C}^\infty \ul{T}^0_2 \mc{N}$.
Moreover, we maintain all the definitions and notations established throughout Section \ref{sec.fol}.

Note the horizontal covariant differential $\nabla$ already provides a covariant notion of tensorial differentiation in the horizontal directions.
Next, we wish to formalize a covariant notion of differentiation in the $t$-direction.
Afterwards, we define the ``inverse" of this operation---a covariant notion of \emph{integration} in the $t$-direction.

\subsection{Evolution} \label{sec.cfol_ev}

The identifications $\Xi_\tau$ (see Section \ref{sec.fol_hor}) can be utilized to define a \emph{vertical Lie derivative} of horizontal tensor fields.
Given $A \in \mc{C}^\infty \ul{T}^r_l \mc{N}$, we define
\[ \mf{L}_t A [\tau] = \lim_{ \tau^\prime \rightarrow \tau } \frac{ A [\tau^\prime] - A [\tau] }{ \tau^\prime - \tau } \in \mc{C}^\infty T^r_l \mc{S} \text{,} \]
which together define the Lie derivative $\mf{L}_t A \in \mc{C}^\infty \ul{T}^r_l \mc{N}$.
Heuristically, $\mf{L}_t A$ measures how $A$ evolves as $t$ increases, \emph{with respect to the diffeomorphisms $\Xi_\tau$}.
Note in particular that $A$ is equivariant if and only if $\mf{L}_t A$ vanishes identically.

\begin{remark}
Note that $\mf{L}_t$ is independent of $\gamma$ and $\epsilon$.
\end{remark}

\begin{remark}
$\mf{L}_t$ can alternately be defined as the (standard) Lie derivative with respect to the lift of the coordinate vector field $d/dt$ on $[0, \delta]$ to $\mc{N}$.
\end{remark}

Next, define the \emph{second fundamental form} (with respect to $\gamma$) to be
\[ k = \frac{1}{2} \mf{L}_t \gamma \in \mc{C}^\infty \ul{T}^0_2 \mc{N} \text{.} \]
This describes the evolution of the metrics $\gamma [\tau]$ as $\tau$ increases.
Of particular importance will be the $\gamma$-trace of $k$, called the \emph{mean curvature}, or \emph{expansion}, of $k$:
\[ \trace k = \gamma^{ab} k_{ab} \in \mc{C}^\infty \mc{N} \text{.} \]
Elementary computations yield the following basic identities:
\begin{equation} \label{eq.str_ev_vol} \frac{1}{2} ( \mf{L}_t \gamma^{-1} )^{ab} = - \gamma^{ac} \gamma^{bd} k_{cd} \text{,} \qquad ( \mf{L}_t \epsilon )_{ab} = ( \trace k ) \epsilon_{ab} \text{.} \end{equation}

A direct calculation also yields the following commutator identity for $\mf{L}_t$ and $\nabla$.

\begin{proposition} \label{thm.comm_lie}
If $\Psi \in \mc{C}^\infty \ul{T}^r_l \mc{N}$, then
\begin{align}
\label{eq.comm_lie} [ \mf{L}_t, \nabla_a ] \Psi_{u_1 \dots u_l}^{v_1 \dots v_r} &= - \sum_{i = 1}^l \gamma^{cd} ( \nabla_a k_{u_i c} + \nabla_{u_i} k_{ac} - \nabla_c k_{a u_i} ) \Psi_{u_1 \hat{d}_i u_l}^{v_1 \dots v_r} \\
\notag &\qquad + \sum_{j = 1}^r \gamma^{c v_j} ( \nabla_a k_{d c} + \nabla_d k_{ac} - \nabla_c k_{ad} ) \Psi_{u_1 \dots u_l}^{v_1 \hat{d}_j v_r} \text{.}
\end{align}
Here, $u_1 \hat{d}_i u_l$ denotes the set of indices $u_1 \dots u_l$, but with $u_i$ replaced by $d$.
The upper index notation $v_1 \hat{d}_j v_r$ is defined analogously.
\end{proposition}

We can now use $\mf{L}_t$ and $k$ to define a corresponding ($\gamma$-)\emph{covariant} derivative along the $t$-direction.
Given $\Psi \in \mc{C}^\infty \ul{T}^r_l \mc{N}$, we define $\nabla_t \Psi \in \mc{C}^\infty \ul{T}^r_l \mc{N}$ by
\[ \nabla_t \Psi_{u_1 \dots u_l}^{v_1 \dots v_r} = \mf{L}_t \Psi_{u_1 \dots u_l}^{v_1 \dots v_r} - \sum_{i = 1}^l \gamma^{cd} k_{u_i c} \Psi_{u_1 \hat{d}_i u_l}^{v_1 \dots v_r} + \sum_{j = 1}^r \gamma^{c v_j} k_{cd} \Psi_{u_1 \dots u_l}^{v_1 \hat{d}_j v_r} \text{,} \]
where we use the same multi-index conventions as in Proposition \ref{thm.comm_lie}.
In particular, note that $\nabla_t$ and $\mf{L}_t$ coincide in the case of scalar fields.

Of particular importance is the curl of $k$:
\[ \mf{C} \in \mc{C}^\infty \ul{T}^0_3 \mc{N} \text{,} \qquad \mf{C}_{abc} = \nabla_b k_{ac} - \nabla_c k_{ab} \text{.} \]
With this, we can state the commutation formula for $\nabla_t$ and $\nabla$, which is a result of basic computations that we leave to the reader.

\begin{proposition} \label{thm.comm}
If $\Psi \in \mc{C}^\infty \ul{T}^r_l \mc{N}$, then
\begin{align}
\label{eq.comm_cov} [ \nabla_t, \nabla_a ] \Psi_{u_1 \dots u_l}^{v_1 \dots v_r} &= - \gamma^{cd} k_{ac} \nabla_d \Psi_{u_1 \dots u_l}^{v_1 \dots v_r} - \sum_{i = 1}^l \gamma^{cd} \mf{C}_{a u_i c} \Psi_{u_1 \hat{d}_i u_l}^{v_1 \dots v_r} \\
\notag &\qquad + \sum_{j = 1}^r \gamma^{c v_j} \mf{C}_{adc} \Psi_{u_1 \dots u_l}^{v_1 \hat{d}_j v_r} \text{.}
\end{align}
\end{proposition}

The constrast between \eqref{eq.comm_lie} and \eqref{eq.comm_cov} will play a fundamental role in the analysis.
Observe that in \eqref{eq.comm_lie}, we have terms on the right-hand side of the form $\nabla k \otimes \Psi$.
In \eqref{eq.comm_cov}, however, the gradient $\nabla k$ is replaced by $\mf{C}$, which can possess additional structure and regularity in certain situations.
\footnote{For example, if $\gamma$ are metrics induced from a larger pseudo-Riemannian manifold $(M, g)$, then $\mf{C}$ is present in the Codazzi equations relating the curvature of $M$ to the $\mc{S}_\tau$'s.}

We say that $\Psi \in \mc{C}^\infty \ul{T}^r_l \mc{N}$ is $t$\emph{-parallel} (with respect to $\gamma$) iff $\nabla_t \Psi \equiv 0$.
\footnote{Such fields were sometimes called \emph{Fermi-transported}, e.g., in \cite{kl_rod:cg}.}
Most importantly, one can see that $\gamma$, $\gamma^{-1}$, and $\epsilon$ are all $t$-parallel.
Consequently, the Levi-Civita connections $\nabla$ along with the operator $\nabla_t$ combine to form vector bundle connections for the horizontal bundles $\ul{T}^r_l \mc{N}$.
In addition, these connections are compatible with the natural bundle metrics $\langle \cdot, \cdot \rangle$ for the $\ul{T}^r_l \mc{N}$'s induced by $\gamma$.

Given $F \in \mc{C}^\infty T^r_l \mc{S}$, one can solve for the unique $\mf{p} F \in \mc{C}^\infty \ul{T}^r_l \mc{N}$, called the $t$-\emph{parallel transport} of $F$ from $0$, such that $\mf{p} F$ is $t$-parallel and $\mf{p} F [0] = F$.
Note that $t$-parallel horizontal tensor fields preserve the tensor norm, i.e.,
\[ | \mf{p} F | |_{ (\tau, x) } = | F |_x \text{,} \qquad x \in \mc{S} \text{.} \]
Finally, note that since $\mf{L}_t$ and $\nabla_t$ coincide for scalars fields, the equivariant and $t$-parallel propagators $\mf{e}$ and $\mf{p}$ coincide in the scalar setting.

\subsection{Covariant Integration} \label{sec.cfol_cint}

Next, given $\Psi \in \mc{C}^\infty \ul{T}^r_l \mc{N}$, we define its \emph{covariant $t$-integral} $\cint^t_0 \Psi$ from $t = 0$ to be the unique element of $\mc{C}^\infty \ul{T}^r_l \mc{N}$ satisfying
\[ \nabla_t \cint^t_0 \Psi = \Psi \text{,} \qquad \Psi [0] \equiv 0 \text{.} \]

A more explicit description can also be given using $t$-parallel frames and coframes.
If $U$ is an open neighborhood of $\mc{S}$, and if $\tilde{e}_1, \tilde{e}_2 \in \mc{C}^\infty T^1_0 U$ forms a local frame on $U$, then the $t$-parallel transports $e_a = \mf{p} \tilde{e}_a$, $a \in \{ 1, 2 \}$, form a local frame on each of the $U_\tau$'s.
Moreover, if $\tilde{e}_1$ and $\tilde{e}_2$ are orthonormal, then so are $e_1$ and $e_2$.
A completely analogous construction can be made using coframes instead of frames.

Indexing with respect to such a local $t$-parallel frame and coframe, we see that
\[ ( \cint^t_0 \Psi )_{u_1 \dots u_l}^{v_1 \dots v_r} |_{ (\tau, x) } = \int_0^\tau ( \Psi_{u_1 \dots u_l}^{v_1 \dots v_r} ) |_{ (w, x) } dw \text{,} \qquad x \in U \text{.} \]
Here, the right-hand side represents the standard integral over the scalar $\Psi_{u_1 \dots u_l}^{v_1 \dots v_r}$.
Note that if $\phi \in \mc{C}^\infty \mc{N}$, then $\cint^t_0 \phi$ is the usual integral:
\[ \cint^t_0 \phi |_{ (\tau, x) } = \int_0^\tau \phi |_{ (w, x) } dw \text{,} \qquad x \in \mc{S} \text{.} \]
As a result, one can view $\cint^t_0$ as a covariant extension of the standard integral.

From the above and the fundamental theorem of calculus, we have
\[ \cint^t_0 \nabla_t \Psi = \Psi - \mf{p} ( \Psi [0] ) \text{,} \qquad \Psi \in \mc{C}^\infty \ul{T}^r_l \mc{N} \text{.} \]
In particular, when $\Psi [0]$ vanishes entirely, $\cint^t_0$ acts as a formal inverse to $\nabla_t$.
In fact, when we ``integrate" covariant evolution equations involving horizontal tensor fields, we are actually applying these covariant integral operators $\cint^t_0$.

\begin{proposition} \label{thm.est_trace}
If $\Psi \in \mc{C}^\infty \ul{T}^r_l \mc{N}$, then
\begin{equation} \label{eq.est_trace} | \cint^t_0 \Psi | \leq \cint^t_0 | \Psi | \text{,} \qquad | \Psi | \leq | \mf{p} ( \Psi [0] ) | + | \cint^t_0 \nabla_t \Psi | \text{.} \end{equation}
\end{proposition}

\begin{proof}
Indexing with respect to an orthonormal $t$-parallel frame, we have
\[ | \cint^t_0 \Psi |^2 |_{ (\tau, x) } = \sum_{ \substack{u_1 \dots u_l \\ v_1 \dots v_r} } ( \cint^t_0 \Psi_{u_1 \dots u_l}^{v_1 \dots v_r} )^2 |_{ (\tau, x) } = \sum_{ \substack{u_1 \dots u_l \\ v_1 \dots v_r} } \paren{ \int_0^\tau \Psi_{u_1 \dots u_l}^{v_1 \dots v_r} |_{ (w, x) } d w }^2 \]
for any $x \in \mc{S}$.
By the integral Minkowski inequality,
\[ | \cint^t_0 \Psi | |_{ (\tau, x) } \leq \int_0^\tau { \bigg [ } \sum_{ \substack{u_1 \dots u_l \\ v_1 \dots v_r} } ( \Psi_{u_1 \dots u_l}^{v_1 \dots v_r} )^2 |_{ (w, x) } { \bigg ] }^{1/2} dw = \int_0^\tau | \Psi | |_{(w, x)} dw \text{,} \]
which proves the first inequality of \eqref{eq.est_trace}.
Using the same frame, we also have
\[ \Psi_{u_1 \dots u_l}^{v_1 \dots v_r} |_{ (\tau, x) } = \Psi_{u_1 \dots u_l}^{v_1 \dots v_r} |_{ (0, x) } + ( \cint^t_0 \nabla_t \Psi )_{u_1 \dots u_l}^{v_1 \dots v_r} |_{ (\tau, x) } \text{.} \]
Squaring the above and summing over all the indices yields
\[ | \Psi |^2 |_{ (\tau, x) } = \sum_{ \substack{u_1 \dots u_l \\ v_1 \dots v_r} } \brak{ \Psi_{u_1 \dots u_l}^{v_1 \dots v_r} |_{ (0, x) } + ( \cint^t_0 \nabla_t \Psi )_{u_1 \dots u_l}^{v_1 \dots v_r} }^2 |_{ (\tau, x) } \text{.} \]
Applying Minkowski's inequality yields the second part of \eqref{eq.est_trace}.
\end{proof}

\begin{remark}
As $t$-parallel fields preserve the pointwise tensor norm, the second inequality of \eqref{eq.est_trace} implies the following estimate:
\[ | \Psi | |_{ (\tau, x) } \leq | \Psi |_{ (0, x) } + | \cint^t_0 \nabla_t \Psi | |_{ (\tau, x) } \text{,} \qquad x \in \mc{S} \text{.} \]
\end{remark}

Finally, we briefly discuss commutations involving $\cint^t_0$.
Consider an operator
\[ L: \mc{C}^\infty \ul{T}^{r_1}_{l_1} \mc{N} \rightarrow \mc{C}^\infty \ul{T}^{r_2}_{l_2} \mc{N} \text{.} \]
By the relations between $\cint^t_0$ and $\nabla_t$, we have the formula
\[ [ \cint^t_0, L ] \Psi = - \cint^t_0 [ \nabla_t, L ] \cint^t_0 \Psi - \mf{p} ( L \cint^t_0 \Psi [0] ) \text{.} \]
If the last term on the right-hand side vanishes, then $[ \cint^t_0, L ] \Psi = - \cint^t_0 [ \nabla_t, L ] \cint^t_0 \Psi$.

Note that $\cint^t_0$ commutes with any such $L$ that commutes with $\nabla_t$, as long as $L$ maps any initially vanishing field to another initially vanishing field.
In particular, since $\nabla_t$ commutes with all contractions, metric contractions, and volume form contractions, then $\cint^t_0$ commutes with these operations as well.
Another important case is when $L = \nabla$, as we can now use \eqref{eq.comm_cov} to commute $\nabla$ and $\cint^t_0$:

\begin{proposition} \label{thm.comm_inv}
If $\Psi \in \mc{C}^\infty \ul{T}^r_l \mc{N}$, then
\begin{align}
\label{eq.comm_inv} [ \cint^t_0, \nabla_a ] \Psi_{u_1 \dots u_l}^{v_1 \dots v_r} &= \gamma^{cd} \cint^t_0 ( k_{ac} \nabla_d \cint^t_0 \Psi_{u_1 \dots u_l}^{v_1 \dots v_r} ) + \sum_{i = 1}^l \gamma^{cd} \cint^t_0 ( \mf{C}_{a u_i c} \cint^t_0 \Psi_{u_1 \hat{d}_i u_l}^{v_1 \dots v_r} ) \\
\notag &\qquad - \sum_{j = 1}^r \gamma^{c v_j} \cint^t_0 ( \mf{C}_{adc} \cint^t_0 \Psi_{u_1 \dots u_l}^{v_1 \hat{d}_j v_r} ) \text{.}
\end{align}
\end{proposition}

Finally, integrating the second identity of \eqref{eq.str_ev_vol}, we have
\[ \mf{L}_t \{ \exp [ - \cint^t_0 ( \trace k ) ] \cdot \epsilon \} \equiv 0 \text{.} \]
This equivariance implies the following identity:
\[ \{ \exp [ - \cint^t_0 ( \trace k ) ] \cdot \epsilon \} [\tau] = \epsilon [0] \text{.} \]
Define the associated \emph{Jacobian} (of $\epsilon$, with respect to $\epsilon [0]$) to be the factor
\[ \mc{J} = \exp \cint^t_0 ( \trace k ) \in \mc{C}^\infty \mc{N} \text{.} \]
$\mc{J}$ acts as a ``change of measure" quantity, as it satisfies
\begin{equation} \label{eq.jacobian} \epsilon [\tau] = \mc{J} [\tau] \cdot \epsilon [0] \text{,} \qquad \nabla_t \mc{J} = \trace k \cdot \mc{J} \text{.} \end{equation}

\subsection{Evolutionary Bounds} \label{sec.cfol_reg}

We had defined various iterated integral norms in Section \ref{sec.fol_hor} for $(\mc{N}, \gamma)$.
Given the covariant evolutionary structures defined in Sections \ref{sec.cfol_ev} and \ref{sec.cfol_cint}, we can now define the following additional norms:
\begin{itemize}
\item We can reverse the order of integration.
Given $p, q \in [1, \infty)$, we define
\begin{align*}
\| \Psi \|_{ L^{q, p}_{x, t} } &= \brak{ \int_{ \mc{S} } \paren{ \int_0^\delta \left. | \Psi |^p \mc{J}^\frac{p}{q} \right|_{ (\tau, x) } d \tau }^\frac{q}{p} d \epsilon [0]_x }^\frac{1}{q} \text{,} \\
\| \Psi \|_{ L^{q, \infty}_{x, t} } &= \brak{ \int_{ \mc{S} } \paren{ \left. \sup_{ 0 \leq \tau \leq \delta } | \Psi | \mc{J}^\frac{1}{q} \right|_{ (\tau, x) } }^q d \epsilon [0]_x }^\frac{1}{q} \text{.}
\end{align*}
Furthermore, when $q = \infty$, we define
\begin{align*}
\| \Psi \|_{ L^{\infty, p}_{x, t} } &= \sup_{ x \in \mc{S} } \paren{ \int_0^\delta \left. | \Psi |^p \right|_{ (\tau, x) } d \tau }^\frac{1}{p} \text{,} \qquad \| \Psi \|_{ L^{\infty, \infty}_{x, t} } = \sup_{ x \in \mc{S} } \sup_{ 0 \leq \tau \leq \delta } | \Psi | |_{ (\tau, x) } \text{.}
\end{align*}

\item We also define the following iterated first-order Sobolev norm
\[ \| \Psi \|_{ N^1_{t, x} } = \| \nabla_t \Psi \|_{ L^{2, 2}_{t, x} } + \| \nabla \Psi \|_{ L^{2, 2}_{t, x} } + \| \Psi \|_{ L^{2, 2}_{t, x} } \text{.} \]
\end{itemize}

In Section \ref{sec.fol}, we derived various estimates for $(\mc{N}, \gamma)$, under the assumption that the $(\mc{S}, \gamma [\tau])$'s were ``uniformly regular" in the sense of the \ass{R0}{}, \ass{R1}{}, and \ass{R2}{} conditions.
Our next task is to \emph{derive} these \ass{R0}{}, \ass{R1}{}, and \ass{R2}{} conditions from relatively weak assumptions on how the geometries of the $\mc{S}_\tau$'s evolve.

These evolutionary conditions are stated as integral bounds on the second fundamental form $k$.
The bounds we will reference throughout the paper are below:
\begin{align}
\label{eqr.sff_tr} \| \trace k \|_{ L^{\infty, 1}_{x, t} } &\leq 2 B \text{,} \\
\label{eqr.sff} \| k \|_{ L^{\infty, 1}_{x, t} } &\leq B \text{,} \\
\label{eqr.sffd_tr} \| \nabla (\trace k) \|_{ L^{2, 1}_{x, t} } &\leq 2 B \text{,} \\
\label{eqr.sffd} \| \nabla k \|_{ L^{2, 1}_{x, t} } &\leq B \text{,} \\
\label{eqr.sffcurl} \inf \{ \| \Phi \|_{ L^{4, \infty}_{x, t} } \mid \Phi \in \mc{C}^\infty \ul{T}^0_3 \mc{N} \text{, } \nabla_t \Phi = \mf{C} \} &\leq B \text{.}
\end{align}
Note \eqref{eqr.sff} trivially implies \eqref{eqr.sff_tr}, and similarly for \eqref{eqr.sffd} and \eqref{eqr.sffd_tr}.
Moreover, \eqref{eqr.sffcurl} states that some covariant $t$-antiderivative of $\mf{C}$ has $L^4_x$-control.

We now discuss some basic consequences of some of the above conditions.
First of all, the bound \eqref{eqr.sff_tr} trivially implies uniform bounds for the Jacobian $\mc{J}$.

\begin{proposition} \label{thm.vol_comp}
If \eqref{eqr.sff_tr} holds, then
\begin{equation} \label{eq.vol_comp} e^{-2 B} \leq \mc{J} \leq e^{2 B} \text{.} \end{equation}
\end{proposition}

Next, we apply \eqref{eq.vol_comp} to derive some basic integrated calculus estimates.

\begin{proposition} \label{thm.int_ineq}
Assume \eqref{eqr.sff_tr}, and fix $q \in [1, \infty]$.
If $\Psi \in \mc{C}^\infty \ul{T}^r_l \mc{N}$, then
\begin{equation} \label{eq.int_ineq} \| \cint^t_0 \Psi \|_{ L^{q, \infty}_{x, t} } \lesssim_B \| \Psi \|_{ L^{q, 1}_{x, t} } \text{.} \end{equation}
\end{proposition}

\begin{proof}
First, assume $q < \infty$.
Applying \eqref{eq.est_trace} along with \eqref{eq.vol_comp}, then
\[ \| \cint^t_0 \Psi \|_{ L^{q, \infty}_{x, t} }^q \leq e^{2 B} \int_{ \mc{S} } \paren{ \int_0^\delta | \Psi | |_{ (\tau, x) } d \tau }^q d \epsilon [0]_x \leq e^{4 B} \| \Psi \|_{ L^{q, 1}_{x, t} }^q \text{.} \]
The remaining case $q = \infty$ is proved similarly:
\[ \| \cint^t_0 \Psi \|_{ L^{\infty, \infty}_{x, t} } \leq \sup_{ x \in \mc{S} } \int_0^\delta | \Psi | |_{ (\tau, x) } d \tau = \| \Psi \|_{ L^{\infty, 1}_{x, t} } \text{.} \qedhere \]
\end{proof}

\begin{corollary} \label{thm.int_ineq_bi}
Fix $p_1, p_2, q \in [1, \infty]$ and $q_1, q_2 \in [q, \infty]$, with
\[ q_1^{-1} + q_2^{-1} = q^{-1} \text{,} \qquad p_1^{-1} + p_2^{-1} = 1 \text{.} \]
Assume that \eqref{eqr.sff_tr} holds.
If $\Psi_i \in \mc{C}^\infty \ul{T}^{r_i}_{l_i} \mc{N}$, where $i \in \{ 1, 2 \}$, then
\begin{equation} \label{eq.int_ineq_bi} \| \cint^t_0 ( \Psi_1 \otimes \Psi_2 ) \|_{ L^{q, \infty}_{x, t} } \lesssim_B \| \Psi_1 \|_{ L^{q_1, p_1}_{x, t} } \| \Psi_2 \|_{ L^{q_2, p_2}_{x, t} } \text{.} \end{equation}
\end{corollary}

\begin{proof}
This follows trivially from \eqref{eq.int_ineq} and H\"older's inequality.
\end{proof}

Finally, we prove a basic estimate for the derivative of the Jacobian.

\begin{proposition} \label{thm.vold_comp}
If \eqref{eqr.sff} holds, and if $q \in [1, \infty]$, then
\begin{equation} \label{eq.vold_comp} \| \nabla \mc{J} \|_{ L^{q, \infty}_{x, t} } \lesssim_B \| \nabla ( \trace k ) \|_{ L^{q, 1}_{x, t} } \text{.} \end{equation}
\end{proposition}

\begin{proof}
We begin by applying \eqref{eq.vol_comp} to obtain
\[ \| \nabla \mc{J} \|_{ L^{q, \infty}_{x, t} } = \| \mc{J} \cdot \nabla \cint^t_0 ( \trace k ) \|_{ L^{q, \infty}_{x, t} } \lesssim \| \nabla \cint^t_0 ( \trace k ) \|_{ L^{q, \infty}_{x, t} } \text{.} \]
For any $(\tau, x) \in \mc{N}$, we have from \eqref{eq.est_trace} and \eqref{eq.comm_inv} that
\[ | \nabla \cint^t_0 ( \trace k ) | |_{ (\tau, x) } \leq \int_0^\delta | \nabla ( \trace k ) | |_{ (w, x) } dw + \int_0^\tau | k | | \nabla \cint^t_0 ( \trace k ) | |_{ (w, x) } dw \text{.} \]
By the Gr\"onwall inequality, then
\[ | \nabla \cint^t_0 ( \trace k ) | |_{ (\tau, x) } \leq ( \exp \| k \|_{ L^{\infty, 1}_{x, t} } ) \int_0^\delta | \nabla ( \trace k ) | |_{ (w, x) } dw \text{.} \]
Taking an $L^{q, \infty}_{x, t}$-norm of the above, we obtain
\[ \| \nabla \cint^t_0 ( \trace k ) \|_{ L^{q, \infty}_{x, t} } \lesssim \| \nabla ( \trace k ) \|_{ L^{q, 1}_{x, t} } \text{.} \]
Combining the above completes the proof of \eqref{eq.vold_comp}.
\end{proof}

\subsection{Frame Estimates} \label{sec.est_frame}

The next task is to apply our evolutionary assumptions to derive estimates for equivariant and $t$-parallel fields.
Later, we will apply these estimates in order to control families of equivariant and $t$-parallel horizontal frames.
Throughout, we let $U$ denote an open subset of $\mc{S}$.

We begin first with equivariant fields.
These are essentially the estimates resulting from the ``weak regularity" condition used in \cite{kl_rod:cg, kl_rod:stt, parl:bdc, shao:bdc_nv, wang:cg, wang:cgp}.

\begin{proposition} \label{thm.equiv_est_0}
If $Z \in \mc{C}^\infty \ul{T}^r_l \mc{N}_U$ is equivariant, and if \eqref{eqr.sff} holds, then
\begin{equation} \label{eq.equiv_est_0} \| Z \|_{ L^{\infty, \infty}_{t, x} } \leq e^{ (r + l) B } \| Z [0] \|_{ L^\infty_x } \text{.} \end{equation}
\end{proposition}

\begin{proof}
Since $\mf{L}_t Z \equiv 0$, then $\nabla_t Z$ is the sum of $r + l$ terms, each of the form $k \otimes Z$, with one $\gamma$-contraction.
Thus, by \eqref{eq.est_trace}, if $(\tau, x) \in \mc{N}$, then
\begin{align*}
| Z | |_{ (\tau, x) } &\leq | Z | |_{ (0, x) } + \cint^t_0 | \nabla_t Z | |_{ (\tau, x) } \\
&\leq | Z | |_{ (0, x) } + ( r + l ) \int_0^\tau | k | | Z | |_{ (w, x) } dw \text{.}
\end{align*}
Applying Gr\"onwall's inequality to the above yields
\[ | Z |_{ (\tau, x) } \leq | Z | |_{ (0, x) } \exp [ ( r + l ) \| k \|_{ L^{\infty, 1}_{x, t} } ] \leq e^{(r + l) B} | Z | |_{ (0, x) } \text{.} \]
The estimate \eqref{eq.equiv_est_0} follows immediately from the above.
\end{proof}

We also obtain an estimate for the covariant derivative of a equivariant field.
Just as the zero-order estimate \eqref{eq.equiv_est_0} required some control for $k$, a corresponding first-order estimate will require analogous control for the derivative of $k$.

\begin{proposition} \label{thm.equiv_est_1}
If $Z \in \mc{C}^\infty \ul{T}^r_l \mc{N}_U$ is equivariant, and if \eqref{eqr.sff} holds, then for any $q \in [1, \infty]$, we have the following estimate:
\begin{equation} \label{eq.equiv_est_1} \| \nabla Z \|_{ L^{q, \infty}_{x, t} } \lesssim_{B, r, l} \| \nabla Z [0] \|_{ L^q_x } + \| Z [0] \|_{ L^\infty_x } \| \nabla k \|_{ L^{q, 1}_{x, t} } \text{.} \end{equation}
\end{proposition}

\begin{proof}
By the commutation formula \eqref{eq.comm_cov}, we have the estimate
\[ | \nabla_t \nabla Z | \leq | \nabla \nabla_t Z | + | k | | \nabla Z | + (r + l) | \mf{C} | | Z | \text{.} \]
Recalling the form of $\nabla_t Z$ in the proof of Proposition \ref{thm.equiv_est_0}, and noting the crude estimate $| \mf{C} | \leq 2 | \nabla k |$, then we have the bound
\[ | \nabla_t \nabla Z | \leq (r + l + 1) | k | | \nabla Z | + 3 (r + l) | Z | | \nabla k | \text{.} \]

As a result of the above and of \eqref{eq.est_trace}, if $(\tau, x) \in \mc{N}$, then
\begin{align*}
| \nabla Z | |_{ (\tau, x) } &\leq | \nabla Z | |_{ (0, x) } + 3 (r + l) \| Z \|_{ L^{\infty, \infty}_{t, x} } \int_0^\tau | \nabla k | |_{ (w, x) } dw \\
&\qquad + ( r + l + 1 ) \int_0^\tau | k | | \nabla Z | |_{ (w, \omega) } dw \text{.}
\end{align*}
Applying \eqref{eq.equiv_est_0} and Gronwall's inequality yields
\[ \sup_{0 \leq \tau \leq \delta} | \nabla Z | |_{ (\tau, x) } \lesssim | \nabla Z | |_{ (0, x) } + \| Z [0] \|_{ L^\infty_x } \int_0^\delta | \nabla k |_{ (w, x) } dw \text{.} \]
Taking an $L^q$-norm of the above over $\mc{S}$ and recalling \eqref{eq.vol_comp} completes the proof.
\end{proof}

\begin{remark}
We can only expect $L^{2, 2}_{t, x}$-type bounds for $\nabla k$ in our motivational problem of regular null cones with bounded curvature flux.
Thus, by applying Proposition \ref{thm.equiv_est_1}, we can only obtain $L^{2, \infty}_{x, t}$-type estimates for $\nabla Z$.
\end{remark}

We now present estimates analogous to Propositions \ref{thm.equiv_est_0} and \ref{thm.equiv_est_1}, but for $t$-parallel fields.
First of all, if $\Psi \in \mc{C}^\infty \ul{T}^r_l \mc{N}_U$ is $t$-parallel, then $\nabla_t | \Psi |^2 \equiv 0$, and hence $| \Psi |$ is constant with respect to $t$.
This implies the following identity.

\begin{proposition} \label{thm.par_est_0}
If $\Psi \in \mc{C}^\infty \ul{T}^r_l \mc{N}_U$ is $t$-parallel, then
\begin{equation} \label{eq.par_est_0} \| \Psi \|_{ L^{\infty, \infty}_{t, x} } = \| \Psi [0] \|_{ L^\infty_x } \text{.} \end{equation}
\end{proposition}

On the other hand, since $\nabla \Psi$ is no longer $t$-parallel, the above is no longer true for covariant derivatives of $\Psi$.
However, given some control for $k$, we can still provide some control for $\nabla \Psi$.
One example is the following proposition.

\begin{proposition} \label{thm.par_est_1}
Suppose $\Psi \in \mc{C}^\infty \ul{T}^r_l \mc{N}_U$ is $t$-parallel, and assume \eqref{eqr.sff} holds.
If $q \in [1, \infty]$, and if $\Phi \in \mc{C}^\infty \ul{T}^0_3 \mc{N}$ is such that $\nabla_t \Phi = \mf{C}$, then
\begin{equation} \label{eq.par_est_1} \| \nabla \Psi \|_{ L^{q, \infty}_{x, t} } \lesssim_B \| \nabla \Psi [0] \|_{ L^q_x } + (r + l) \| \Psi [0] \|_{ L^\infty_x } \| \Phi \|_{ L^{q, \infty}_{x, t} } \text{.} \end{equation}
\end{proposition}

\begin{proof}
Like in the proof of Proposition \ref{thm.equiv_est_1}, we apply \eqref{eq.comm_cov} and \eqref{eq.est_trace} to obtain
\[ | \nabla \Psi | |_{ (\tau, x) } \lesssim | \nabla \Psi | |_{ (0, x) } + | \cint^t_0 ( k \otimes \nabla \Psi ) | |_{ (\tau, x) } + (r + l) | \cint^t_0 ( \mf{C} \otimes \Psi ) | |_{ (\tau, x) } \]
for any $(\tau, x) \in \mc{N}$.
Since
\[ \cint^t_0 ( \mf{C} \otimes \Psi ) = \cint^t_0 ( \nabla_t \Phi \otimes \Psi ) = \Phi \otimes \Psi - \mf{p} \{ ( \Phi \otimes \Psi ) [0] \} \text{,} \]
where we integrated by parts and recalled that $\Psi$ is $t$-parallel, then
\begin{align*}
| \nabla \Psi |_{ (\tau, x) } &\lesssim | \nabla \Psi |_{ (0, x) } + \int_0^\tau | k | | \nabla \Psi | |_{ (w, x) } dw + (r + l) \| \Psi \|_{ L^{\infty, \infty}_{t, x} } \sup_{ 0 \leq w \leq \delta } | \Phi |_{ (w, x) } \\
&\lesssim | \nabla \Psi |_{ (0, x) } + (r + l) \| \Psi [0] \|_{ L^\infty_x } \sup_{ 0 \leq w \leq \delta } | \Phi |_{ (w, x) } + \int_0^\tau | k | | \nabla \Psi | |_{ (w, x) } dw \text{.}
\end{align*}
Applying Gr\"onwall's inequality to the above yields
\[ \sup_{ 0 \leq \tau \leq \delta } | \nabla \Psi |_{ (\tau, x) } \lesssim | \nabla \Psi |_{ (0, x) } + (r + l) \| \Psi [0] \|_{ L^\infty_x } \sup_{ 0 \leq \tau \leq \delta } | \Phi |_{ (\tau, x) } \text{.} \]
Taking an $L^q$-norm of this over $\mc{S}$ and recalling \eqref{eq.vol_comp} proves \eqref{eq.par_est_1}.
\end{proof}

In summary, the estimate \eqref{eq.equiv_est_1} depends on the regularity of $\nabla k$, while \eqref{eq.par_est_1} depends only on $\mf{C}$.
Thus, if $\mf{C}$ has better bounds than $\nabla k$, then Proposition \ref{thm.par_est_1} may yield strictly better control than Proposition \ref{thm.equiv_est_1}.

\begin{remark}
In our motivational problem of regular null cones with bounded curvature flux, we can find $\Phi$ satisfying the hypotheses of Proposition \ref{thm.par_est_1} which has good $L^{4, \infty}_{x, t}$-control.
Thus, Proposition \ref{thm.par_est_1} yields strictly better estimates for $t$-parallel frames than Proposition \ref{thm.equiv_est_1} does for equivariant frames.
\end{remark}

\subsection{Propagation of Regularity} \label{sec.est_prop}

In Section \ref{sec.geom_reg}, we described regularity conditions (\ass{r0}{}, \ass{r1}{}, \ass{r2}{}) on a fixed surface $(\mc{S}, h)$.
Next, in Section \ref{sec.fol_reg}, we described their analogues (\ass{R0}{}, \ass{R1}{}, and \ass{R2}{}) on $(\mc{N}, \gamma)$.
These roughly correspond to the fixed surface regularity conditions applying uniformly to every $(\mc{S}, \gamma [\tau])$.

Here, we add to this discussion the notions of covariant evolution introduced in this section.
We will show that if we assume the integral bounds \eqref{eqr.sff_tr}-\eqref{eqr.sffcurl} on our foliation $(\mc{N}, \gamma)$, then any regularity conditions that hold on the initial surface $(\mc{S}, \gamma [0])$ can be propagated to every $(\mc{S}, \gamma [\tau])$ in a very explicit manner.

For convenience, we first define the following regularity conditions on $(\mc{N}, \gamma)$.
\begin{itemize}
\item $(\mc{N}, \gamma)$ satisfies \ass{F0}{C, N, B}, with data $\{ U_i, \varphi_i, \eta_i \}_{i = 1}^N$, iff $(\mc{S}, \gamma [0])$ satisfies \ass{r0}{C, N}, with the same data, and \eqref{eqr.sff} holds.

\item $(\mc{N}, \gamma)$ satisfies \ass{F1}{C, N, B}, with data $\{ U_i, \varphi_i, \eta_i, \tilde{\eta}_i, e^i \}_{i = 1}^N$, iff $(\mc{S}, \gamma [0])$ satisfies \ass{r1}{C, N}, with the same data, and \eqref{eqr.sff}, \eqref{eqr.sffd_tr}, and \eqref{eqr.sffcurl} hold.

\item $(\mc{N}, \gamma)$ satisfies \ass{F2}{C, N, B}, with data $\{ U_i, \varphi_i, \eta_i, \tilde{\eta}_i, e^i \}_{i = 1}^N$, iff $(\mc{S}, \gamma [0])$ satisfies \ass{r2}{C, N}, with the same data, and \eqref{eqr.sff}, \eqref{eqr.sffd}, and \eqref{eqr.sffcurl} hold.
\end{itemize}
Note the \ass{F1}{} condition trivially implies the \ass{F0}{} condition, and similarly for the \ass{F2}{} and \ass{F1}{} conditions.
We can also show that if the \ass{F0}{}, \ass{F1}{}, or \ass{F2}{} condition holds, then so does the \ass{R0}{}, \ass{R1}{}, or \ass{R2}{} condition, respectively, with explicitly constructed data.
This result can be stated precisely as follows.

\begin{proposition} \label{thm.reg_prop}
Given an open $U \subseteq \mc{S}$ and $e = (e_1, e_2) \in \mc{C}^\infty T^1_0 U \times \mc{C}^\infty T^1_0 U$, we define $\mf{p} e$ to be the $t$-parallel transports of the components of $e$:
\[ \mf{p} e = ( \mf{p} e_1, \mf{p} e_2 ) \in \mc{C}^\infty \ul{T}^1_0 \mc{N}_U \times \mc{C}^\infty \ul{T}^1_0 \mc{N}_U \text{.} \]
The following statements hold:
\begin{itemize}
\item Suppose $(\mc{N}, \gamma)$ satisfies \ass{F0}{C, N, B}, with data $\{ U_i, \varphi_i, \eta_i \}_{i = 1}^N$.
Then, $(\mc{N}, \gamma)$ also satisfies \ass{R0}{C^\prime, N}, with the same data $\{ U_i, \varphi_i, \eta_i \}_{i = 1}^N$, for some constant $C^\prime$ depending on $C$, $N$, and $B$.

\item Suppose $(\mc{N}, \gamma)$ satisfies \ass{F1}{C, N, B}, with data $\{ U_i, \varphi_i, \eta_i, \tilde{\eta}_i, e^i \}_{i = 1}^N$.
Then, $(\mc{N}, \gamma)$ also satisfies \ass{R1}{C^\prime, N}, with data $\{ U_i, \varphi_i, \eta_i, \tilde{\eta}_i, \mf{p} e^i \}_{i = 1}^N$, for some constant $C^\prime$ depending on $C$, $N$, and $B$.

\item Suppose $(\mc{N}, \gamma)$ satisfies \ass{F2}{C, N, B}, with data $\{ U_i, \varphi_i, \eta_i, \tilde{\eta}_i, e^i \}_{i = 1}^N$.
Then, $(\mc{N}, \gamma)$ also satisfies \ass{R2}{C^\prime, N}, with data $\{ U_i, \varphi_i, \eta_i, \tilde{\eta}_i, \mf{p} e^i \}_{i = 1}^N$, for some constant $C^\prime$ depending on $C$, $N$, and $B$.
\end{itemize}
\end{proposition}

\begin{proof}
Throughout, we adopt the same conventions as in Section \ref{sec.fol}---we denote the equivariant transports of $\varphi_i$, $\eta_i$, and $\tilde{\eta}_i$ by $\varphi_i$, $\eta_i$, and $\tilde{\eta}_i$, respectively.

We begin with the first statement (assuming \ass{F0}{C, N, B} holds).
By definition,
\[ | \mc{S} |_\tau = \int_{ \mc{S} } \mc{J} |_{ (\tau, x) } d \epsilon [0]_x \text{.} \]
Since $(\mc{S}, \gamma [0])$ satisfies \ass{r0}{C, N}, then \eqref{eq.vol_comp} yields that $| \mc{S} |_\tau \simeq 1$.
Next, letting $X \in \mc{C}^\infty \ul{T}^1_0 \mc{N}$ be equivariant, and letting $x \in U_i$, with $1 \leq i \leq N$, we have
\[ | X |^2 |_{ (\tau, x) } = | X |^2 |_{ (0, x) } + 2 \int_0^\tau k (X, X) |_{ (w, x) } dw \text{,} \]
and hence
\[ \abs{ | X |^2 |_{ (\tau, x) } - | X |^2_{ (0, x) } } \leq \int_0^\tau | k | |_{ (w, x) } | X |^2 |_{ (w, x) } dw \text{.} \]
Applying Gr\"onwall's inequality and \eqref{eqr.sff} yields
\[ | X |^2 |_{ (\tau, x) } \simeq | X |^2 |_{ (0, x) } \text{.} \]
By the \ass{r0}{C, N} condition for $(\mc{S}, \gamma [0])$, we see that $\gamma [\tau]$ also satisfies the uniform positivity property in the \ass{r0}{} condition (with some constant depending on $C$, $N$, and $B$).
It follows that $(\mc{N}, \gamma)$ satisfies \ass{R0}{C^\prime, N} for some appropriate $C^\prime$.

Next, suppose \ass{F1}{C, N, B} holds.
Applying \eqref{eqr.sffcurl}, \eqref{eq.par_est_1}, and \ass{r1}{C, N}, we obtain
\[ \| \nabla ( \mf{p} e^i_a ) \|_{ L^{4, \infty}_{x, t} } \lesssim 1 \text{,} \qquad 1 \leq i \leq N \text{,} \quad a \in \{ 1, 2 \} \text{,} \]
since $\mf{p} e^i_a$ is by definition $t$-parallel.
Now, let $\vartheta_i \in \mc{C}^\infty \mc{N}_{U_i}$ denote the area density for the $\varphi_i$-coordinates (see Section \ref{sec.fol_reg}).
Since $\vartheta_i [\tau] = \mc{J} [\tau] \cdot \vartheta_i [0]$ by definition, then
by \ass{r1}{C, N}, \eqref{eqr.sffd_tr}, \eqref{eq.vol_comp}, and \eqref{eq.vold_comp} we have
\[ \| \nabla \vartheta_i \|_{ L^{2, \infty}_{x, t} } \lesssim 1 \text{.} \]
From this, it follows that $(\mc{N}, \gamma)$ satisfies \ass{R1}{C^\prime, N} for some appropriate $C^\prime$, with the data given in the statement of the proposition.

Finally, suppose \ass{F2}{C, N, B} holds.
Since $\mf{e} \partial^i_a$ is equivariant by definition, then
\[ \| \nabla ( \mf{e} \partial^i_a ) \|_{ L^{2, \infty}_{x, t} } \lesssim 1 \text{,} \]
where we applied \eqref{eqr.sffd}, \ass{r2}{C, N}, and \eqref{eq.equiv_est_1}.
As a result, $(\mc{N}, \gamma)$ satisfies \ass{R2}{C^\prime, N}, for the appropriate $C^\prime$ and with the specified data, as desired.
\end{proof}

\subsection{Foliation Estimates} \label{sec.cfol_sob}

We conclude this section by establishing some tensorial Sobolev estimates on all of $\mc{N}$.
These were also proved in \cite[Lemma 3.2]{wang:cg}.
For completeness, we repeat some of the proofs found there.

\begin{proposition} \label{thm.nsob_ineq_pre}
Assume \eqref{eqr.sff_tr}, and let $\Psi \in \mc{C}^\infty \ul{T}^r_l \mc{N}$.
Then,
\begin{equation} \label{eq.nsob_ineq_pre} \| \Psi \|_{ L^{2, \infty}_{x, t} } \lesssim_B \| \Psi [0] \|_{ L^2_x } + \| \nabla_t \Psi \|_{ L^{2, 2}_{t, x} }^\frac{1}{2} \| \Psi \|_{ L^{2, 2}_{t, x} }^\frac{1}{2} \text{.} \end{equation}
\end{proposition}

\begin{proof}
For any $x \in \mc{S}$, we have
\begin{align*}
| \Psi |^2 |_{ (\tau, x) } &= | \Psi |^2 |_{ (0, x) } + \int_0^\tau \nabla_t | \Psi |^2 |_{ (w, x) } dw \\
&\leq | \Psi |^2 |_{ (0, x) } + \int_0^\tau | \nabla_t \Psi | | \Psi | |_{ (w, x) } dw \text{.}
\end{align*}
Taking the supremum over all $\tau \in [0, \delta]$ yields
\[ \sup_{ 0 \leq \tau \leq \delta } | \Psi |^2 |_{ (\tau, x) } \leq | \Psi |^2 |_{ (0, x) } + \int_0^\delta | \nabla_t \Psi | | \Psi | |_{ (w, x) } dw \text{.} \]
Integrating over $\mc{S}$ and applying H\"older's inequality yields \eqref{eq.nsob_ineq_pre}.
\end{proof}

\begin{proposition} \label{thm.nsob_ineq}
Assume $(\mc{N}, \gamma)$ satisfies \ass{F0}{C, N, B}.
If $\Psi \in \mc{C}^\infty \ul{T}^r_l \mc{N}$, then
\begin{equation} \label{eq.nsob_ineq} \| \Psi \|_{ L^{4, \infty}_{x, t} } \lesssim_{C, N, B} \| \Psi [0] \|_{ L^4_x } + \| \nabla_t \Psi \|_{ L^{2, 2}_{t, x} }^\frac{1}{2} ( \| \nabla \Psi \|_{ L^{2, 2}_{t, x} } + \| \Psi \|_{ L^{2, 2}_{t, x} } )^\frac{1}{2} \text{.} \end{equation}
\end{proposition}

\begin{proof}
First, by Proposition \ref{thm.reg_prop}, the \ass{r0}{C^\prime, N} condition holds on every $(\mc{S}, \gamma [\tau])$ for some constant $C^\prime$ depending on $C$, $N$, and $B$.

Repeating roughly the first part of the proof of \eqref{eq.nsob_ineq_pre}, we obtain for any $x \in \mc{S}$,
\[ \sup_{ 0 \leq \tau \leq \delta } | \Psi |^4 |_{ (\tau, x) } \lesssim | \Psi |^4 |_{ (0, x) } + \int_0^\delta | \nabla_t \Psi | | \Psi |^3 |_{ (w, x) } \text{.} \]
Integrating the above over $\mc{S}$ and applying H\"older's inequality yields
\[ \| \Psi \|_{ L^{4, \infty}_{x, t} } \lesssim \| \Psi [0] \|_{ L^4_x } + \| \nabla_t \Psi \|_{ L^{2, 2}_{t, x} }^\frac{1}{4} \| \Psi \|_{ L^{6, 6}_{t, x} }^\frac{3}{4} \text{.} \]
The inequality \eqref{eq.gns_1_pre} yields for any $\tau$ that
\begin{align*}
\| \Psi [\tau] \|_{ L^6_x }^6 &\lesssim \{ \| \nabla ( | \Psi |^3 ) [\tau] \|_{ L^1_x } + \| | \Psi |^3 [\tau] \|_{ L^1_x } \}^2 \\
&\lesssim ( \| \nabla \Psi [\tau] \|_{ L^2_x } + \| \Psi [\tau] \|_{ L^2_x } )^2 \| \Psi [\tau] \|_{ L^4_x }^4 \text{.}
\end{align*}
Integrating the above over the time interval $[0, \delta]$, we obtain
\[ \| \Psi \|_{ L^{6, 6}_{t, x} }^\frac{3}{4} \lesssim ( \| \nabla \Psi \|_{ L^{2, 2}_{t, x} } + \| \Psi \|_{ L^{2, 2}_{t, x} } )^\frac{1}{4} \| \Psi \|_{ L^{4, \infty}_{x, t} }^\frac{1}{2} \text{.} \]
Combining all the above yields \eqref{eq.nsob_ineq}.
\end{proof}

\section{Parallel Scalar Reductions} \label{sec.thm}

Assume the foliation $(\mc{N}, \gamma)$ and all definitions and notations associated with it, as in Sections \ref{sec.fol} and \ref{sec.cfol}.
In Section \ref{sec.cfol}, we showed that the bounds \eqref{eqr.sff_tr}-\eqref{eqr.sffcurl} suffice to propagate the regularity conditions of Section \ref{sec.geom_reg} from $\mc{S}_0$ uniformly to all the $\mc{S}_\tau$'s.
This was shown explicitly in Proposition \ref{thm.reg_prop}.

In this section, we revisit the scalar reduction scheme of Section \ref{sec.fol_scal}, but with the specific data constructed from Proposition \ref{thm.reg_prop}.
We will then use this construction to give simple and concise proofs of the main bilinear product estimates of this paper.
The basic strategy is similar to that of Section \ref{sec.fol_est}; we reduce the geometric tensorial estimates to their corresponding estimates in Euclidean space, which are proved in Appendix \ref{sec.eucl}.
The main new feature of the constructions in this section is that the local orthonormal frames in the data for the \ass{R1}{} condition here are $t$-parallel.
As a result, these frames behave very well with respect to the covariant operators $\nabla_t$ and $\cint^t_0$ that will be present in our main estimates.

\subsection{Scalar Reductions} \label{sec.thm_scal}

Suppose for the moment that $(\mc{N}, \gamma)$ satisfies \ass{F1}{C, N, B}, with data $\{ U_i, \varphi_i, \eta_i, \tilde{\eta}_i, e^i \}_{i = 1}^N$.
Proposition \ref{thm.reg_prop} then implies that $(\mc{N}, \gamma)$ satisfies \ass{R1}{C^\prime, N}, with data $\{ U_i, \varphi_i, \eta_i, \tilde{\eta}_i, \mf{p} e^i \}_{i = 1}^N$, where $C^\prime$ depends on $C$, $N$, and $B$.
For convenience, we will call the above data for the \ass{R1}{C^\prime, N} condition the \emph{parallel data} induced by the data for the \ass{F1}{C, N, B} condition.

We can now consider the scalar reduction scheme of Section \ref{sec.fol_scal}, with this parallel data.
In particular, any element $X \in i \mf{X}^r_l$, where $1 \leq i \leq N$, is $t$-parallel.

\begin{proposition} \label{thm.scalar_red_H}
Let $\Psi \in \mc{C}^\infty \ul{T}^r_l \mc{N}$, and assume $(\mc{N}, \gamma)$ satisfies \ass{F1}{C, N, B}, with data $\{ U_i, \varphi_i, \eta_i, \tilde{\eta}_i, e^i \}_{i = 1}^N$.
Then, with respect to the induced parallel data,
\begin{equation} \label{eq.scalar_red_H} \sum_{ X \in i \mc{X}^r_l } \| \Psi (X) \|_{ N^1_{t, x} } \lesssim_{C, N, B, r, l} \| \Psi \|_{ N^1_{t, x} } \text{,} \qquad 1 \leq i \leq N \text{.} \end{equation}
\end{proposition}

\begin{proof}
This follows from \eqref{eq.scalar_red_D} and the fact that each $X \in \mf{p} \mc{X}^r_l (i)$ is $t$-parallel.
\end{proof}

First of all, we can apply this \emph{parallel} scalar reduction scheme to compare two geometric Besov norms, \emph{with respect to different metrics}.
One useful statement of this property is given in the subsequent proposition.

\begin{proposition} \label{thm.norm_comp_besov}
Assume $(\mc{N}, \gamma)$ satisfies \ass{F1}{C, N, B}, and suppose $a \in [1, \infty]$ and $s \in (-1, 1)$.
If $\Psi \in \mc{C}^\infty \ul{T}^r_l \mc{N}$ is $t$-parallel, then
\begin{align}
\label{eq.norm_comp_besov} \| \Psi \|_{ B^{a, \infty, s}_{\ell, t, x} } &\lesssim_{ C, N, B, s, r, l } \| \Psi [0] \|_{ B^{a, s}_{\ell, x} } \text{.}
\end{align}
\end{proposition}

\begin{proof}
Applying \eqref{eq.comp_main}, we have
\[ \| \Psi \|_{ B^{a, \infty, s}_{\ell, t, x} } \lesssim \sum_{i = 1}^N \sum_{X \in i \mc{X}^r_l} \| \eta_i \Psi (X) \circ \varphi_i^{-1} \|_{ B^{a, \infty, s}_{\ell, t, x} } \text{.} \]
Moreover, for each $1 \leq i \leq N$ and $X \in i \mc{X}^r_l$, the localized quantity $\eta_i \Psi (X)$ is $t$-parallel, that is, it is independent of the $t$-variable.
As a result, by \eqref{eq.comp_mainf},
\[ \| \Psi \|_{ B^{a, \infty, s}_{\ell, t, x} } \lesssim \sum_{i = 1}^N \sum_{X \in i \mc{X}^r_l} \| \eta_i ( \Psi [0] ) ( X [0] ) \circ \varphi_i^{-1} \|_{ B^{a, s}_{\ell, x} } \lesssim \| \Psi [0] \|_{ B^{a, s}_{\ell, x} } \text{.} \qedhere \]
\end{proof}

Moreover, by taking advantage also of our parallel scalar reduction scheme, we can prove a fractional Sobolev and Besov variant of Proposition \ref{thm.nsob_ineq}.

\begin{proposition} \label{thm.nsob_trace}
Assume $(\mc{N}, \gamma)$ satisfies \ass{F1}{C, N, B}.
If $\Psi \in \mc{C}^\infty \ul{T}^r_l \mc{N}$, then
\begin{equation} \label{eq.nsob_trace} \| \Psi \|_{ B^{2, \infty, 1/2}_{\ell, t, x} } \lesssim_{C, N, B, r, l} \| \Psi [0] \|_{ H^{1/2}_x } + \| \nabla_t \Psi \|_{ L^{2, 2}_{t, x} }^\frac{1}{2} ( \| \nabla \Psi \|_{ L^{2, 2}_{t, x} } + \| \Psi \|_{ L^{2, 2}_{t, x} } )^\frac{1}{2} \text{.} \end{equation}
\end{proposition}

\begin{proof}
Let $\{ U_i, \varphi_i, \eta_i, \tilde{\eta}_i, \mf{p} e^i \}_{i = 1}^N$ denote the parallel data induced from the \ass{F1}{} condition.
By Proposition \ref{thm.comp_main}, it suffices to show
\[ \| \Psi \|_{ \mc{B}^{2, \infty, 1/2}_{\ell, t, x} } \lesssim \| \Psi [0] \|_{ H^{1/2}_x } + \| \nabla_t \Psi \|_{ L^{2, 2}_{t, x} }^\frac{1}{2} ( \| \nabla \Psi \|_{ L^{2, 2}_{t, x} } + \| \Psi \|_{ L^{2, 2}_{t, x} } )^\frac{1}{2} \text{.} \]
The first, and main, step is to apply \eqref{eqc.nsob_trace}:
\begin{align*}
\| \Psi \|_{ \mc{B}^{2, \infty, 1/2}_{\ell, t, x} } &\lesssim \sum_{i = 1}^N \sum_{ X \in i \mc{X}^r_l } \| \eta_i \Psi (X) \circ \varphi_i^{-1} \|_{ B^{2, \infty, 1/2}_{k, t, x} } \\
&\lesssim \sum_{i = 1}^N \sum_{ X \in i \mc{X}^r_l } \| \eta_i \Psi (X) \circ \varphi_i^{-1} [0] \|_{ B^{2, 1/2}_{k, x} } \\
&\qquad + \sum_{i = 1}^N \sum_{ X \in i \mc{X}^r_l } \| \partial_t [ \eta_i \Psi (X) \circ \varphi_i^{-1} ] \|_{ L^{2, 2}_{t, x} }^\frac{1}{2} \| \partial [ \eta_i \Psi (X) \circ \varphi_i^{-1} ] \|_{ L^{2, 2}_{t, x} }^\frac{1}{2} \\
&\qquad + \sum_{i = 1}^N \sum_{ X \in i \mc{X}^r_l } \| \partial_t [ \eta_i \Psi (X) \circ \varphi_i^{-1} ] \|_{ L^{2, 2}_{t, x} }^\frac{1}{2} \| \eta_i \Psi (X) \circ \varphi_i^{-1} \|_{ L^{2, 2}_{t, x} }^\frac{1}{2} \\
&= I_1 + I_2 + I_3 \text{.}
\end{align*}
Applying \eqref{eq.comp_mainf} and Proposition \ref{thm.sobolev}, we have
\[ I_1 \lesssim \| \Psi [0] \|_{ B^{2, 1/2}_{\ell, x} } \simeq \| \Psi [0] \|_{ H^{1/2}_x } \text{.} \]
Furthermore, by \ass{F1}{C, N, B}, Proposition \ref{thm.reg_prop}, \eqref{eqr.change_of_coord}, and \eqref{eq.scalar_red_D}, we can also bound
\[ I_2 + I_3 \lesssim \| \nabla_t \Psi \|_{ L^{2, 2}_{t, x} }^\frac{1}{2} ( \| \nabla \Psi \|_{ L^{2, 2}_{t, x} } + \| \Psi \|_{ L^{2, 2}_{t, x} } )^\frac{1}{2} \text{.} \]
Note in particular we used that each $X \in i \mc{X}^r_l$ is $t$-parallel.
\end{proof}

Finally, analogous constructions can be made using equivariant coordinate vector fields.
\footnote{In fact, scalar reductions in \cite{kl_rod:stt, wang:cg} were made in this fashion.}
We will briefly require this in Section \ref{sec.thm_stt}.
Assume that $(\mc{N}, \gamma)$ satisfies \ass{F2}{C, N, B}, with data $\{ U_i, \varphi_i, \eta_i, \tilde{\eta}_i, e^i \}_{i = 1}^N$.
Given $1 \leq i \leq N$, we define the family
\[ i \mc{Z} = \{ \partial^i_a = \mf{e} \partial^i_a \in \mc{C}^\infty \ul{T}^1_0 \mc{N}_{U_i} \mid a \in \{ 1, 2 \} \} \text{,} \]
where the $\partial^i_a$'s are the equivariant transports of the coordinate vector fields for the $\varphi_i$-coordinates (see Section \ref{sec.fol_reg}).
Moreover, from the \ass{F2}{C, N, B} condition and Proposition \ref{thm.reg_prop}, we have for any $Z \in i \mc{Z}$ the estimates
\begin{equation} \label{eq.cframe_basis_prop} \| Z \|_{ L^{\infty, \infty}_{t, x} } \lesssim_{C, B} 1 \text{,} \qquad \| \nabla Z \|_{ L^{2, \infty}_{x, t} } \lesssim_{C, B} 1 \text{.} \end{equation}

\subsection{Non-Integrated Product Estimates} \label{sec.thm_nint}

We now have the requisite tools to prove our main bilinear product estimates.
We begin with the non-integrated product estimates, which do not involve the integral operator $\cint^t_0$.

\begin{theorem} \label{thm.est_prod}
Assume that $(\mc{N}, \gamma)$ satisfies \ass{F1}{C, N}.
Furthermore, consider horizontal tensor fields $\Psi \in \mc{C}^\infty \ul{T}^{r_1}_{l_1} \mc{N}$ and $\Phi \in \mc{C}^\infty \ul{T}^{r_2}_{l_2} \mc{N}$.
\begin{itemize}
\item If $a \in [1, \infty]$, $s \in (-1, 1)$, and $\Psi$ is $t$-parallel, then
\begin{align}
\label{eq.est_prod_imp} \| \Phi \otimes \Psi \|_{ B^{a, 2, s}_{\ell, t, x} } &\lesssim_{ C, N, B, s, r_1, l_1, r_2, l_2 } ( \| \nabla \Phi \|_{ L^{2, 2}_{t, x} } + \| \Phi \|_{ L^{\infty, 2}_{x, t} } ) \| \Psi [0] \|_{ B^{a, s}_{\ell, x} } \text{.}
\end{align}

\item In addition, the following estimate holds:
\begin{align}
\label{eq.est_prod_ex} \| \Phi \otimes \Psi \|_{ B^{1, \infty, 0}_{\ell, t, x} } &\lesssim_{ C, N, B, r_1, l_1, r_2, l_2 } ( \| \Phi \|_{ N^1_{t, x} } + \| \Phi [0] \|_{ H^{1/2}_x } ) \\
\notag &\qquad\qquad\qquad\qquad \cdot ( \| \Psi \|_{ N^1_{t, x} } + \| \Psi [0] \|_{ H^{1/2}_x } ) \text{.}
\end{align}
\end{itemize}
\end{theorem}

\begin{proof}
Let $\{ U_i, \varphi_i, \eta_i, \tilde{\eta}_i, \mf{p} e_i \}_{i = 1}^N$ be the parallel data induced from the \ass{F1}{C, N} condition.
Note that by Proposition \ref{thm.comp_main}, in both \eqref{eq.est_prod_imp} and \eqref{eq.est_prod_ex}, we can replace the geometric norms by their coordinate-based analogues.

For \eqref{eq.est_prod_imp}, we employ, as usual, a scalar decomposition:
\[ \| \Phi \otimes \Psi \|_{ \mc{B}^{a, 2, s}_{\ell, t, x} } = \sum_{i = 1}^N \sum_{ X \in i \mc{X}^{r_1}_{l_1} } \sum_{ Y \in i \mc{X}^{r_2}_{l_2} } \| [ \tilde{\eta}_i \Phi(X) \cdot \eta_i \Psi (Y) ] \circ \varphi_i^{-1} \|_{ B^{a, 2, s}_{\ell, t, x} } \text{.} \]
Since $\Psi$ is $t$-parallel, then $\eta_i \Psi (Y)$ is $t$-parallel for each $1 \leq i \leq N$ and $Y \in i \mc{X}^{r_2}_{l_2}$.
Thus, we can apply \eqref{eqc.est_prod_imp} to the right-hand side above:
\begin{align*}
\| \Phi \otimes \Psi \|_{ \mc{B}^{a, 2, s}_{\ell, t, x} } &\lesssim \sup_{ 1 \leq i \leq N } \sum_{ X \in \mc{X}^{r_1}_{l_1} (i) } \{ \| \nabla [ \Phi (X) ] \|_{ L^{2, 2}_{t, x} } + \| \Phi (X) \|_{ L^{\infty, 2}_{x, t} } \} \\
&\qquad \cdot \sum_{ i = 1 }^N \sum_{ Y \in \mc{X}^{r_2}_{l_2} (i) } \| \eta_i \Psi (Y) \circ \varphi_i^{-1} [0] \|_{ B^{a, s}_{\ell, x} } \text{.}
\end{align*}
Applying \eqref{eqr.change_of_coord}, \eqref{eq.scalar_red_D}, \eqref{eq.comp_mainf}, and Proposition \ref{thm.reg_prop} completes the proof of \eqref{eq.est_prod_imp}.

Similarly, to prove \eqref{eq.est_prod_ex}, we decompose and apply \eqref{eqc.est_prod_ex}:
\footnote{The $N^1$-norm below is the analogue on the Euclidean space $[0, \delta] \times \R^2$ of the similarly named norm defined in Section \ref{sec.cfol_reg}; see the beginning of Appendix \ref{sec.eucl}.}
\begin{align*}
\| \Phi \otimes \Psi \|_{ \mc{B}^{1, \infty, 0}_{\ell, t, x} } &= \sum_{i = 1}^N \sum_{ X \in i \mc{X}^{r_1}_{l_1} } \sum_{ Y \in i \mc{X}^{r_2}_{l_2} } \| [ \tilde{\eta}_i \Phi(X) \cdot \eta_i \Psi (Y) ] \circ \varphi_i^{-1} \|_{ B^{1, \infty, 0}_{\ell, t, x} } \\
&\lesssim \sum_{i = 1}^N \sum_{ X \in i \mc{X}^{r_1}_{l_1} } \{ \| \tilde{\eta}_i \Phi (X) \circ \varphi_i^{-1} \|_{ N^1_{t, x} } + \| \tilde{\eta}_i \Phi (X) \circ \varphi_i^{-1} [0] \|_{ H^{1/2}_x } \} \\
&\qquad \cdot \sum_{ Y \in i \mc{X}^{r_2}_{l_2} } \{ \| \eta_i \Psi (Y) \circ \varphi_i^{-1} \|_{ N^1_{t, x} } + \| \eta_i \Psi (Y) \circ \varphi_i^{-1} [0] \|_{ H^{1/2}_x } \} \text{.}
\end{align*}
Similar to the preceding proofs, we can now apply \eqref{eqr.change_of_coord}, \eqref{eq.comp_mainf}, Proposition \ref{thm.reg_prop}, and \eqref{eq.scalar_red_H} in order to obtain the desired estimate \eqref{eq.est_prod_ex}.
\end{proof}

\begin{remark}
Like in Theorem \ref{thm.est_prod_elemf}, the estimates \eqref{eq.est_prod_imp} and \eqref{eq.est_prod_ex} still hold if the tensor products $\Phi \otimes \Psi$ on the left-hand sides are replaced by zero or more contractions, metric contractions, and volume form contractions applied to $\Phi \otimes \Psi$.
\end{remark}

\subsection{Integrated Product Estimates} \label{sec.thm_int}

Next are the integrated product estimates.
The basic strategy is the same as before, except we require one additional observation: contractions by $t$-parallel fields commute with the covariant integrals $\cint^t_0$.
This is due to the fact that such contractions commute with $\nabla_t$.

\begin{theorem} \label{thm.est_trace_sh}
Assume that $(\mc{N}, \gamma)$ satisfies \ass{F1}{C, N}.
Furthermore, consider horizontal tensor fields $\Psi \in \mc{C}^\infty \ul{T}^{r_1}_{l_1} \mc{N}$ and $\Phi \in \mc{C}^\infty \ul{T}^{r_2}_{l_2} \mc{N}$.
\begin{itemize}
\item If $a \in [1, \infty]$ and $s \in (-1, 1)$, then
\begin{align}
\label{eq.est_trace_sh} \| \cint^t_0 ( \Phi \otimes \Psi ) \|_{ B^{a, \infty, s}_{\ell, t, x} } &\lesssim_{ C, N, B, s, r_1, l_1, r_2, l_2 } ( \| \nabla \Phi \|_{ L^{2, 2}_{t, x} } + \| \Phi \|_{ L^{\infty, 2}_{x, t} } ) \| \Psi \|_{ B^{a, 2, s}_{\ell, t, x} } \text{.}
\end{align}

\item If $a \in [1, \infty]$ and $s \in (-1, 1)$, then
\begin{align}
\label{eq.est_trace_shp} \| \Phi \otimes \cint^t_0 \Psi \|_{ B^{a, 2, s}_{\ell, t, x} } &\lesssim_{ C, N, B, s, r_1, l_1, r_2, l_2} ( \| \nabla \Phi \|_{ L^{2, 2}_{t, x} } + \| \Phi \|_{ L^{\infty, 2}_{x, t} } ) \| \Psi \|_{ B^{a, 1, s}_{\ell, t, x} } \text{.}
\end{align}

\item In addition, the following estimate holds:
\begin{align}
\label{eq.est_trace_ex} \| \cint^t_0 ( \nabla_t \Phi \otimes \Psi ) \|_{ B^{1, \infty, 0}_{\ell, t, x} } &\lesssim_{ C, N, B, r_1, l_1, r_2, l_2} ( \| \Phi \|_{ N^1_{t, x} } + \| \Phi [0] \|_{ H^{1/2}_x } ) \\
\notag &\qquad\qquad\qquad\qquad \cdot ( \| \Psi \|_{ N^1_{t, x} } + \| \Psi [0] \|_{ H^{1/2}_x } ) \text{.}
\end{align}
\end{itemize}
\end{theorem}

\begin{proof}
Let $\{ U_i, \varphi_i, \eta_i, \tilde{\eta}_i, \mf{p} e_i \}_{i = 1}^N$ be the parallel data induced from the \ass{F1}{C, N} condition.
Again, we can replace all geometric norms by their coordinate analogues.

For \eqref{eq.est_trace_sh}, we once again resort to our usual scalar reduction:
\[ \| \cint^t_0 ( \Phi \otimes \Psi ) \|_{ \mc{B}^{a, \infty, s}_{\ell, t, x} } = \sum_{i = 1}^N \sum_{ X \in i \mc{X}^{r_1}_{l_1} } \sum_{ Y \in i \mc{X}^{r_2}_{l_2} } \| \cint^t_0 [ \tilde{\eta}_i \Phi (X) \cdot \eta_i \Psi (Y) ] \circ \varphi_i^{-1} \|_{ B^{a, \infty, s}_{\ell, t, x} } \text{.} \]
Here, we have noted that $\cint^t_0$ commutes with contractions with both $X$ and $Y$, as well as with multiplication by the ($t$-independent) cutoff functions $\eta_i$ and $\tilde{\eta}_i$.
Applying the Euclidean analogue \eqref{eqc.est_trace_sh} to the above, we have
\begin{align*}
\| \cint^t_0 ( \Phi \otimes \Psi ) \|_{ \mc{B}^{a, \infty, s}_{\ell, t, x} } &\lesssim \sum_{i = 1}^N \sum_{ X \in i \mc{X}^{r_1}_{l_1} } \{ \| \partial [ \tilde{\eta}_i \Phi (X) \circ \varphi_i^{-1} ] \|_{ L^{2, 2}_{t, x} } + \| \tilde{\eta}_i \Phi (X) \circ \varphi_i^{-1} \|_{ L^{\infty, 2}_{x, t} } \} \\
&\qquad \cdot \sum_{ Y \in i \mc{X}^{r_2}_{l_2} } \| \eta_i \Psi (Y) \circ \varphi_i^{-1} \|_{ B^{a, 2, s}_{\ell, t, x} } \text{.}
\end{align*}
Analogous to the proof of Theorem \ref{thm.est_prod}, we can apply \eqref{eqr.change_of_coord}, \eqref{eq.scalar_red_D}, and Proposition \ref{thm.reg_prop} to the above, which suffices to complete the proof of \eqref{eq.est_trace_sh}.

Similarly, for \eqref{eq.est_trace_shp}, we again decompose
\[ \| \Phi \otimes \cint^t_0 \Psi \|_{ \mc{B}^{a, 2, s}_{\ell, t, x} } = \sum_{i = 1}^N \sum_{ X \in i \mc{X}^{r_1}_{l_1} } \sum_{ Y \in i \mc{X}^{r_2}_{l_2} } \| \{ \tilde{\eta}_i \Phi (X) \cdot \cint^t_0 [ \eta_i \Psi (Y) ] \} \circ \varphi_i^{-1} \|_{ B^{a, 2, s}_{\ell, t, x} } \text{.} \]
Again, $\cint^t_0$ commutes with contractions with $Y$ and $\eta_i$.
Applying \eqref{eqc.est_trace_shp} yields
\begin{align*}
\| \Phi \otimes \cint^t_0 \Psi \|_{ \mc{B}^{a, 2, s}_{\ell, t, x} } &\lesssim \sum_{i = 1}^N \sum_{ X \in i \mc{X}^{r_1}_{l_1} } \{ \| \partial [ \tilde{\eta}_i \Phi (X) \circ \varphi_i^{-1} ] \|_{ L^{2, 2}_{t, x} } + \| \tilde{\eta}_i \Phi (X) \circ \varphi_i^{-1} \|_{ L^{\infty, 2}_{x, t} } \} \\
&\qquad \cdot \sum_{ Y \in i \mc{X}^{r_2}_{l_2} } \| \eta_i \Psi (Y) \circ \varphi_i^{-1} \|_{ B^{a, 1, s}_{\ell, t, x} } \text{.}
\end{align*}
From here, the proof is completed using the same reasoning as for \eqref{eq.est_trace_sh}.

Finally, for \eqref{eq.est_trace_ex}, by the usual scalar reduction, we must control
\[ \sum_{i = 1}^N \sum_{ X \in i \mc{X}^{r_1}_{l_1} } \sum_{ Y \in i \mc{X}^{r_2}_{l_2} } \| \cint^t_0 \{ \nabla_t [ \tilde{\eta}_i \Phi (X) ] \eta_i \Psi (Y) \} \circ \varphi_i^{-1} \|_{ B^{1, \infty, 0}_{\ell, t, x} } \text{.} \]
On the right-hand side, since the quantities being differentiated and integrated are all scalar, then $\nabla_t$ and $\cint^t_0$ are the standard derivative and integral, respectively, with respect to $t$.
Thus, applying \eqref{eqc.est_trace_ex}, the above is bounded by
\footnote{Again, the $N^1$-norm below are the Euclidean analogues.}
\begin{align*}
&\sum_{i = 1}^N \sum_{ X \in i \mc{X}^{r_1}_{l_1} } ( \| \tilde{\eta}_i \Phi (X) \circ \varphi_i^{-1} \|_{ N^1_{t, x} } + \| \tilde{\eta}_i \Phi (X) \circ \varphi_i^{-1} [0] \|_{ H^{1/2}_x } ) \\
&\qquad \cdot \sum_{ Y \in i \mc{X}^{r_2}_{l_2} } ( \| \eta_i \Psi (Y) \circ \varphi_i^{-1} \|_{ N^1_{t, x} } + \| \eta_i \Psi (Y) \circ \varphi_i^{-1} [0] \|_{ H^{1/2}_x } ) \text{.}
\end{align*}
From here, the proof proceeds like the end of the proof of \eqref{eq.est_prod_ex}.
\end{proof}

\begin{remark}
Like in Theorem \ref{thm.est_prod}, the estimates \eqref{eq.est_trace_sh}-\eqref{eq.est_trace_ex} still hold if zero or more contractions, metric contractions, and volume form contractions are applied to the quantities within the norms on the left-hand sides.
\end{remark}

\subsection{The Sharp Trace Theorem} \label{sec.thm_stt}

Finally, we discuss another proof of the sharp trace theorem that was essential to \cite{kl_rod:cg, parl:bdc, shao:bdc_nv, wang:cg, wang:tbdc}.
In particular, this proof avoids direct applications of the geometric L-P theory, which in turn allows us to avoid altogether the Gauss curvatures of the $\mc{S}_\tau$'s.

First, we establish the following preliminary Besov estimate.

\begin{lemma} \label{thm.besov_sh_sc}
Assume $(\mc{N}, \gamma)$ satisfies \ass{F2}{C, N, B}.
If $\phi \in C^\infty \mc{N}$, then
\footnote{The families $i \mc{Z}$ were defined in Section \ref{sec.fol_scal}.}
\begin{align}
\label{eq.besov_sh_sc} \| \phi \|_{ L^{\infty, \infty}_{t, x} } &\lesssim_{ C, N, B } \sum_{ i = 1 }^N \sum_{ Z \in i \mc{Z} } \| ( \eta_i \cdot Z \phi ) \circ \varphi_i^{-1} \|_{ B^{1, \infty, 0}_{\ell, t, x} } \\
\notag &\qquad\qquad + \sum_{ i = 1 }^N \sum_{ Z \in i \mc{Z} } \| ( \tilde{\eta}_i \cdot Z \phi ) \circ \varphi_i^{-1} \|_{ L^{\infty, 2}_{t, x} } + \| \phi \|_{ L^{\infty, 2}_{t, x} } \text{.}
\end{align}
\end{lemma}

\begin{proof}
We begin by applying \eqref{eqr.change_of_coord} and the classical sharp Besov embedding \eqref{eq.besov_sh}:
\begin{align*}
\| \phi \|_{ L^{\infty, \infty}_{t, x} } &\lesssim \sum_{ i = 1 }^N \| \eta_i \phi \circ \varphi_i^{-1} \|_{ L^{\infty, \infty}_{t, x} } \\
&\lesssim \sum_{ i = 1 }^N \sum_{ Z \in i \mc{Z} } [ \| ( \eta_i Z \phi ) \circ \varphi_i^{-1} \|_{ B^{1, \infty, 0}_{\ell, t, x} } + \| ( Z \eta_i \cdot \phi ) \circ \varphi_i^{-1} \|_{ B^{1, \infty, 0}_{\ell, t, x} } ] + \| \phi \|_{ L^{\infty, 2}_{t, x} } \text{.}
\end{align*}
It remains only to control the second term on the right-hand side, which we denote by $I$.
From the \ass{F2}{C, N, B} condition (see Proposition \ref{thm.reg_prop} and \ass{R2}{}), we have
\[ \| \partial ( \eta_i \circ \varphi_i^{-1} ) \|_{ L^{\infty, \infty}_{t, x} } \lesssim 1 \text{,} \qquad \| \partial^2 ( \eta_i \circ \varphi_i^{-1} ) \|_{ L^{\infty, \infty}_{t, x} } \lesssim 1 \text{.} \]
Thus, by a trivial embedding along with H\"older's inequality,
\begin{align*}
I &\lesssim \sum_{ i = 1 }^N \{ \| \partial [ \partial ( \eta_i \circ \varphi_i^{-1} ) \cdot ( \phi \circ \varphi_i^{-1} ) ] \|_{ L^{\infty, 2}_{t, x} } + \| \partial ( \eta_i \circ \varphi_i^{-1} ) \cdot ( \phi \circ \varphi_i^{-1} ) \|_{ L^{\infty, 2}_{t, x} } \} \\
&\lesssim \sum_{ i = 1 }^N [ \| \phi \circ \varphi_i^{-1} \|_{ L^{\infty, 2}_{t, x} } + \| ( \tilde{\eta}_i \circ \varphi_i^{-1} ) \cdot \partial ( \phi \circ \varphi_i^{-1} ) \|_{ L^{\infty, 2}_{t, x} } ] \\
&\lesssim \| \phi \|_{ L^{\infty, 2}_{t, x} } + \sum_{ i = 1 }^N \sum_{ Z \in i \mc{Z} } \| ( \tilde{\eta}_i \cdot Z \phi ) \circ \varphi_i^{-1} \|_{ L^{\infty, 2}_{t, x} } \text{.} \qedhere
\end{align*}
\end{proof}

We are now prepared to prove the main sharp trace estimate.

\begin{theorem} \label{thm.sharp_trace}
Assume $(\mc{N}, \gamma)$ satisfies \ass{F2}{C, N, B}.
Let $\Psi \in \mc{C}^\infty \ul{T}^r_l \mc{N}$, and suppose $\Psi_1, \Psi_2 \in \mc{C}^\infty \ul{T}^r_{l + 1} \mc{N}$ are such that the decomposition
\[ \nabla \Psi = \nabla_t \Psi_1 + \Psi_2 \]
holds.
Then, we have the following estimate:
\begin{align}
\label{eq.sharp_trace} \| \Psi \|_{ L^{\infty, 2}_{x, t} } &\lesssim_{ C, N, B, r, l } ( 1 + \| k \|_{ L^{2, \infty}_{x, t} } ) ( \| \Psi_1 \|_{ N^1_{t, x} } + \| \Psi_1 [0] \|_{ H^{1/2}_x } + \| \Psi_2 \|_{ B^{1, 2, 0}_{\ell, t, x} } ) \\
\notag &\qquad\qquad\quad + ( 1 + \| k \|_{ L^{2, \infty}_{x, t} } ) ( \| \Psi \|_{ N^1_{t, x} } + \| \Psi [0] \|_{ H^{1/2}_x } ) \text{.}
\end{align}
\end{theorem}

\begin{proof}
Applying Lemma \ref{thm.besov_sh_sc}, along with \eqref{eqr.change_of_coord} and \eqref{eq.comp_main}, we obtain
\begin{align*}
\| \Psi \|_{ L^{\infty, 2}_{x, t} }^2 &\lesssim \| \cint_0^t | \Psi |^2 \|_{ L^{\infty, \infty}_{t, x} } \\
&\lesssim \sum_{ i = 1 }^N \sum_{ Z \in i \mc{Z} } [ \| \cint_0^t ( \tilde{\eta}_i Z | \Psi |^2 ) \|_{ B^{1, \infty, 0}_{\ell, t, x} } + \| \cint_0^t ( \tilde{\eta}_i Z | \Psi |^2 ) \|_{ L^{\infty, 2}_{t, x} } ] + \| \cint_0^t | \Psi |^2 \|_{ L^{\infty, 2}_{t, x} } \\
&= I_1 + I_2 + I_3 \text{.}
\end{align*}
Here, the Besov norm in $I_1$ is the \emph{geometric} Besov norm.
Note that we could treat $\tilde{\eta}_i Z$ as a global vector field, i.e., as an element of $\mc{C}^\infty \ul{T}^1_0 \mc{N}$, due to the support of $\tilde{\eta}_i$.
The lower-order term $I_3$ can be handled using \eqref{eq.gns_1} and \eqref{eq.int_ineq_bi}:
\[ I_3 \lesssim \| \Psi \|_{ L^{4, 2}_{x, t} }^2 \lesssim \| \Psi \|_{ N^1_{t, x} }^2 \text{.} \]
Moreover, applying \ass{F0}{C, N, B}, \eqref{eq.int_ineq_bi}, and \eqref{eq.cframe_basis_prop}, we obtain
\begin{align*}
I_2 &\lesssim \| \cint_0^t ( \nabla \Psi \otimes \Psi ) \|_{ L^{\infty, 2}_{t, x} } \lesssim \| \nabla \Psi \|_{ L^{2, 2}_{t, x} } \| \Psi \|_{ L^{\infty, 2}_{x, t} } \lesssim \| \Psi \|_{ N^1_{t, x} } \| \Psi \|_{ L^{\infty, 2}_{x, t} } \text{.}
\end{align*}

For $I_1$, writing $Z | \Psi |^2 = 2 \langle \nabla \Psi, Z \otimes \Psi \rangle$, we have
\begin{align*}
I_1 &\lesssim \sum_{ i = 1 }^N \sum_{ Z \in i \mc{Z} } ( \| \cint_0^t \langle \nabla_t \Psi_1 , \tilde{\eta}_i Z \otimes \Psi \rangle \|_{ B^{1, \infty, 0}_{\ell, t, x} } + \| \cint_0^t \langle \Psi_2, \tilde{\eta}_i Z \otimes \Psi \rangle \|_{ B^{1, \infty, 0}_{\ell, t, x} } ) \text{.}
\end{align*}
Applying \eqref{eq.cframe_basis_prop}, \eqref{eq.est_trace_sh}, \eqref{eq.est_trace_ex}, and the remark following Theorem \ref{thm.est_trace_sh}, then
\begin{align*}
I_1 &\lesssim ( \| \Psi_1 \|_{ N^1_{t, x} } + \| \Psi_1 [0] \|_{ H^{1/2}_x } ) \sum_{ i = 1 }^N \sum_{ Z \in i \mc{Z} } \{ \| \tilde{\eta}_i Z \otimes \Psi \|_{ N^1_{t, x} } + \| ( \tilde{\eta}_i Z \otimes \Psi ) [0] \|_{ H^{1/2}_x } \} \\
&\qquad + \| \Psi_2 \|_{ B^{1, 2, 0}_{\ell, t, x} } \sum_{ i = 1 }^N \sum_{ Z \in i \mc{Z} } \| \nabla ( \tilde{\eta}_i Z \otimes \Psi ) \|_{ L^{2, 2}_{t, x} } + \| \Psi_2 \|_{ B^{1, 2, 0}_{\ell, t, x} } \| \Psi \|_{ L^{\infty, 2}_{x, t} } \text{.}
\end{align*}

Fix now any such $\tilde{\eta}_i Z$.
Applying H\"older's inequality yields
\[ \| \tilde{\eta}_i Z \otimes \Psi \|_{ N^1_{t, x} } \lesssim \| \Psi \|_{ N^1_{t, x} } + ( \| \nabla_t Z \|_{ L^{2, \infty}_{x, t} } + \| \nabla Z \|_{ L^{2, \infty}_{x, t} } ) \| \Psi \|_{ L^{\infty, 2}_{x, t} } \text{.} \]
Applying \eqref{eq.cframe_basis_prop} and recalling that $| \nabla_t Z | \lesssim | k | | Z |$, then
\[ \| \tilde{\eta}_i Z \otimes \Psi \|_{ N^1_{t, x} } \lesssim \| \Psi \|_{ N^1_{t, x} } + ( 1 + \| k \|_{ L^{2, \infty}_{x, t} } ) \| \Psi \|_{ L^{\infty, 2}_{x, t} } \text{.} \]
Moreover, by Proposition \ref{thm.sobolev}, \eqref{eq.est_prod_elemf}, and \eqref{eq.cframe_basis_prop},
\begin{align*}
\| ( \tilde{\eta}_i Z \otimes \Psi ) [0] \|_{ H^{1/2}_x } &\lesssim ( \| \nabla Z [0] \|_{ L^2_x } + \| Z [0] \|_{ L^\infty_x } ) \| \Psi [0] \|_{ H^{1/2}_x } \lesssim \| \Psi [0] \|_{ H^{1/2}_x } \text{.}
\end{align*}

Finally, combining all the above, we obtain
\begin{align*}
\| \Psi \|_{ L^{\infty, 2}_{x, t} }^2 &\lesssim ( \| \Psi \|_{ N^1_{t, x} } + \| \Psi_1 \|_{ N^1_{t, x} } + \| \Psi_1 [0] \|_{ H^{1/2}_x } + \| \Psi_2 \|_{ B^{1, 2, 0}_{\ell, t, x} } ) \\
&\qquad \cdot [ \| \Psi \|_{ N^1_{t, x} } + \| \Psi [0] \|_{ H^{1/2}_x } + ( 1 + \| k \|_{ L^{2, \infty}_{x, t} } ) \| \Psi \|_{ L^{\infty, 2}_{x, t} } ] \text{.}
\end{align*}
Applying a weighted Young's inequality completes the proof.
\end{proof}

\section{Weakly Spherical Foliations} \label{sec.curv}

In this section, we restrict ourselves to the special case of foliations $\mc{S}$ being diffeomorphic to $\Sph^2$.
We maintain the same objects and notations as before, in particular the fixed surface $(\mc{S}, h)$ and the foliation $(\mc{N}, \gamma)$.
We will discuss the role of the Gauss curvature in deriving elliptic estimates.
In particular, we would like to consider situations in which one has extremely weak curvature control.

\subsection{Weak Elliptic Estimates} \label{sec.curv_ell}

The first task is to quantify our weak curvature control.
This is accomplished by defining the following curvature regularity conditions, one for $(\mc{S}, h)$ and a corresponding one for $(\mc{N}, \gamma)$.

First, we say $(\mc{S}, h)$ satisfies \ass{k}{C, D}, with data $( f, W, V )$, iff:
\begin{itemize}
\item $f \in \mc{C}^\infty \mc{S}$ satisfies, for any $x \in \mc{S}$, the bounds
\[ C^{-1} \leq f |_x \leq C \text{.} \]

\item $V \in \mc{C}^\infty T^0_1 \mc{S}$ and $W \in \mc{C}^\infty \mc{S}$ satisfy
\[ \| V \|_{ H^{1/2}_x } \leq D \text{,} \qquad \| W \|_{ L^2_x } \leq D \text{.} \]

\item $\mc{K}$ can be decomposed in the form
\[ \mc{K} - f = h^{ab} \nabla_a V_b + W \text{.} \]
\end{itemize}
Moreover, we will only consider the case in which $D$ is very small.
The \ass{k}{} condition can then be interpreted as each $(\mc{S}, \gamma [\tau])$ being very close, in a very weak sense, to the standard Euclidean sphere.
More specifically, $\mc{K}$ is comparable to $1$, except for a ``good" error term $W$ that is $L^2$-bounded, and a ``bad" error term, which is not $L^2$-bounded but can be expressed as a divergence of an $H^{1/2}$-controlled $1$-form $V$.

Similarly, we say $(\mc{N}, \gamma)$ satisfies \ass{K}{C, D}, with data $( f, W, V )$, iff:
\begin{itemize}
\item $f \in \mc{C}^\infty \mc{N}$ satisfies, for any $(\tau, x) \in \mc{N}$, the bounds
\[ C^{-1} \leq f |_{(\tau, x)} \leq C \text{.} \]

\item $V \in \mc{C}^\infty \ul{T}^0_1 \mc{N}$ and $W \in \mc{C}^\infty \mc{N}$ satisfy
\[ \| V \|_{ B^{2, \infty, 1/2}_{\ell, t, x} } \leq D \text{,} \qquad \| W \|_{ L^{\infty, 2}_{t, x} } \leq D \text{.} \]

\item $\mc{K}$ can be decomposed in the form
\[ \mc{K} - f = \gamma^{ab} \nabla_a V_b + W \text{.} \]
\end{itemize}
Note that if $(\mc{N}, \gamma)$ satisfies \ass{K}{C, N}, with data $( f, W, V )$, then any $(\mc{S}, \gamma [\tau])$ trivially satisfies \ass{k}{C, D}, with data $( f [\tau], W [\tau], V [\tau])$.

We now discuss basic elliptic estimates on $(\mc{S}, h)$ using the \ass{k}{} condition.
First is the following general divergence-curl estimate.

\begin{proposition} \label{thm.div_curl_est}
Assume $(\mc{S}, h)$ satisfies \ass{r1}{C, N} and \ass{k}{C, D}, with $D \ll 1$ sufficiently small.
Then, for any $F \in \mc{C}^\infty T^r_{l+1} \mc{S}$,
\footnote{Here, $h^{ab} \nabla_a F_b \in \mc{C}^\infty T^r_l \mc{S}$ refers to the metric contraction of $\nabla F$ in the derivative component and a fixed covariant component of $F$.  The expression $\omega^{ab} \nabla_a F_b$ is defined similarly.}
\begin{equation} \label{eq.div_curl_est} \| \nabla F \|_{ L^2_x } \lesssim_{C, N, r, l} \| h^{ab} \nabla_a F_b \|_{ L^2_x } + \| \omega^{ab} \nabla_a F_b \|_{ L^2_x } + \| F \|_{ L^2_x } \text{,} \end{equation}
\end{proposition}

\begin{proof}
A standard integration by parts yields the identity
\[ \int_{ \mc{S} } | \nabla F |^2 d \omega = \int_{ \mc{S} } | h^{ab} \nabla_a F_b |^2 d \omega + \int_\mc{S} | \omega^{ab} \nabla_a F_b |^2 d \omega + \int_\mc{S} E \cdot d \omega \text{,} \]
where the error term $E$ is a sum of $L$ terms, with $L \lesssim r + l + 1$, and where each such term can be expressed as contractions of $\mc{K} \otimes F \otimes F$.

Let $(f, W, V)$ be the data for the \ass{k}{C, D} condition.
To control the error terms, we decompose $\mc{K}$ using the \ass{k}{C, D} and integrate the divergence term by parts:
\begin{align*}
\abs{ \int_\mc{S} \mc{K} \otimes F \otimes F \cdot d \omega } &\lesssim \abs{ \int_\mc{S} ( \mc{K} - f ) \otimes F \otimes F \cdot d \omega } + \| F \|_{ L^2_x }^2 \\
&\lesssim \int_\mc{S} | V | | \nabla F | | F | d \omega + \int_\mc{S} | W | | F |^2 d \omega + \| F \|_{ L^2_x }^2 \\
&= I_1 + I_2 + \| F \|_{ L^2_x }^2 \text{.}
\end{align*}
For $I_1$, we apply H\"older's inequality, \eqref{eq.gns_1}, \eqref{eq.sob_frac_shf}, and \ass{k}{C, D}:
\begin{align*}
(r + l + 1) I_1 &\lesssim (r + l + 1) \| V \|_{ L^4_x } \| \nabla F \|_{ L^2_x } \| F \|_{ L^4_x } \\
&\lesssim D [ \| \nabla F \|_{ L^2_x }^2 + (r + l + 1)^4 \| F \|_{ L^2_x }^2 ] \text{,}
\end{align*}
Similarly, for $I_2$, we apply H\"older's inequality, \eqref{eq.gns_1}, and \ass{k}{C, D}:
\begin{align*}
(r + l + 1) I_2 &\lesssim (r + l + 1) \| W \|_{ L^2_x } \| F \|_{ L^4_x }^2 \\
&\lesssim D [ \| \nabla F \|_{ L^2_x }^2 + (r + l + 1)^2 \| F \|_{ L^2_x }^2 ] \text{.}
\end{align*}

Combining all of the above, we obtain
\[ \| \nabla F \|_{ L^2_x }^2 \lesssim \| h^{ab} \nabla_a F_b \|_{ L^2_x }^2 + \| \omega^{ab} \nabla_a F_b \|_{ L^2_x }^2 + D \| \nabla F \|_{ L^2_x }^2 + (r + l + 1)^4 \| F \|_{ L^2_x }^2 \text{.} \]
Since $D$ is small, the desired estimate \eqref{eq.div_curl_est} follows.
\end{proof}

\begin{corollary} \label{thm.ell_est_sc}
Assume $(\mc{S}, h)$ satisfies \ass{r1}{C, N} and \ass{k}{C, D}, with $D \ll 1$ sufficiently small.
If $f \in \mc{C}^\infty \mc{S}$, then the following estimate holds:
\begin{equation} \label{eq.ell_est_sc} \| \nabla^2 f \|_{ L^2_x } \lesssim_{C, N} \| \lapl f \|_{ L^2_x } + \| \nabla f \|_{ L^2_x } \text{.} \end{equation}
\end{corollary}

\begin{proof}
Apply \eqref{eq.div_curl_est}, with $F = \nabla f$.
\end{proof}

\begin{corollary} \label{thm.glp_bochner_sc}
Assume $(\mc{S}, h)$ satisfies \ass{r1}{C, N} and \ass{k}{C, D}, with $D \ll 1$ sufficiently small.
Then, for any integer $k \geq 0$ and for any $f \in \mc{C}^\infty \mc{S}$,
\begin{align}
\label{eq.glp_bochner_sc} \| \nabla^2 P_k f \|_{ L^2_x } + \| P_k \nabla^2 f \|_{ L^2_x } &\lesssim_{C, N, r, l} 2^{ 2 k } \| f \|_{ L^2_x } \text{,} \\
\notag \| \nabla^2 P_{< 0} f \|_{ L^2_x } + \| P_{< 0} \nabla^2 f \|_{ L^2_x } &\lesssim_{C, N, r, l} \| f \|_{ L^2_x } \text{.}
\end{align}
\end{corollary}

\begin{proof}
The estimate for $\nabla^2 P_k f$ in \eqref{eq.glp_bochner_sc} follows from Proposition \ref{thm.glp} and \eqref{eq.ell_est_sc}:
\[ \| \nabla^2 P_k f \|_{ L^2_x } \lesssim \| \lapl P_k f \|_{ L^2_x } + \| \nabla P_k f \|_{ L^2_x } \lesssim 2^{2 k} \| f \|_{ L^2_x } \text{.} \]
The estimate for $P_k \nabla^2 f$ follows by a standard duality argument.
\footnote{Recall in particular that all the geometric L-P operators commute with metric contractions.}
The second part of \eqref{eq.glp_bochner_sc} can be proved using similar methods.
\end{proof}

\begin{remark}
Similar estimates hold in the foliation setting.
For example, if $(\mc{N}, \gamma)$ satisfies \ass{R1}{C, N} and \ass{K}{C, D}, if $p \in [1, \infty]$, and if $\Psi \in \mc{C}^\infty \ul{T}^r_l \mc{N}$, then
\[ \| \nabla \Psi \|_{ L^{p, 2}_{t, x} } \lesssim_{C, N, r, l} \| \gamma^{ab} \nabla_a \Psi_b \|_{ L^{p, 2}_{t, x} } + \| \epsilon^{ab} \nabla_a \Psi_b \|_{ L^{p, 2}_{t, x} } + \| \Psi \|_{ L^{p, 2}_{t, x} } \text{.} \]
Other elliptic estimates on $(\mc{S}, h)$ will have analogous extensions to $(\mc{N}, \gamma)$.
\end{remark}

\subsection{Hodge Elliptic Estimates} \label{sec.curv_hodge}

Next, we derive similar elliptic estimates for the symmetric Hodge operators.
The basic strategy is the same as in Proposition \ref{thm.div_curl_est}.

\begin{proposition} \label{thm.hodge_est}
Assume $(\mc{S}, h)$ satisfies \ass{r1}{C, N} and \ass{k}{C, D}, with $D \ll 1$ sufficiently small.
Then, the following Hodge-elliptic estimates hold:
\begin{itemize}
\item If $X \in \mc{C}^\infty H_1 \mc{S}$, then
\begin{equation} \label{eq.hodge_est_D1} \| \nabla X \|_{ L^2_x } + \| X \|_{ L^2_x } \lesssim_{C, N} \| \mc{D}_1 X \|_{ L^2_x } \text{.} \end{equation}

\item If $X \in \mc{C}^\infty H_2 \mc{S}$, then
\begin{equation} \label{eq.hodge_est_D2} \| \nabla X \|_{ L^2_x } + \| X \|_{ L^2_x } \lesssim_{C, N} \| \mc{D}_2 X \|_{ L^2_x } \text{.} \end{equation}

\item If $X \in \mc{C}^\infty H_0 \mc{S}$, then
\begin{equation} \label{eq.hodge_est_D1a} \| \nabla X \|_{ L^2_x } \simeq \| \mc{D}_1^\ast X \|_{ L^2_x } \text{.} \end{equation}

\item If $X \in \mc{C}^\infty H_1 \mc{S}$, then
\begin{equation} \label{eq.hodge_est_D2a} \| \nabla X \|_{ L^2_x } \lesssim_{C, N} \| \mc{D}_2^\ast X \|_{ L^2_x } + \| X \|_{ L^2_x } \text{.} \end{equation}
\end{itemize}
\end{proposition}

\begin{proof}
For \eqref{eq.hodge_est_D1}, we use the first identity in \eqref{eq.hodge_id} to obtain
\[ \| \nabla X \|_{ L^2_x }^2 + \int_{ \mc{S} } f | X |^2 d \epsilon = \| \mc{D}_1 X \|_{ L^2_x }^2 - \int_{ \mc{S} } ( \mc{K} - f ) | X |^2 d \epsilon \text{.} \]
The last term on the right-hand side can be handled in an analogous manner as in the proof of \eqref{eq.div_curl_est}.
Recalling the lower bound for $f$ in \ass{k}{C, D}, then
\[ \| \nabla X \|_{ L^2_x }^2 + \| X \|_{ L^2_x }^2 \lesssim \| \mc{D}_1 X \|_{ L^2_x }^2 + D ( \| \nabla X \|_{ L^2_x }^2 + \| X \|_{ L^2_x }^2 ) \text{,} \]
and \eqref{eq.hodge_est_D1} follows.
The remaining estimates \eqref{eq.hodge_est_D2}-\eqref{eq.hodge_est_D2a} are similarly proved.
\end{proof}

Suppose for the moment that $(\mc{S}, h)$ satisfies \ass{r1}{C, N} and \ass{k}{C, D}, with $D$ small in the sense of Proposition \ref{thm.hodge_est}.
Then, \eqref{eq.hodge_est_D1} and \eqref{eq.hodge_est_D2} imply that $\mc{D}_1$ and $\mc{D}_2$ are one-to-one, and that both operators have $L^2$-bounded inverses.
Furthermore, we can extend these inverses to $L^2$-bounded operators
\[ \mc{D}_i^{-1} : \mc{C}^\infty H_{i-1} \mc{S} \rightarrow \mc{C}^\infty H_i \mc{S} \text{,} \qquad i \in \{ 1, 2 \} \text{,} \]
by defining $\mc{D}_i^{-1} X$ to be the (actual) inverse of $\mc{D}_i$ acting on the $L^2$-orthogonal projection of $X$ onto the (closed) range of $\mc{D}_i$.
If we let $\mc{P}_i$ denote this $L^2$-projection onto the range of $\mc{D}_i$, then by the above definitions, we have
\[ \mc{D}_i^{-1} \mc{D}_i = I \text{,} \qquad \mc{D}_i \mc{D}_i^{-1} = \mc{P}_i \text{.} \]

Next, one can use the above to partially invert the $\mc{D}_i^\ast$'s as well.
Since $\mc{D}_i$ is injective, then $\mc{D}_i^\ast$ (which also has closed range) is surjective, and its inverse image of any element of $\mc{C}^\infty H_i \mc{S}$ is a coset of the nullspace of $\mc{D}_i^\ast$.
\footnote{To be fully rigorous, we must invoke some functional analytic technicalities and consider the $\mc{D}_i$'s and $\mc{D}_i^\ast$'s as densely defined unbounded operators on the appropriate $L^2$-spaces.}
Since the nullspace of $\mc{D}_i^\ast$ is the orthogonal complement of the range of $\mc{D}_i$, we can define $\mc{D}_i^{\ast -1} X$ to be the unique element of the corresponding inverse image of $X$ that is in the range of $\mc{D}_i$.
In summary, we have the following identities:
\[ \mc{D}_i^{\ast -1} \mc{D}_i^\ast = \mc{P}_i \text{,} \qquad \mc{D}_i^\ast \mc{D}_i^{\ast -1} = I \text{.} \]

In addition, we can make the following observations:
\begin{itemize}
\item The range of $\mc{D}_1$ is the space of ($h$-)mean-free functions.
Thus, $I - \mc{P}_1$ applied to any $X \in \mc{C}^\infty H_0 \mc{S}$ is simply the mean of $X$.

\item The range of $\mc{D}_2$ is the orthogonal complement of the space of conformal Killing $1$-forms on $\mc{S}$ (with respect to $h$).
As a result, $I - \mc{P}_2$ projects the space of $1$-forms onto the subspace of conformal Killing $1$-forms.
\end{itemize}
In particular, as projections, the operators $I - \mc{P}_1$, and $I - \mc{P}_2$ are $L^2$-bounded.

By combining Proposition \ref{thm.hodge_est} with the above constructions of the ``inverse" Hodge elliptic operators, we can establish the following estimates.

\begin{proposition} \label{thm.hodge_inv_est}
Assume $(\mc{S}, h)$ satisfies \ass{r1}{C, N} and \ass{k}{C, D}, with $D \ll 1$ sufficiently small.
Let $\mc{D}$ denote any one of the operators $\mc{D}_1$, $\mc{D}_2$, $\mc{D}_1^\ast$, $\mc{D}_2^\ast$, and suppose $X$ is a smooth section of the appropriate Hodge bundle on $\mc{S}$.
\begin{itemize}
\item The following Hodge-elliptic estimates hold:
\begin{equation} \label{eq.hodge_inv_est} \| \nabla \mc{D}^{-1} X \|_{ L^2_x } + \| \mc{D}^{-1} X \|_{ L^2_x } \lesssim_{C, N} \| X \|_{ L^2_x } \text{.} \end{equation}

\item Furthermore, if $k \geq 0$ is an integer, then
\begin{equation} \label{eq.glp_hodge} \| P_k \mc{D}^{-1} X \|_{ L^2_x } \lesssim_{C, N} 2^{-k} \| X \|_{ L^2_x } \text{,} \qquad \| \mc{D}^{-1} P_k X \|_{ L^2_x } \lesssim_{C, N} 2^{-k} \| X \|_{ L^2_x } \text{.} \end{equation}
\end{itemize}
\end{proposition}

\begin{proof}
The cases $\mc{D} = \mc{D}_1$ and $\mc{D} = \mc{D}_2$ in \eqref{eq.hodge_inv_est} follow immediately from \eqref{eq.hodge_est_D1} and \eqref{eq.hodge_est_D2}.
Since $\mc{D}_1^{-1}$ and $\mc{D}_2^{-1}$ are $L^2$-bounded, their adjoints $( \mc{D}_1^{-1} )^\ast = \mc{D}_1^{\ast -1}$ and $( \mc{D}_2^{-1} )^\ast = \mc{D}_2^{\ast -1}$ are also $L^2$-bounded.
Thus, the cases $\mc{D} = \mc{D}_1^\ast$ and $\mc{D} = \mc{D}_2^\ast$ follow from \eqref{eq.hodge_est_D1a}, \eqref{eq.hodge_est_D2a}, and the above observation.
This proves \eqref{eq.hodge_inv_est}.

Next, the first estimate of \eqref{eq.glp_hodge} follows immediately from \eqref{eq.glp_fbr} and \eqref{eq.hodge_inv_est}:
\[ \| P_k \mc{D}^{-1} X \|_{ L^2_x } \lesssim 2^{-k} \| \nabla \mc{D}^{-1} X \|_{ L^2_x } \lesssim 2^{-k} \| X \|_{ L^2_x } \text{.} \]
The second estimate of \eqref{eq.glp_hodge} again follows from the first by duality.
\end{proof}

\begin{remark}
In particular, the case $\mc{D} = \mc{D}_1^\ast$ in \eqref{eq.hodge_inv_est} contains the Poincar\'e inequality:
\begin{equation} \label{eq.poincare} \| f \|_{ L^2_x } \lesssim_{C, N} \| \nabla f \|_{ L^2_x } \text{,} \qquad f \in \mc{C}^\infty \mc{S} \text{,} \quad \int_{ \mc{S} } f \cdot d \omega = 0 \text{.} \end{equation}
\end{remark}

Finally, we address a technical point and note that higher frequencies have nonlethal interactions with the operators $I - \mc{P}_1$ and $I - \mc{P}_2$.

\begin{proposition} \label{thm.hodge_proj}
Assume $(\mc{S}, h)$ satisfies \ass{r1}{C, N} and \ass{k}{C, D}, with $D \ll 1$ sufficiently small.
Then, for any integer $k \geq 0$, $i \in \{ 0, 1 \}$, and $X \in \mc{C}^\infty H_i \mc{S}$,
\begin{equation} \label{eq.hodge_proj} \| P_k ( I - \mc{P}_{i+1} ) X \|_{ L^2_x } \lesssim_{ C, N } 2^{-k} \| X \|_{ L^2_x } \text{.} \end{equation}
\end{proposition}

\begin{proof}
First, we apply \eqref{eq.glp_fbr} along with either \eqref{eq.hodge_est_D1a} or \eqref{eq.hodge_est_D2a}:
\[ \| P_k ( I - \mc{P}_{i+1} ) X \|_{ L^2_x } \lesssim 2^{-k} [ \| \mc{D}_i^\ast ( I - \mc{P}_{i+1} ) X \|_{ L^2_x } + \| ( I - \mc{P}_{i+1} ) X \|_{ L^2_x } ] \text{.} \]
Since the range of $I - \mc{P}_{i+1}$ lies in the kernel of $\mc{D}_i^\ast$, then we obtain \eqref{eq.hodge_proj}:
\[ \| P_k ( I - \mc{P}_{i+1} ) X \|_{ L^2_x } \lesssim 2^{-k} \| ( I - \mc{P}_{i+1} ) X \|_{ L^2_x } \lesssim 2^{-k} \| X \|_{ L^2_x } \text{.} \qedhere \]
\end{proof}

\subsection{The Weak Besov-Elliptic Estimate} \label{sec.curv_wbdh}

Proposition \ref{thm.hodge_inv_est} showed (under certain regularity conditions) that the operators $\nabla \mc{D}^{-1}$, where $\mc{D}$ denotes any one of $\mc{D}_1$, $\mc{D}_2$, $\mc{D}_1^\ast$, $\mc{D}_2^\ast$, are $L^2$-bounded.
One can then ask a related question: are these operators $\nabla \mc{D}^{-1}$ also bounded in the geometric Besov spaces?

As a preliminary step, we answer this question affirmatively here for the specific case $\mc{D} = \mc{D}_1^\ast$.
In this case, the quantity in question is scalar, hence we can take advantage of the elliptic estimates proved in Section \ref{sec.curv_ell}.

\begin{lemma} \label{thm.intertwining_weak}
Assume $(\mc{S}, h)$ satisfies \ass{r1}{C, N} and \ass{k}{C, D}, with $D \ll 1$ sufficiently small.
Then, for any integers $k, m \geq 0$ and $X \in \mc{C}^\infty H_0 \mc{S}$,
\begin{align}
\label{eq.intertwining_weak} \| P_k \nabla \mc{D}_1^{\ast -1} P_m X \|_{ L^2_x } &\lesssim_{C, N} 2^{- |k - m|} \| P_{\sim m} X \|_{ L^2_x } \text{,} \\
\notag \| P_k \nabla \mc{D}_1^{\ast -1} P_{< 0} X \|_{ L^2_x } &\lesssim_{C, N} 2^{-k} \| P_{\lesssim 0} X \|_{ L^2_x } \text{,} \\
\notag \| P_{< 0} \nabla \mc{D}_1^{\ast -1} P_m X \|_{ L^2_x } &\lesssim_{C, N} 2^{-m} \| P_{\sim m} X \|_{ L^2_x } \text{,} \\
\notag \| P_{< 0} \nabla \mc{D}_1^{\ast -1} P_{< 0} X \|_{ L^2_x } &\lesssim_{C, N} \| P_{\lesssim 0} X \|_{ L^2_x } \text{.}
\end{align}
\end{lemma}

\begin{proof}
We show only the first estimate here, since the remaining bounds are similarly proved and are easier.
First if $k \leq m$, we can apply \eqref{eq.glp_fb} and \eqref{eq.glp_hodge}:
\[ \| P_k \nabla \mc{D}_1^{\ast -1} P_m X \|_{ L^2_x } \lesssim 2^k \| \mc{D}_1^{\ast -1} P_m X \|_{ L^2_x } \lesssim 2^{k - m} \| P_{\sim m} X \|_{ L^2_x } \text{.} \]
On the other hand, if $k \geq m$, then we apply \eqref{eq.glp_fbr} and \eqref{eq.ell_est_sc}:
\begin{align*}
\| P_k \nabla \mc{D}_1^{\ast -1} P_m X \|_{ L^2_x } &\lesssim 2^{-k} [ \| \lapl \mc{D}_1^{\ast -1} P_m X \|_{ L^2_x } + \| \nabla \mc{D}_1^{\ast -1} P_m X \|_{ L^2_x } ] \text{.}
\end{align*}
The second term on the right-hand side is trivially bounded using \eqref{eq.hodge_inv_est}.
For the first term, we recall the identity $\lapl = - \mc{D}_1 \mc{D}_1^\ast$ from \eqref{eq.hodge_sq} along with \eqref{eq.glp_fb}.
Then,
\begin{align*}
\| P_k \nabla \mc{D}_1^{\ast -1} P_m X \|_{ L^2_x } &\lesssim 2^{-k} [ \| \mc{D}_1 P_m X \|_{ L^2_x } + \| P_m X \|_{ L^2_x } ] \lesssim 2^{-k + m} \| P_{\sim m} X \|_{ L^2_x } \text{.} \qedhere
\end{align*}
\end{proof}

Suppose for the moment that $(\mc{N}, \gamma)$ satisfies \ass{R1}{C, N} and \ass{K}{C, D}.
Since every $(\mc{S}, \gamma [\tau])$ satisfies both \ass{r1}{} and \ass{k}{}, then the inverse Hodge operators are well-defined on each $(\mc{S}, \gamma [\tau])$, by the developments of Section \ref{sec.curv_hodge}.
Consequently, we can aggregate these operators in the usual way; given $i \in \{ 1, 2 \}$, we have
\begin{align*}
\mc{D}_i^{-1}: \mc{C}^\infty \ul{H}_{i-1} \mc{N} \rightarrow \mc{C}^\infty \ul{H}_i \mc{N} \text{,} \qquad \mc{D}_i^{\ast -1}: \mc{C}^\infty \ul{H}_i \mc{N} \rightarrow \mc{C}^\infty \ul{H}_{i-1} \mc{N} \text{.}
\end{align*}

\begin{theorem} \label{thm.besov_weak}
Assume $(\mc{N}, \gamma)$ satisfies \ass{R1}{C, N} and \ass{K}{C, D}, with $D \ll 1$ sufficiently small.
Then, for any $a, p \in [1, \infty]$, $s \in (-1, 1)$, and $\xi \in \mc{C}^\infty \ul{H}_0 \mc{N}$,
\begin{equation} \label{eq.besov_weak} \| \nabla \mc{D}_1^{\ast -1} \xi \|_{ B^{a, p, s}_{\ell, t, x} } \lesssim_{C, N, s} \| \xi \|_{ B^{a, p, s}_{\ell, t, x} } \text{.} \end{equation}
\end{theorem}

\begin{proof}
Note that one needs only to consider the cases $a = 1$ and $a = \infty$, as the remaining cases can be retrieved via interpolation arguments.
Moreover, note each $(\mc{S}, \gamma [\tau])$ satisfies both \ass{k}{C, D} and \ass{r1}{C^\prime, N}, for some $C^\prime$ depending on $C$ and $N$.

For the case $a = 1$, we decompose the left-hand side and apply \eqref{eq.intertwining_weak}:
\begin{align*}
\| \nabla \mc{D}_1^{\ast -1} \xi \|_{ B^{1, p, s}_{\ell, t, x} } &\lesssim \sum_{k, m \geq 0} 2^{s k} \| P_k \nabla \mc{D}_1^{\ast -1} P_m \xi \|_{ L^{p, 2}_{t, x} } + \sum_{k \geq 0} 2^{s k} \| P_k \nabla \mc{D}_1^{\ast -1} P_{< 0} \xi \|_{ L^{p, 2}_{t, x} } \\
&\qquad + \sum_{m \geq 0} \| P_{< 0} \nabla \mc{D}_1^{\ast -1} P_m \xi \|_{ L^{p, 2}_{t, x} } + \| P_{< 0} \nabla \mc{D}_1^{\ast -1} P_{< 0} \xi \|_{ L^{p, 2}_{t, x} } \\
&\lesssim \sum_{k, m \geq 0} 2^{- |k - m|} 2^{s k} \| P_{\sim m} \xi \|_{ L^{p, 2}_{t, x} } + \sum_{k \geq 0} 2^{s k} 2^{-k} \| P_{\lesssim 0} \xi \|_{ L^{p, 2}_{t, x} } \\
&\qquad + \sum_{m \geq 0} 2^{-m} \| P_{\sim m} \xi \|_{ L^{p, 2}_{t, x} } + \| P_{\lesssim 0} \xi \|_{ L^{p, 2}_{t, x} } \text{.}
\end{align*}
Evaluating the above sums, we obtain
\begin{align*}
\| \nabla \mc{D}_1^{\ast -1} \xi \|_{ B^{1, p, s}_{\ell, t, x} } &\lesssim \sum_{m \geq 0} 2^{sm} \| P_{\sim m} \xi \|_{ L^{p, 2}_{t, x} } + \| P_{\lesssim 0} \xi \|_{ L^{p, 2}_{t, x} } \lesssim \| \xi \|_{ B^{1, p, s}_{\ell, t, x} } \text{.}
\end{align*}
Similarly, if $a = \infty$, a similar decomposition and application of \eqref{eq.intertwining_weak} yields
\begin{align*}
\| \nabla \mc{D}_1^{\ast -1} \xi \|_{ B^{\infty, p, s}_{\ell, t, x} } &\lesssim \sup_{k \geq 0} \left( \sum_{m \geq 0} 2^{- |k - m|} 2^{s k} \| P_{\sim m} \xi \|_{ L^{p, 2}_{t, x} } + 2^{s k} 2^{-k} \| P_{\lesssim 0} \xi \|_{ L^{p, 2}_{t, x} } \right) \\
&\qquad + \sum_{m \geq 0} 2^{-m} \| P_{\sim m} \xi \|_{ L^{p, 2}_{t, x} } + \| P_{\lesssim 0} \xi \|_{ L^{p, 2}_{t, x} } \\
&\lesssim \| \xi \|_{ B^{\infty, p, s}_{\ell, t, x} } \text{.} \qedhere
\end{align*}
\end{proof}

As usual, by considering $p = \infty$ and equivariant $\gamma$ in Theorem \ref{thm.besov_weak}, we obtain the direct analogue of \eqref{eq.besov_weak} for a single surface.

\begin{corollary} \label{thm.besov_weakf}
Assume $(\mc{S}, h)$ satisfies \ass{r1}{C, N} and \ass{k}{C, D}, with $D \ll 1$ sufficiently small.
Then, for any $a \in [1, \infty]$, $s \in (-1, 1)$, and $X \in \mc{C}^\infty H_0 \mc{S}$,
\begin{equation} \label{eq.besov_weakf} \| \nabla \mc{D}_1^{\ast -1} X \|_{ B^{a, s}_{\ell, x} } \lesssim_{C, N, s} \| X \|_{ B^{a, s}_{\ell, x} } \text{.} \end{equation}
\end{corollary}

\subsection{Conformal Renormalization} \label{sec.curv_conf}

One may next ask whether the remaining operators $\nabla \mc{D}^{-1}$, with $\mc{D}$ being any one of $\mc{D}_1$, $\mc{D}_2$, or $\mc{D}_2^\ast$, are also bounded in these geometric Besov norms.
In these nonscalar settings, however, one must deal much more directly with the curvature $\mc{K}$.
This adds considerable difficulty to the process, since $\mc{K}$ is so irregular under our \ass{k}{} and \ass{K}{} conditions.

In the remainder of this section, we give an affirmative answer to the aforementioned question when $s = 0$ by using a conformal renormalization argument.
We begin the process here by constructing the conformal transformation.
Suppose for now that $(\mc{N}, \gamma)$ satisfies \ass{R1}{C, N} and \ass{K}{C, D}, the latter with data $(f, W, V)$ and with $D$ sufficiently small.
Recall the \ass{K}{C, D} condition gives a decomposition of $\mc{K}$ into a ``good" term which is in $L^2_x$ and a ``bad" term which is not.

The main idea is that the ``bad" term has a special divergence form, which will allow us to renormalize it away.
Indeed, since $\gamma^{ab} \nabla_a V_b$ has zero mean on each $\mc{S}_\tau$, we can define $u \in \mc{C}^\infty \mc{N}$ to be the solution of the Poisson equations
\begin{equation} \label{eq.conf_curv} \lapl u = \gamma^{ab} \nabla_a V_b \text{,} \qquad \int_\mc{S} u [\tau] \cdot d \epsilon [\tau] = 0 \text{,} \qquad \tau \in [0, \delta] \text{.} \end{equation}
With this $u$, we have our desired conformal transformation,
\begin{equation} \label{eq.conf_transform} \bar{\gamma} = e^{2 u} \gamma \text{,} \qquad \bar{\epsilon} = e^{2 u} \epsilon \text{.} \end{equation}

Like in Section \ref{sec.fol_conf}, geometric objects defined with respect to $\bar{\gamma}$ and $\bar{\epsilon}$ will be denoted with a bar above.
By \eqref{eq.conf_gc}, we have
\begin{align}
\label{eq.conf_gc_id} e^{2 u} \bar{\mc{K}} = \mc{K} - \lapl u = f + W \text{.}
\end{align}
Note that the worst divergence term in $\mc{K}$ is no longer present in $\bar{\mc{K}}$.
As a result, we expect $\bar{\mc{K}}$ to have $L^2$-control.
However, actually achieving this bound is a bit subtle, as one must first control the conformal factor $u$.

\begin{proposition} \label{thm.conf}
Assume $(\mc{N}, \gamma)$ satisfies \ass{R1}{C, N} and \ass{K}{C, D}, the latter with data $( f, W, V )$ and with $D \ll 1$ sufficiently small.
Furthermore, define $u \in \mc{C}^\infty \mc{N}$ and the conformally transformed metric $\bar{\gamma} \in \mc{C}^\infty \ul{T}^0_2 \mc{N}$ as in \eqref{eq.conf_curv} and \eqref{eq.conf_transform}.
\begin{itemize}
\item The following estimates hold for $u$:
\begin{equation} \label{eq.conf_good} \| \nabla u \|_{ B^{2, \infty, 1/2}_{\ell, t, x} } \lesssim_{ C, N } D \text{,} \qquad \| u \|_{ L^{\infty, \infty}_{t, x} } + \| \nabla u \|_{ L^{\infty, 4}_{t, x} } \lesssim_{ C, N } D \text{.} \end{equation}

\item $(\mc{N}, \bar{\gamma})$ satisfies \ass{R1}{C^\prime, N} for some constant $C^\prime$ depending on $C$ and $N$.

\item $(\mc{N}, \bar{\gamma})$ satisfies \ass{K}{C^\prime, D^\prime} for some constants $C^\prime$ and $D^\prime$ depending on $C$ and $N$, with data $( e^{-2u} f, e^{-2u} W, 0)$ and with $D^\prime \ll 1$ sufficiently small.

\item The following estimate holds:
\begin{equation} \label{eq.conf_curv_est} \| \bar{\mc{K}} - e^{-2 u} f \|_{ \bar{L}^{\infty, 2}_{t, x} } \lesssim_{C, N} D \text{.} \end{equation}

\item If $p \in [1, \infty]$ and $\Psi \in \mc{C}^\infty \ul{T}^r_l \mc{N}$, then
\begin{equation} \label{eq.besov_embed} \| \Psi \|_{ \bar{L}^{p, \infty}_{t, x} } \lesssim_{C, N, r, l} \| \Psi \|_{ \bar{B}^{1, p, 1}_{\ell, t, x} } \text{.} \end{equation}

\item If $a \in [1, \infty]$, $p \in [1, \infty]$, $s \in (-1, 1)$, and $\Psi \in \mc{C}^\infty \ul{T}^r_l \mc{N}$, then
\begin{align}
\label{eq.besov_nabla} \| \bar{\nabla} \Psi \|_{ \bar{B}^{a, p, s}_{\ell, t, x} } &\lesssim_{C, N, s, r, l} \| \Psi \|_{ \bar{B}^{a, p, s + 1}_{\ell, t, x} } \text{,} \\
\notag \| \Psi \|_{ \bar{B}^{a, p, s + 1}_{\ell, t, x} } &\lesssim_{C, N, s, r, l} \| \bar{\nabla} \Psi \|_{ \bar{B}^{a, p, s}_{\ell, t, x} } + \| \Psi \|_{ \bar{L}^{p, 2}_{t, x} } \text{.}
\end{align}

\item Let $\bar{\mc{D}}$ be one of $\bar{\mc{D}}_1$, $\bar{\mc{D}}_2$, $\bar{\mc{D}}_1^\ast$, $\bar{\mc{D}}_2^\ast$, with $\xi$ a smooth section of the appropriate Hodge bundle on $\mc{N}$.
If $a \in [1, \infty]$, $p \in [1, \infty]$, and $s \in (-1, 1)$, then
\begin{equation} \label{eq.besov_hodge} \| \bar{\mc{D}}^{-1} \xi \|_{ \bar{B}^{a, p, s + 1}_{\ell, t, x} } \lesssim_{C, N, s} \| \xi \|_{ \bar{B}^{a, p, s}_{\ell, t, x} } \text{.} \end{equation}
\end{itemize}
\end{proposition}

\begin{proof}
We begin by complexifying $u$.
Noting that $\lapl = - \mc{D}_1 \mc{D}_1^\ast$ and defining
\footnote{Note $\mc{D}_1^{\ast -1}$ is well-defined from the arguments in Section \ref{sec.curv_hodge}.}
\[ \nu \in \mc{C}^\infty \ul{H}_0 \mc{N} \text{,} \qquad \nu = \mc{D}_1^{\ast -1} V \text{,} \]
it follows that $u = \real \nu$.
To control $\nu$, we apply \eqref{eq.besov_weak} in order to obtain
\[ \| \nabla \nu \|_{ B^{2, \infty, 1/2}_{\ell, t, x} } = \| \nabla \mc{D}_1^{\ast -1} V \|_{ B^{2, \infty, 1/2}_{\ell, t, x} } \lesssim \| V \|_{ B^{2, \infty, 1/2}_{\ell, t, x} } \leq D \text{.} \]
Next, applying Proposition \ref{thm.gns_ineq} and \eqref{eq.sob_frac_sh} yields
\begin{align*}
\| \nu \|_{ L^{\infty, \infty}_{t, x} } + \| \nabla \nu \|_{ L^{\infty, 4}_{t, x} } &\lesssim \| \nabla \nu \|_{ L^{\infty, 4}_{t, x} } + \| \nu \|_{ L^{\infty, 4}_{t, x} } \lesssim \| \nabla \nu \|_{ B^{2, \infty, 1/2}_{\ell, t, x} } + \| \nu \|_{ L^{\infty, 2}_{t, x} } \text{.}
\end{align*}
Since the Poincar\'e inequality \eqref{eq.poincare} implies
\[ \| \nu \|_{ L^{\infty, 2}_{t, x} } \lesssim \| \nabla \nu \|_{ L^{\infty, 2}_{t, x} } \lesssim \| \nabla \nu \|_{ B^{2, \infty, 1/2}_{\ell, t, x} } \text{,} \]
then from the above, we have obtained
\begin{align*}
\| \nu \|_{ L^{\infty, \infty}_{t, x} } + \| \nabla \nu \|_{ L^{\infty, 4}_{t, x} } &\lesssim \| \nabla \nu \|_{ B^{2, \infty, 1/2}_{\ell, t, x} } \lesssim D \text{.}
\end{align*}
Since $u = \real \nu$, this proves \eqref{eq.conf_good}.

As a result of \eqref{eq.conf_good}, now Proposition \ref{thm.conf_reg} applies to $(\mc{N}, \gamma)$, and hence $(\mc{N}, \bar{\gamma})$ satisfies \ass{R1}{C^\prime, N}.
Furthermore, Proposition \ref{thm.conf_regf} applies to every $(\mc{S}, \gamma [\tau])$.
In addition, from \eqref{eq.conf_unif}, \eqref{eq.conf_comp}, and \eqref{eq.conf_gc_id}, we can derive \eqref{eq.conf_curv_est}:
\begin{align*}
\| \bar{\mc{K}} - e^{-2 u} f \|_{ \bar{L}^{\infty, 2}_{t, x} } = \| e^{-2 u} W \|_{ \bar{L}^{\infty, 2}_{t, x} } \lesssim \| W \|_{ L^{\infty, 2}_{t, x} } \leq D \text{.}
\end{align*}
Since $e^{-2 u} f \simeq 1$, again by \eqref{eq.conf_unif}, then $(\mc{N}, \bar{\gamma})$ also satisfies the \ass{K}{C^\prime, D^\prime}.

Finally, the remaining estimates \eqref{eq.besov_embed}-\eqref{eq.besov_hodge} are proved in Appendix \ref{sec.cell}.
\end{proof}

\begin{remark}
Note that by taking considering equivariant $\gamma$, we can once again obtain a direct fixed-surface analogue of Proposition \ref{thm.conf}.
\end{remark}

\subsection{Improved Besov Estimates} \label{sec.curv_impr}

We now return to the question of whether an analogue of Theorem \ref{thm.besov_weak} holds for $\nabla \mc{D}^{-1}$, where $\mc{D}$ is any one of $\mc{D}_1$ or $\mc{D}_2$.
We show that in our setting, such an analogue does hold when $s = 0$.

An attempt at a direct proof of this statement immediately encounters troubles, since $\mc{K}$ is so irregular.
Our strategy is to make use of the conformally smoothed metric defined in Proposition \ref{thm.conf}.
In particular, \eqref{eq.besov_nabla} and \eqref{eq.besov_hodge} answer the above question affirmatively for $(\mc{N}, \bar{\gamma})$.
Here, we complete our argument by showing that we can pull back some of these estimates for $(\mc{N}, \bar{\gamma})$ back to $(\mc{N}, \gamma)$.

\begin{theorem} \label{thm.besov_impr}
Assume $(\mc{N}, \gamma)$ satisfies \ass{R1}{C, N} and \ass{K}{C, D}, with $D \ll 1$ sufficiently small.
In addition, suppose $p \in [1, \infty]$.
\begin{itemize}
\item If $\Psi \in \mc{C}^\infty \ul{T}^r_l \mc{N}$, then
\begin{equation} \label{eq.besov_impr_sh} \| \Psi \|_{ L^{p, \infty}_{t, x} } \lesssim_{ C, N, r, l } \| \nabla \Psi \|_{ B^{1, p, 0}_{\ell, t, x} } + \| \Psi \|_{ L^{p, 2}_{t, x} } \text{.} \end{equation}

\item If $a \in [1, \infty]$, if $\mc{D}$ is any one of the operators $\mc{D}_1$, $\mc{D}_2$, $\mc{D}_1^\ast$, and if $\xi$ is a smooth section of the appropriate Hodge bundle on $\mc{N}$, then
\begin{equation} \label{eq.besov_impr_bdd} \| \nabla \mc{D}^{-1} \xi \|_{ B^{a, p, 0}_{\ell, t, x} } \lesssim_{C, N} \| \xi \|_{ B^{a, p, 0}_{\ell, t, x} } \text{.} \end{equation}
\end{itemize}
\end{theorem}

\begin{proof}
Let $(f, W, V)$ be the data associated with the \ass{K}{C, D} condition, and let $u$ and $\bar{\gamma}$ be defined as in \eqref{eq.conf_curv} and \eqref{eq.conf_transform}.
In particular, the conclusions of Proposition \ref{thm.conf} hold for our conformally renormalized system $(\mc{N}, \bar{\gamma})$.
The first step is to establish the following intermediate estimates, for any $i \in \{ 1, 2 \}$:
\begin{align}
\label{eql.besov_impr_nabla_pre} \| \bar{\nabla} \Psi - \nabla \Psi \|_{ \bar{B}^{a, p, 0}_{\ell, t, x} } &\lesssim_{ C, N, r, l } D \| \Psi \|_{ \bar{B}^{a, p, 1}_{\ell, t, x} } \text{,} \\
\label{eql.besov_impr_hodge_pre} \| \mc{D}_i^{-1} \xi \|_{ \bar{B}^{a, p, 1}_{\ell, t, x} } &\lesssim_{ C, N } \| \xi \|_{ B^{a, p, 0}_{\ell, t, x} } \text{.}
\end{align}

For \eqref{eql.besov_impr_nabla_pre}, we begin with \eqref{eq.besov_triv_1} and the identity \eqref{eq.conf_cd}:
\[ \| \bar{\nabla} \Psi - \nabla \Psi \|_{ \bar{B}^{a, p, 0}_{\ell, t, x} } \lesssim \| \nabla u \otimes \Psi \|_{ \bar{B}^{1, p, 0}_{\ell, t, x} } \text{.} \]
Applying \eqref{eq.besov_triv_2}, \eqref{eq.est_prod_sob}, and \eqref{eq.conf_besov} yields \eqref{eql.besov_impr_nabla_pre}:
\[ \| \bar{\nabla} \Psi - \nabla \Psi \|_{ \bar{B}^{a, p, 0}_{\ell, t, x} } \lesssim \| \nabla u \|_{ B^{2, \infty, 1/2}_{\ell, t, x} } \| \Psi \|_{ \bar{B}^{2, p, 1/2}_{\ell, t, x} } \lesssim D \| \Psi \|_{ \bar{B}^{a, p, 1}_{\ell, t, x} } \text{.} \]

For \eqref{eql.besov_impr_hodge_pre}, we use the identity $\bar{\mc{D}}_i^{-1} \bar{\mc{D}}_i = I$ in order to obtain
\begin{align*}
\| \mc{D}_i^{-1} \xi \|_{ \bar{B}^{a, p, 1}_{\ell, t, x} } = \| \bar{\mc{D}}_i^{-1} \bar{\mc{D}}_i \mc{D}_i^{-1} \xi \|_{ \bar{B}^{a, p, 1}_{\ell, t, x} } \lesssim \| \bar{\mc{D}}_i \mc{D}_i^{-1} \xi \|_{ \bar{B}^{a, p, 0}_{\ell, t, x} } \text{,}
\end{align*}
where we used \eqref{eq.besov_hodge} in the last step.
By \eqref{eq.conf_hodge}, \eqref{eq.conf_unif}, \eqref{eq.est_prod_elem}, and \eqref{eq.conf_besov},
\begin{align*}
\| \mc{D}_i^{-1} \xi \|_{ \bar{B}^{a, p, 1}_{\ell, t, x} } = \| e^{-2 u} \mc{D}_i \mc{D}_i^{-1} \xi \|_{ \bar{B}^{a, p, 0}_{\ell, t, x} } \lesssim \| \mc{D}_i \mc{D}_i^{-1} X \|_{ B^{a, p, 0}_{\ell, t, x} } \text{.}
\end{align*}
Recalling that $\mc{D}_i \mc{D}_i^{-1} = I - ( I - \mc{P}_i )$, then
\[ \| \mc{D}_i^{-1} \xi \|_{ \bar{B}^{a, p, 1}_{\ell, t, x} } \lesssim \| \xi \|_{ B^{a, p, 0}_{\ell, t, x} } + \| ( I - \mc{P}_i ) \xi \|_{ B^{a, p, 0}_{\ell, t, x} } \lesssim \| \xi \|_{ B^{a, p, 0}_{\ell, t, x} } \text{,} \]
where in the last step, we applied \eqref{eq.hodge_proj}.
This proves \eqref{eql.besov_impr_hodge_pre}.

We are now ready for the main estimates.
From \eqref{eq.besov_nabla} and \eqref{eql.besov_impr_nabla_pre}, we obtain
\[ \| \Psi \|_{ \bar{B}^{1, p, 1}_{\ell, t, x} } \lesssim \| \bar{\nabla} \Psi \|_{ \bar{B}^{1, p, 0}_{\ell, t, x} } + \| \Psi \|_{ \bar{L}^{p, 2}_{t, x} } \lesssim \| \nabla \Psi \|_{ \bar{B}^{1, p, 0}_{\ell, t, x} } + \| \Psi \|_{ \bar{L}^{p, 2}_{t, x} } + D \| \Psi \|_{ \bar{B}^{1, p, 1}_{\ell, t, x} } \text{.} \]
Since $D$ is small, then by \eqref{eq.besov_embed}, we have
\[ \| \Psi \|_{ L^{p, \infty}_{t, x} } \lesssim \| \Psi \|_{ \bar{B}^{1, p, 1}_{\ell, t, x} } \lesssim \| \nabla \Psi \|_{ \bar{B}^{1, p, 0}_{\ell, t, x} } + \| \Psi \|_{ \bar{L}^{p, 2}_{t, x} } \text{.} \]
The estimate \eqref{eq.besov_impr_sh} now follows from \eqref{eq.conf_comp} and \eqref{eq.conf_besov}.

For \eqref{eq.besov_impr_bdd}, first note the case $\mc{D} = \mc{D}_1^\ast$ is just a special case of \eqref{eq.besov_weak}, thus we need only consider when $\mc{D}$ is $\mc{D}_1$ or $\mc{D}_2$.
We begin by applying \eqref{eq.conf_besov} and \eqref{eql.besov_impr_nabla_pre}:
\begin{align*}
\| \nabla \mc{D}^{-1} \xi \|_{ B^{a, p, 0}_{\ell, t, x} } \lesssim \| \nabla \mc{D}^{-1} \xi \|_{ \bar{B}^{a, p, 0}_{\ell, t, x} } \lesssim \| \bar{\nabla} \mc{D}^{-1} \xi \|_{ \bar{B}^{a, p, 0}_{\ell, t, x} } + D \| \mc{D}^{-1} \xi \|_{ \bar{B}^{a, p, 1}_{\ell, t, x} } \text{.}
\end{align*}
The proof is completed by applying \eqref{eq.besov_nabla} and \eqref{eql.besov_impr_hodge_pre}:
\[ \| \nabla \mc{D}^{-1} \xi \|_{ B^{a, p, 0}_{\ell, t, x} } \lesssim \| \mc{D}^{-1} \xi \|_{ \bar{B}^{a, p, 1}_{\ell, t, x} } \lesssim \| \xi \|_{ B^{a, p, 0}_{\ell, t, x} } \text{.} \qedhere \]
\end{proof}

\begin{remark}
Note that the proof of \eqref{eq.besov_impr_bdd} fails when $\mc{D} = \mc{D}_2^\ast$, since $\mc{D}_2^\ast$, in contrast to $\mc{D}_1$, $\mc{D}_1^\ast$, and $\mc{D}_2$, fails to be conformally invariant.
\end{remark}

Finally, from the \ass{K}{} condition, we can still obtain the same $H^{-1/2}$-estimate for $\mc{K}$ that was obtained in \cite{kl_rod:cg, wang:cg}.
Moreover, this can be accomplished very quickly using the conformal renormalization scheme from Section \ref{sec.curv_conf}.

\begin{proposition} \label{thm.curv_sob}
Assume $(\mc{N}, \gamma)$ satisfies \ass{R1}{C, N} and \ass{K}{C, D}, the latter with data $(f, W, V)$, and with $D \ll 1$ sufficiently small.
Then,
\begin{equation} \label{eq.curv_sob} \| \mc{D}_1 V \|_{ B^{2, \infty, -1/2}_{\ell, t, x} } \lesssim D \text{,} \qquad \| \mc{K} - f \|_{ B^{2, \infty, -1/2}_{\ell, t, x} } \lesssim D \text{.} \end{equation}
\end{proposition}

\begin{proof}
By the \ass{K}{C, D} condition, we need only show
\[ \| W \|_{ B^{2, \infty, -1/2}_{\ell, t, x} } \lesssim D \text{,} \qquad \| \mc{D}_1 V \|_{ B^{2, \infty, -1/2}_{\ell, t, x} } \lesssim D \text{.} \]
Moreover, as the first inequality is trivial, it remains only to show the bound for $V$.
As before, we let $u$ and $\bar{\gamma}$ be defined as in \eqref{eq.conf_curv} and \eqref{eq.conf_transform}.

From \eqref{eq.conf_hodge}, \eqref{eq.conf_unif}, \eqref{eq.est_prod_elem}, and \eqref{eq.conf_besov}, we have
\[ \| \mc{D}_1 V \|_{ B^{2, \infty, -1/2}_{\ell, t, x} } \lesssim \| e^{2u} \bar{\mc{D}}_1 V \|_{ \bar{B}^{2, \infty, -1/2}_{\ell, t, x} } \lesssim \| \bar{\mc{D}}_1 V \|_{ \bar{B}^{2, \infty, -1/2}_{\ell, t, x} } \text{.} \]
Applying \eqref{eq.conf_besov} again, along with \eqref{eq.besov_nabla}, yields
\[ \| \mc{D}_1 V \|_{ B^{2, \infty, -1/2}_{\ell, t, x} } \lesssim \| V \|_{ \bar{B}^{2, \infty, 1/2}_{\ell, t, x} } \lesssim \| V \|_{ B^{2, \infty, 1/2}_{\ell, t, x} } \text{.} \qedhere \]
\end{proof}

\section{Regular Null Cones in Vacuum Spacetimes} \label{sec.nc}

In this final section, we provide an informal discussion of how the abstract formalisms in this paper apply to our main setting of interest: regular null cones on Lorentzian manifolds.
For simplicity, we only discuss geodesically foliated truncated null cones in Einstein-vacuum spacetimes, beginning from a sphere; this is the setting found in \cite{kl_rod:cg}.
Other variants of this setting (e.g., other foliations of null cones, null cones beginning from a vertex, null cones in non-vacuum spacetimes) can also be described using the formalisms of this section, although the specifics will differ very slightly from the development here.

\subsection{Geometry of Null Cones} \label{sec.nc_geom}

Let $(M, g)$ be a four-dimensional time-oriented Lorentzian manifold.
Also, assume $(M, g)$ is \emph{Einstein-vacuum}, i.e., that $\operatorname{Ric}_g \equiv 0$.
\begin{itemize}
\item Fix a compact two-dimensional submanifold $\mc{S}$ of $M$ diffeomorphic to $\Sph^2$.

\item We smoothly assign to each $x \in \mc{S}$ a future (or alternately, past) null vector $\ell_x$ at $x$ which is in addition orthogonal to $\mc{S}$.

\item For each $x \in \mc{S}$, we let $\lambda_x$ denote the future (or past) null geodesic satisfying the initial conditions $\lambda_x (0) = x$ and $\lambda_x^\prime (0) = \ell_x$.

\item Let the \emph{truncated null cone} $\mc{N}$ denote a smooth portion of the null hypersurface traced out by the congruence $\{ \lambda_x | x \in \mc{S} \}$ of null geodesics.

\item Let $t: \mc{N} \rightarrow \R$ denote the \emph{affine parameter}, which maps $\lambda_x (\tau) \in \mc{N}$, where $x \in \mc{S}$ and $\tau \in \R$, to the number $\tau$.
\end{itemize}

An important question is to determine when the above constructions are well-defined and smooth.
It turns out that (see, e.g., \cite{haw_el:gr, on:srg}) this is the case as long as $\mc{N}$ avoids both null conjugate and null cut locus points as one traverses the geodesics $\lambda_x$, $x \in \mc{S}$.
Here, we assume a priori that the above holds for $\mc{N}$.
In particular, the affine parameter $t$ is well-defined, and $t \in \mc{C}^\infty (\mc{N})$.

Identifying $\mc{S}$ with $\Sph^2$, we can now consider $\mc{N}$ as the foliation
\[ \mc{N} \simeq [0, \delta] \times \Sph^2 \text{,} \qquad \delta > 0 \text{.} \]
In addition, we define the following on $\mc{N}$:
\begin{itemize}
\item Define the vector field $L$ on $\mc{N}$ to be the tangent vector fields of the $\lambda_x$'s.
By definition, $L$ is geodesic, and $L t \equiv 1$ everywhere.
This implies that the level sets $\mc{S}_\tau$ of $t$ are Riemannian submanifolds of $\mc{N}$ and $M$.

\item Let the horizontal metric $\gamma \in \mc{C}^\infty \ul{T}^0_2 \mc{N}$ be defined as the metrics on the $\mc{S}_\tau$'s induced from the spacetime metric $g$.

\item Let $\ul{L}$ denote the conjugate null vector field on $\mc{N}$, which is orthogonal to every $\mc{S}_\tau$ and satisfies $g (L, \ul{L}) \equiv -2$.
Note that $\ul{L}$ is transverse to $\mc{N}$: for any $z \in \mc{N}$, the vector $\ul{L} |_z$ is a tangent vector for $M$, but not $\mc{N}$.
\end{itemize}
Next, we define the \emph{Ricci coefficients}, which are horizontal connection quantities that describe the derivatives of $L$ and $\ul{L}$ in the directions tangent to $\mc{N}$.
\begin{itemize}
\item Define $2$-tensors $\chi, \ul{\chi} \in \mc{C}^\infty \ul{T}^0_2 \mc{N}$ by
\[ \chi (X, Y) = g ( D_X L, Y ) \text{,} \qquad \ul{\chi} (X, Y) = g ( D_X \ul{L}, Y ) \text{,} \qquad X, Y \in \mc{C}^\infty \ul{T}^1_0 \mc{N} \text{,} \]
where $D$ is the restriction of the spacetime Levi-Civita connection to $\mc{N}$.

\item Define $\zeta \in \mc{C}^\infty \ul{T}^0_1 \mc{N}$ by
\[ \zeta (X) = \frac{1}{2} g ( D_X L, \ul{L} ) \text{,} \qquad X \in \mc{C}^\infty \ul{T}^1_0 \mc{N} \text{.} \]
\end{itemize}
Let $R$ denote the spacetime Riemann curvature tensor associated with $g$.
We define the following curvature components, which comprise the standard null decomposition of $R$ with respect to our given geodesic foliation.
\begin{alignat*}{3}
\alpha &\in \mc{C}^\infty \ul{T}^0_2 \mc{N} \text{,} &\qquad \alpha (X, Y) &= R (L, X, L, Y) \text{,} \\
\beta &\in \mc{C}^\infty \ul{T}^0_1 \mc{N} \text{,} &\qquad \beta (X) &= \frac{1}{2} R (L, X, L, \ul{L}) \text{,} \\
\rho &\in \mc{C}^\infty \mc{N} \text{,} &\qquad \rho &= \frac{1}{4} R (L, \ul{L}, L, \ul{L}) \text{,} \\
\sigma &\in \mc{C}^\infty \mc{N} \text{,} &\qquad \sigma &= \frac{1}{4} {}^\star R (L, \ul{L}, L, \ul{L}) \text{,} \\
\ul{\beta} &\in \mc{C}^\infty \ul{T}^0_1 \mc{N} \text{,} &\qquad \ul{\beta} (X) &= \frac{1}{2} R (\ul{L}, X, \ul{L}, L) \text{,} \\
\ul{\alpha} &\in \mc{C}^\infty \ul{T}^0_2 \mc{N} \text{,} &\qquad \ul{\alpha} (X, Y) &= R (\ul{L}, X, \ul{L}, Y) \text{,}
\end{alignat*}
where ${}^\star R$ denotes the left (spacetime) Hodge dual of $R$.
In the vacuum setting, these coefficients comprise all the independent components of $R$.

By a direct computation, cf. \cite{kl_rod:cg}, one can see that, with respect to $\gamma$,
\begin{equation} \label{eq.nc_sff} k = \chi \text{.} \end{equation}
We can define the evolutionary covariant derivative $\nabla_t$ as before.
One can show that this operator $\nabla_t$ coincides with the appropriate projection $\nabla_L$ onto the $\mc{S}_\tau$'s of the spacetime covariant derivative operator $D_L$.

\begin{remark}
In previous works, e.g., \cite{kl_rod:cg, kl_rod:stt, parl:bdc, shao:bdc_nv, wang:cg, wang:cgp}, the ``horizontal covariant evolution" operation was defined to be this projection of $D_L$ mentioned above.
Our definition of $\nabla_t$ here generalizes this description.
\end{remark}

The Ricci and curvature coefficients defined above are related to each other via a family of geometric differential equations, known as the \emph{null structure equations}.

\begin{proposition} \label{thm.structure}
The following structure equations hold on $\mc{N}$.
\begin{itemize}
\item Evolution equations:
\begin{align}
\label{eq.structure_ev} \nabla_t \chi_{ab} &= - \gamma^{cd} \chi_{ac} \chi_{bd} - \alpha_{ab} \text{,} \\
\notag \nabla_t \zeta_a &= - 2 \gamma^{bc} \chi_{ab} \zeta_c - \beta_a \text{,} \\
\notag \nabla_t \ul{\chi}_{ab} &= - ( \nabla_a \zeta_b + \nabla_b \zeta_a ) - \frac{1}{2} \gamma^{cd} ( \chi_{ac} \ul{\chi}_{bd} + \chi_{bc} \ul{\chi}_{ad} ) + 2 \zeta_a \zeta_b + \rho \gamma_{ab} \text{.}
\end{align}

\item Null Bianchi equations:
\begin{align}
\label{eq.structure_nb} \nabla_t \beta_a &= \mc{D}_2 \alpha_a - 2 ( \trace \chi ) \beta_a + \gamma^{bc} \zeta_b \alpha_{ac} \text{,} \\
\notag \nabla_t ( \rho + i \sigma ) &= \mc{D}_1 \beta - \frac{3}{2} ( \trace \chi ) ( \rho + i \sigma ) - ( \gamma^{ab} - i \epsilon^{ab} ) ( \zeta_a \beta_b + \frac{1}{2} \gamma^{cd} \ul{\chi}_{ac} \alpha_{bd} ) \text{,} \\
\notag \nabla_t \ul{\beta}_a &= \mc{D}_1^\ast ( \rho - i \sigma ) - ( \trace \chi ) \ul{\beta}_a + 3 \zeta_a \rho - 3 \epsilon_a{}^b \zeta_b \sigma + 2 \gamma^{bc} \ul{\hat{\chi}}_{ab} \beta_c \text{.}
\end{align}

\item Gauss-Codazzi equations:
\begin{align}
\label{eq.structure_gc} \nabla_b \chi_{ac} - \nabla_c \chi_{ab} &= - \epsilon_{bc} \epsilon_a{}^d \beta_d + \chi_{ab} \zeta_c - \chi_{ac} \zeta_b \text{,} \\
\notag \mc{K} &= - \rho + \frac{1}{2} ( \gamma^{ac} \gamma^{bd} - \gamma^{ab} \gamma^{cd} ) \chi_{ab} \ul{\chi}_{cd} \text{.}
\end{align}
\end{itemize}
\end{proposition}

For derivations of the null structure equations, see, e.g., \cite{chr_kl:stb_mink, kl_nic:stb_mink}.
For the full list of structure equations on geodesically foliated null cones in Einstein-vacuum spacetimes, see \cite{kl_rod:cg, wang:cg}.
For other foliations, see \cite{parl:bdc, shao:bdc_nv, wang:tbdc}.
These equations play a fundamental role in controlling the geometry of $\mc{N}$.

In particular, note that the curl $\mf{C}$ of $k$ is described precisely by the first equation of \eqref{eq.structure_gc}.
In fact, this explicit formula for $\mf{C}$ is essential for explaining how the estimates in this paper apply to this setting of regular null cones.

\subsection{Finite Curvature Flux} \label{sec.nc_cf}

For simplicity, we assume now that $\delta = 1$, as was done in \cite{kl_rod:cg}.
In addition, we assume now small curvature flux,
\begin{equation} \label{eq.curv_flux} \| \alpha \|_{ L^{2, 2}_{t, x} } + \| \beta \|_{ L^{2, 2}_{t, x} } + \| \rho \|_{ L^{2, 2}_{t, x} } + \| \sigma \|_{ L^{2, 2}_{t, x} } + \| \rho \|_{ L^{2, 2}_{t, x} } \leq C \text{,} \end{equation}
where $C$ is some sufficiently small constant.
\footnote{In some cases, one can also consider bounded but possibly large curvature flux, as long as the size $\delta$ of the null cone segment is sufficiently small with respect to the flux; see \cite{shao:bdc_nv}.}
We must also assume that certain quantities depending on the Ricci coefficients have similarly small values on $\mc{S}_0$.
We will not expand this point here; for details, see \cite{kl_rod:cg}.

Under the above assumptions, one has associated bounds for the Ricci coefficients $\chi$, $\zeta$, and $\ul{\chi}$.
Some examples of these bounds are as follows:
\begin{align}
\label{eq.bootstrap} \| \chi - r^{-1} \gamma \|_{ N^1_{t, x} } + \| \chi - r^{-1} \gamma \|_{ L^{\infty, 2}_{x, t} } &\leq \Delta \text{,} \\
\notag \| \zeta \|_{ N^1_{t, x} } + \| \zeta \|_{ L^{\infty, 2}_{x, t} } &\leq \Delta \text{,} \\
\notag \| \ul{\chi} + r^{-1} \gamma \|_{ L^{2, \infty}_{x, t} } &\leq \Delta \text{.}
\end{align}
Here, $r \in C^\infty \mc{N}$ maps each $(\tau, x) \in \mc{N}$ to the ``radius" of $\mc{S}_\tau$, i.e.,
\[ 4 \pi \cdot r^2 |_{ (\tau, x) } = | \mc{S}_\tau | \text{.} \]
The constant $\Delta$ in \eqref{eq.bootstrap} is some positive number related to $C$.
\footnote{In the setting of \cite{kl_rod:cg}, this $\Delta$ must be sufficiently small, but large with respect to $C$.}

One generally reaches the estimates \eqref{eq.bootstrap} in two ways:
\begin{itemize}
\item Under appropriate conditions, these bounds are consequences of \eqref{eq.curv_flux}.
For example, these are some of the main estimates proved in \cite{kl_rod:cg}.

\item These bounds are also used as bootstrap assumptions.
This was the strategy adopted in \cite{kl_rod:cg} in order to prove the main estimates \eqref{eq.bootstrap}.
In this process, one assumes \eqref{eq.bootstrap} in a standard continuity argument and proceeds to prove a strictly improved version of \eqref{eq.bootstrap}, e.g., with $\Delta$ replaced by $\Delta/2$.
\end{itemize}
The regularity assumptions \eqref{eqr.sff_tr}-\eqref{eqr.sffcurl} and their consequences reveal why the bootstrap assumptions mentioned above are fundamental to the analysis in \cite{kl_rod:cg}.
Indeed, in order to make use of many of the calculus estimates required in this analysis, one had to first show that such regularity conditions hold on $\mc{N}$.
That these conditions hold follows only from appropriate bootstrap assumptions, including \eqref{eq.bootstrap}.

Below, we shall briefly sketch how the abstract assumptions used throughout this section follow from \eqref{eq.bootstrap}.
In other words, we will justify that the estimates proved in this paper apply to our intended settings.
\footnote{Although we only describe the setting of \cite{kl_rod:cg} here, these arguments, with a few minor adjustments, can also be applied to other variants of the general null cone problem.}

First of all, the assumptions on $\chi$ in the first line of \eqref{eq.bootstrap} along with the observation \eqref{eq.nc_sff} imply the estimates \eqref{eqr.sff_tr}-\eqref{eqr.sffd}.
Moreover, since $\mc{S}$ is compact, as it is diffeomorphic to $\Sph^2$, then the regularity conditions from Section \ref{sec.geom_reg} hold for some constants $C, N$ depending on the geometry of $(\mc{S}, \gamma [0])$.

We now consider the condition \eqref{eqr.sffcurl}, which is the new observation not applied in previous works \cite{kl_rod:cg, parl:bdc, shao:bdc_nv, wang:cg, wang:cgp} on various null cone settings.
Recall that since $k = \chi$, then $\mf{C}$ is given precisely by the Codazzi equation in \eqref{eq.structure_gc}.
Combining this with the second evolution equation in \eqref{eq.structure_ev}, we see that
\[ \cint^t_0 \mf{C} \sim \zeta + \xi \text{,} \]
where $\xi$ represents ``lower-order" terms obtained by integrating the remaining terms in the above structure equations.
From the bootstrap conditions and \eqref{eq.nsob_ineq},
\footnote{This also requires appropriate control for the initial value of $\zeta$.}
\[ \| \zeta \|_{ L^{4, \infty}_{x, t} } \lesssim \Delta \text{.} \]
On top of this, with a bit of extra work, the ``lower-order" terms $\xi$ can also be controlled in the $N^1$-norm.
Combining the above with \eqref{eq.nsob_ineq}, we have
\[ \| \cint^t_0 \mf{C} \|_{ L^{4, \infty}_{x, t} } \lesssim \Delta \text{,} \]
which establishes \eqref{eqr.sffcurl}.
Thus, $(\mc{N}, \gamma)$ indeed satisfies the \ass{F2}{C, N, B} condition.

The above heuristic argument shows that the estimates found in Sections \ref{sec.fol}, \ref{sec.cfol}, and \ref{sec.thm} are applicable to the null cone setting of \cite{kl_rod:cg}.
It remains only to demonstrate the \ass{K}{} condition and hence validate the elliptic estimates of Section \ref{sec.curv}.

From a direct calculation, one can see that $\rho$ can be expressed as
\[ \rho = - \gamma^{ab} \nabla_a \zeta_b + E \text{,} \]
where $E$ is a ``sufficiently nice" error term.
This, combined with the second equation in \eqref{eq.structure_gc} (i.e., the Gauss equation), yields the relation
\[ \mc{K} - \frac{1}{r} = \gamma^{ab} \nabla_a \zeta_b + \xi \text{,} \]
where $\xi$ is an error term that can be shown to be $L^2_x$-controlled.
Furthermore, from the bootstrap assumptions and from \eqref{eq.nsob_trace}, one can show that $\zeta$ can be controlled in Sobolev-type norms involving $1/2$ horizontal derivatives.

Since $r$ is comparable to $1$ in the setting of \cite{kl_rod:cg} (this is again a consequence of the bootstrap assumptions), we see that $(\mc{N}, \gamma)$ indeed satisfies the \ass{K}{} condition.
As a result, Section \ref{sec.curv} now provides both an improvement and a dramatic simplification of various Hodge and elliptic estimates involving Besov norms.

\appendix

\section{Estimates on Euclidean Spaces} \label{sec.eucl}

In this appendix, we establish Euclidean analogues of many of the product estimates in this paper.
For the sake of consistency, we retain much of the formalisms detailed in Sections \ref{sec.geom} and \ref{sec.fol}.
More specifically, here we examine the special cases
\[ \mc{S} = \R^2 \text{,} \qquad \mc{N} = [0, \delta] \times \R^2 \text{,} \]
with $h$ and $\gamma$ representing the standard Euclidean metric on $\R^2$.

In this section, we will be dealing mostly with two families of functions:
\begin{itemize}
\item Let $\mc{S}_x \R^2$ denote the ``Schwartz space" of smooth functions on $\R^2$ for which all its partial derivatives are rapidly decreasing.

\item Let $\mc{C}^\infty_t \mc{S}_x \mc{N}$ denote the space of smooth functions on $\mc{N}$ for which all of its partial derivatives, when restricted to any level set of $t$, belong in $\mc{S}_x \R^2$.
\end{itemize}
Moreover, throughout the section, we will assume $f, g \in \mc{S}_x \R^2$ and $\phi, \psi \in \mc{C}^\infty_t \mc{S}_x \mc{N}$.

We will adopt the notations of the previous sections whenever convenient.
In particular, this applies to many of the integral norms we use here, such as the norms in Sections \ref{sec.fol} and \ref{sec.cfol}, which can be applied to any $\phi \in \mc{C}^\infty_t \mc{S}_x \mc{N}$.

For consistency, we also adopt the following notational conventions:
\begin{itemize}
\item Let $\partial$ denote the usual Euclidean gradient on $\R^2$.

\item Let $\partial_t$ denote the $t$-partial derivative on $\mc{N} = [0, \delta] \times \R^2$.
\end{itemize}
Thus, the $N^1_{t, x}$-norm is given by
\[ \| \phi \|_{ N^1_{t, x} } = \| \partial_t \phi \|_{ L^{2, 2}_{t, x} } + \| \partial \phi \|_{ L^{2, 2}_{t, x} } + \| \phi \|_{ L^{2, 2}_{t, x} } \text{.} \]

\subsection{Littlewood-Paley Theory} \label{sec.eucl_lp}

We very briefly review some basic elements of the classical dyadic Littlewood-Paley theory on $\R^2$.
\begin{itemize}
\item Fix a cutoff function $\varsigma \in \mc{S}_x \R^2$, supported in the annulus $1/2 \leq | \xi | \leq 2$, and define $\varsigma_k \in \mc{S}_x \R^2$, $k \in \Z$, by $\varsigma_k ( \xi ) = \varsigma ( 2^{-k} \xi )$, such that
\[ \sum_{k \in \Z} \varsigma_k = \chi_{ \R^2 \setminus \{ (0, 0) \} } \text{.} \]

\item For any $k \in \Z$, the L-P ``projection" $E_k$ is defined as a Fourier multiplier operator, $\mc{F} ( E_k f ) = \varsigma_k \mc{F} f$, where $\mc{F}$ is the usual Fourier transform on $\R^2$.

\item In addition, for any $k \in \Z$, we can define the operators
\[ E_{\sim k} = E_{k - 1} + E_k + E_{k + 1} \text{,} \qquad E_{< k} = \sum_{ l = -\infty }^{k - 1} E_l \text{,} \qquad E_{\lesssim k} = E_{< k + 1} \text{.} \]
\end{itemize}

The properties satisfied by these multipliers are well-known; see, e.g., \cite{kl:intro_anal, tao:disp_eq}.
Here, we list some basic properties we will use later.
\begin{itemize}
\item \emph{Almost Orthogonality:} If $k_1, k_2 \in \Z$ and $| k_1 - k_2 | > 1$, then $E_{k_1} E_{k_2} \equiv 0$.
This implies for any $k \in \Z$ the self-replication properties
\[ E_k = E_k E_{\sim k} \text{,} \qquad E_{< k} = E_{< k} E_{\lesssim k} \text{.} \]

\item \emph{Boundedness:} For any $q \in [1, \infty]$ and $k \in \Z$,
\begin{equation} \label{eqc.glp_bdd} \| E_k f \|_{ L^q_x } + \| E_{< k} f \|_{ L^q_x } + \| E_{\geq k} f \|_{ L^q_x } \lesssim \| f \|_{ L^q_x } \text{.} \end{equation}

\item \emph{Finite Band:} For any $q \in [1, \infty]$ and $k \in \Z$,
\begin{align}
\label{eqc.glp_fb} \| \partial E_{< 0} f \|_{ L^q_x } + 2^{-k} \| \partial E_k f \|_{ L^q_x } \lesssim \| f \|_{ L^q_x } \text{,} \qquad 2^k \| E_k f \|_{ L^q_x } \lesssim \| \partial f \|_{ L^q_x } \text{.}
\end{align}

\item \emph{Bernstein Inequalities:} If $q, q^\prime \in [1, \infty]$, $q \leq q^\prime$, and $k \in \Z$, then
\begin{equation} \label{eqc.glp_bernstein} \| E_{< 0} f \|_{ L^{q^\prime}_x } \lesssim \| f \|_{ L^q_x } \text{,} \qquad \| E_k f \|_{ L^{q^\prime}_x } \lesssim 2^{ 2 k ( \frac{1}{q} - \frac{1}{q^\prime} ) } \| f \|_{ L^q_x } \text{.} \end{equation}
\end{itemize}

Another important point is to observe that we can write $2^k E_k = \partial \tilde{E}_k$ for any $k \in \Z$, where the operator $\tilde{E}_k$ is like $E_k$ but is constructed from a slightly different Fourier multiplier.
This new operator $\tilde{E}_k$ satisfies the many of the same properties as $E_k$, including having the same Fourier support as $E_k$ and sharing the basic estimates \eqref{eqc.glp_bdd}-\eqref{eqc.glp_bernstein} (although with different constants).

We extend these L-P operators to act on functions on $\mc{N}$ by taking the appropriate Fourier localization of $\phi$ with respect to only the spatial components.

L-P operators provide a practical way to express fractional Sobolev and Besov norms on $\R^2$.
Given $s \in \R$, the standard $L^2$-based Sobolev norms satisfy
\[ \| f \|_{ H^s_x }^2 = \| \langle \partial \rangle^s f \|_{ L^2_x }^2 \simeq_s \sum_{k \geq 0} 2^{2sk} \| E_k f \|_{ L^2_x }^2 + \| E_{< 0} f \|_{ L^2_x }^2 \text{.} \]
Next, given any $a \in [1, \infty)$ and $s \in \R$, we define the general Besov-type norms
\begin{align*}
\| f \|_{ B^{a, s}_{\ell, x} }^a &= \| f \|_{ B^s_{2, a} }^a = \sum_{k \geq 0} 2^{ask} \| E_k f \|_{ L^2_x }^a + \| E_{< 0} f \|_{ L^2_x }^a \text{,} \\
\| f \|_{ B^{\infty, s}_{\ell, x} } &= \| f \|_{ B^s_{2, \infty} } = \max \left( \sup_{k \geq 0} \| E_k f \|_{ L^2_x }, \| E_{< 0} f \|_{ L^2_x } \right) \text{.}
\end{align*}
Similarly, if in addition $p \in [1, \infty]$, then we can define the norms
\begin{align*}
\| \phi \|_{ B^{a, p, s}_{\ell, t, x} } &= \sum_{k \geq 0} 2^{ask} \| E_k \phi \|_{ L^{p, 2}_{t, x} }^a + \| E_{< 0} \phi \|_{ L^{p, 2}_{t, x} }^a \text{,} \\
\| \phi \|_{ B^{\infty, p, s}_{\ell, t, x} } &= \max \left( \sup_{k \geq 0} \| E_k \phi \|_{ L^{p, 2}_{t, x} }, \| E_{< 0} \phi \|_{ L^{p, 2}_{t, x} } \right) \text{.}
\end{align*}

For example, recall that one has the following sharp embedding estimate:
\begin{equation} \label{eq.besov_sh} \| \phi \|_{ L^{\infty, \infty}_{t, x} } \lesssim \| \phi \|_{ B^{1, \infty, 1}_{\ell, t, x} } \lesssim \| \partial \phi \|_{ B^{0, \infty, 1}_{\ell, t, x} } + \| \phi \|_{ L^{\infty, 2}_{t, x} } \text{.} \end{equation}

\subsection{Non-Integrated Product Estimates I} \label{sec.eucl_nint}

We begin the process of establishing the Euclidean analogues of our main bilinear product estimates.
We start here with the simpler non-integrated bilinear product estimates, for which we will require an L-P decomposition of only one of the factors.

\begin{lemma} \label{thmc.intertwining_prod}
Fix $k, l \geq 0$, let $f, g \in \mc{S}_x \R^2$, and let $\phi, \psi \in \mc{C}^\infty_t \mc{S}_x \mc{N}$.
\begin{itemize}
\item The following estimates hold:
\begin{align}
\label{eqc.intertwining_prod_elem} \| E_k ( f \cdot E_l g ) \|_{ L^2_x } &\lesssim 2^{ - |k - l| } ( \| \partial f \|_{ L^2_x } + \| f \|_{ L^\infty_x } ) \| E_{\sim l} g \|_{ L^2_x } \text{,} \\
\notag \| E_k ( f \cdot E_{< 0} g ) \|_{ L^2_x } &\lesssim 2^{ - k } ( \| \partial f \|_{ L^2_x } + \| f \|_{ L^\infty_x } ) \| E_{\lesssim 0} g \|_{ L^2_x } \text{,} \\
\notag \| E_{< 0} ( f \cdot E_l g ) \|_{ L^2_x } &\lesssim 2^{ - l } ( \| \partial f \|_{ L^2_x } + \| f \|_{ L^\infty_x } ) \| E_{\sim l} g \|_{ L^2_x } \text{,} \\
\notag \| E_{< 0} ( f \cdot E_{< 0} g ) \|_{ L^2_x } &\lesssim \| f \|_{ L^\infty_x } \| E_{\lesssim 0} g \|_{ L^2_x } \text{.} 
\end{align}

\item If $\psi$ is $t$-parallel, then
\footnote{In other words, if $\psi (\tau, x) = \psi (0, x)$ for every $\tau \in [0, \delta]$ and $x \in \R^2$.}
\begin{align}
\label{eqc.intertwining_prod_imp} \| E_k ( \phi \cdot E_l \psi ) \|_{ L^{2, 2}_{t, x} } &\lesssim 2^{ - |k - l| } ( \| \partial \phi \|_{ L^{2, 2}_{t, x} } + \| \phi \|_{ L^{\infty, 2}_{x, t} } ) \| E_{\sim l} \psi [0] \|_{ L^2_x } \text{,} \\
\notag \| E_k ( \phi \cdot E_{< 0} \psi ) \|_{ L^{2, 2}_{t, x} } &\lesssim 2^{ - k } ( \| \partial \phi \|_{ L^{2, 2}_{t, x} } + \| \phi \|_{ L^{\infty, 2}_{x, t} } ) \| E_{\lesssim 0} \psi [0] \|_{ L^2_x } \text{,} \\
\notag \| E_{< 0} ( \phi \cdot E_l \psi ) \|_{ L^{2, 2}_{t, x} } &\lesssim 2^{ - l } ( \| \partial \phi \|_{ L^{2, 2}_{t, x} } + \| \phi \|_{ L^{\infty, 2}_{x, t} } ) \| E_{\sim l} \psi [0] \|_{ L^2_x } \text{,} \\
\notag \| E_{< 0} ( \phi \cdot E_{< 0} \psi ) \|_{ L^{2, 2}_{t, x} } &\lesssim \| \phi \|_{ L^{\infty, 2}_{x, t} } \| E_{\lesssim 0} \psi [0] \|_{ L^2_x } \text{.}
\end{align}
\end{itemize}
\end{lemma}

\begin{proof}
We will only prove the first estimate in each set, as the estimates containing low-frequency projections tend to be easier and can be proved similarly.

Suppose first that $l \geq k$.
In the case of \eqref{eqc.intertwining_prod_elem}, we obtain
\begin{align*}
\| E_k ( f P_l g ) \|_{ L^2_x } &\lesssim 2^{-l} \| E_k ( f \partial \tilde{E}_l g ) \|_{ L^2_x } \lesssim 2^{-l} [ \| E_k \partial ( f \tilde{E}_l g ) \|_{ L^2_x } + \| E_k ( \partial f \cdot \tilde{E}_l g ) \|_{ L^2_x } ] \text{.}
\end{align*}
The two terms on the right-hand side are treated using \eqref{eqc.glp_fb} and \eqref{eqc.glp_bernstein}:
\begin{align*}
\| E_k ( f E_l g ) \|_{ L^2_x } &\lesssim 2^{k - l} [ \| f \tilde{E}_l g \|_{ L^2_x } + \| \partial f \cdot \tilde{E}_l g \|_{ L^1_x } ] \\
&\lesssim 2^{ - |k - l| } ( \| f \|_{ L^\infty_x } + \| \partial f \|_{ L^2_x } ) \| E_{\sim l} g \|_{ L^2_x } \text{.}
\end{align*}
Similarly, for \eqref{eqc.intertwining_prod_imp}, we have
\begin{align*}
\| E_k ( \phi E_l \psi ) \|_{ L^{2, 2}_{t, x} } &\lesssim 2^{-l} [ \| E_k \partial ( \phi \tilde{E}_l \psi ) \|_{ L^{2, 2}_{t, x} } + \| E_k ( \partial \phi \cdot \tilde{E}_l \psi ) \|_{ L^{2, 2}_{t, x} } ] \\
&\lesssim 2^{k - l} [ \| \phi \tilde{E}_l \psi \|_{ L^{2, 2}_{t, x} } + \| \partial \phi \cdot \tilde{E}_l \psi \|_{ L^{2, 1}_{t, x} } ] \\
&\lesssim 2^{ - |k - l| } ( \| \phi \|_{ L^{\infty, 2}_{x, t} } + \| \partial \phi \|_{ L^{2, 2}_{t, x} } ) \| E_{\sim l} \psi \|_{ L^{2, \infty}_{x, t} } \text{,}
\end{align*}
where we again applied \eqref{eqc.glp_fb} and \eqref{eqc.glp_bernstein}.
Since $E_{\sim l} \psi$ is independent of the $t$-variable,
\[ \| E_{\sim l} \psi \|_{ L^{2, \infty}_{x, t} } = \| E_{\sim l} \psi [0] \|_{ L^2_x } \text{.} \]

Next, consider the remaining case $l < k$.
First, for \eqref{eqc.intertwining_prod_elem}, we apply \eqref{eqc.glp_fb}:
\begin{align*}
\| E_k ( f E_l g ) \|_{ L^2_x } &\lesssim 2^{-k} \| \partial ( f E_l g ) \|_{ L^2_x } \lesssim 2^{-k} ( \| \partial f \|_{ L^2_x } \| E_l g \|_{ L^\infty_x } + \| f \|_{ L^\infty_x } \| \partial E_l g \|_{ L^2_x } ) \text{.}
\end{align*}
By \eqref{eqc.glp_fb} and \eqref{eqc.glp_bernstein}, then
\begin{align*}
\| E_k ( f E_l g ) \|_{ L^2_x } &\lesssim 2^{-k + l} ( \| \partial f \|_{ L^2_x } + \| f \|_{ L^\infty_x } ) \| E_{\sim l} g \|_{ L^2_x } \\
&= 2^{- |k - l|} ( \| \partial f \|_{ L^2_x } + \| f \|_{ L^\infty_x } ) \| E_{\sim l} g \|_{ L^2_x } \text{.}
\end{align*}
Similarly, for \eqref{eqc.intertwining_prod_imp}, we apply \eqref{eqc.glp_fb} and \eqref{eqc.glp_bernstein} once again to obtain
\begin{align*}
\| E_k ( \phi E_l \psi ) \|_{ L^{2, 2}_{t, x} } &\lesssim 2^{-k} ( \| \partial \phi \|_{ L^{2, 2}_{t, x} } \| E_l \psi \|_{ L^{\infty, \infty}_{t, x} } + \| \phi \|_{ L^{\infty, 2}_{x, t} } \| \partial E_l \psi \|_{ L^{2, \infty}_{x, t} } ) \\
&\lesssim 2^{- |k - l|} ( \| \partial \phi \|_{ L^{2, 2}_{t, x} } + \| \phi \|_{ L^{\infty, 2}_{x, t} } ) \| E_{\sim l} \psi [0] \|_{ L^2_x } \text{.}
\end{align*}
Once again, we used that $\psi$ and $\partial E_{\sim l} \psi$ are independent of $t$.
\end{proof}

We now apply Lemma \ref{thmc.intertwining_prod} to prove the desired product estimates.

\begin{proposition} \label{thmc.est_prod}
Let $a \in [1, \infty]$, and let $s \in (-1, 1)$.
\begin{itemize}
\item If $p \in [1, \infty]$, and if $\phi, \psi \in \mc{C}^\infty_t \mc{S}_x \mc{N}$, then
\begin{align}
\label{eqc.est_prod_elem} \| \phi \cdot \psi \|_{ B^{a, p, s}_{\ell, t, x} } &\lesssim_s ( \| \partial \phi \|_{ L^{\infty, 2}_{t, x} } + \| \phi \|_{ L^{\infty, \infty}_{t, x} } ) \| \psi \|_{ B^{a, p, s}_{\ell, t, x} } \text{.}
\end{align}

\item If $\phi, \psi \in \mc{C}^\infty_t \mc{S}_x \mc{N}$, and if $\psi$ is $t$-parallel, then
\begin{align}
\label{eqc.est_prod_imp} \| \phi \cdot \psi \|_{ B^{a, 2, s}_{\ell, t, x} } &\lesssim_s ( \| \partial \phi \|_{ L^{2, 2}_{t, x} } + \| \phi \|_{ L^{\infty, 2}_{x, t} } ) \| \psi [0] \|_{ B^{a, s}_{\ell, x} } \text{.}
\end{align}
\end{itemize}
\end{proposition}

\begin{proof}
It suffices to prove the above for $a = 1$ and $a = \infty$, as the remaining cases then follow from standard interpolation techniques.

For \eqref{eqc.est_prod_elem}, if $a = 1$, we decompose and estimate using \eqref{eqc.intertwining_prod_elem}:
\begin{align*}
\| \phi \psi \|_{ B^{1, p, s}_{\ell, t, x} } &\lesssim \sum_{k, l \geq 0} 2^{s k} \| E_k ( \phi E_l \psi ) \|_{ L^{p, 2}_{t, x} } + \sum_{k \geq 0} 2^{s k} \| E_k ( \phi E_{< 0} \psi ) \|_{ L^{p, 2}_{t, x} } \\
&\qquad + \sum_{l \geq 0} \| E_{< 0} ( \phi E_l \psi ) \|_{ L^{p, 2}_{t, x} } + \| E_{< 0} ( \phi E_{< 0} \psi ) \|_{ L^{p, 2}_{t, x} } \\
&\lesssim ( \| \partial \phi \|_{ L^{\infty, 2}_{t, x} } + \| \phi \|_{ L^{\infty, \infty}_{t, x} } ) \sum_{l \geq 0} \paren{ \sum_{k \geq 0} 2^{s k} 2^{- |k - l|} + 2^{-l} } \| E_{\sim l} \psi \|_{ L^{p, 2}_{t, x} } \\
&\qquad + ( \| \partial \phi \|_{ L^{\infty, 2}_{t, x} } + \| \phi \|_{ L^{\infty, \infty}_{t, x} } ) \paren{ \sum_{k \geq 0} 2^{(s - 1) k} + 1 } \| E_{\lesssim 0} \psi \|_{ L^{p, 2}_{t, x} } \\
&\lesssim ( \| \partial \phi \|_{ L^{\infty, 2}_{t, x} } + \| \phi \|_{ L^{\infty, \infty}_{t, x} } ) \| \psi \|_{ B^{1, p, s}_{\ell, t, x} } \text{.}
\end{align*}
Similarly, if $a = \infty$, then again by \eqref{eqc.intertwining_prod_elem},
\begin{align*}
\| \phi \psi \|_{ B^{\infty, p, s}_{\ell, t, x} } &\lesssim ( \| \partial \phi \|_{ L^{\infty, 2}_{t, x} } + \| \phi \|_{ L^{\infty, \infty}_{t, x} } ) \sup_{k \geq 0} \brak{ \sum_{l \geq 0} 2^{s (k - l)} 2^{- |k - l|} + 2^{(s - 1) k} }  \| \psi \|_{ B^{\infty, p, s}_{\ell, t, x} } \\
&\qquad + ( \| \partial \phi \|_{ L^{\infty, 2}_{t, x} } + \| \phi \|_{ L^{\infty, \infty}_{t, x} } ) \paren{ \sum_{l \geq 0} 2^{-l} + 1 } \| \psi \|_{ B^{\infty, p, s}_{\ell, t, x} } \\
&\lesssim ( \| \partial \phi \|_{ L^{\infty, 2}_{t, x} } + \| \phi \|_{ L^{\infty, \infty}_{t, x} } ) \| \psi \|_{ B^{\infty, p, s}_{\ell, t, x} } \text{.}
\end{align*}

The corresponding estimates for \eqref{eqc.est_prod_imp} are similarly proved, using \eqref{eqc.intertwining_prod_imp}.
\end{proof}

\subsection{Integrated Product Estimates I} \label{sec.eucl_int}

Next, we look at simple \emph{integrated} bilinear product estimates, which are similar to those in Theorem \ref{thmc.est_prod}, but also contain the $t$-integral operator $\cint^t_0$.
The key steps are similar to those in Section \ref{sec.eucl_nint}.

\begin{remark}
These integrated estimates, at least in the case $a = 1$ and $s = 0$, can also be found in \cite[Sect. 3]{kl_rod:stt}.
We include them here for completeness.
\end{remark}

\begin{lemma} \label{thmc.intertwining_trace}
Fix $k, l \geq 0$, and let $\phi, \psi \in C^\infty_t \mc{S}_x \mc{N}$.
\begin{itemize}
\item The following estimates hold:
\begin{align}
\label{eqc.intertwining_trace_sh} \| E_k \cint^t_0 ( \phi \cdot E_l \psi ) \|_{ L^{\infty, 2}_{t, x} } &\lesssim 2^{ - |k - l| } ( \| \partial \phi \|_{ L^{2, 2}_{t, x} } + \| \phi \|_{ L^{\infty, 2}_{x, t} } ) \| E_{\sim l} \psi \|_{ L^{2, 2}_{t, x} } \text{,} \\
\notag \| E_k \cint^t_0 ( \phi \cdot E_{< 0} \psi ) \|_{ L^{\infty, 2}_{t, x} } &\lesssim 2^{ - k } ( \| \partial \phi \|_{ L^{2, 2}_{t, x} } + \| \phi \|_{ L^{\infty, 2}_{x, t} } ) \| E_{\lesssim 0} \psi \|_{ L^{2, 2}_{t, x} } \text{,} \\
\notag \| E_{< 0} \cint^t_0 ( \phi \cdot E_l \psi ) \|_{ L^{\infty, 2}_{t, x} } &\lesssim 2^{ - l } ( \| \partial \phi \|_{ L^{2, 2}_{t, x} } + \| \phi \|_{ L^{\infty, 2}_{x, t} } ) \| E_{\sim l} \psi \|_{ L^{2, 2}_{t, x} } \text{,} \\
\notag \| E_{< 0} \cint^t_0 ( \phi \cdot E_{< 0} \psi ) \|_{ L^{\infty, 2}_{t, x} } &\lesssim \| \phi \|_{ L^{\infty, 2}_{x, t} } \| E_{\lesssim 0} \psi \|_{ L^{2, 2}_{t, x} } \text{.}
\end{align}

\item The following estimates hold:
\begin{align}
\label{eqc.intertwining_trace_shp} \| E_k ( \phi \cdot \cint^t_0 E_l \psi ) \|_{ L^{2, 2}_{t, x} } &\lesssim 2^{ - |k - l| } ( \| \partial \phi \|_{ L^{2, 2}_{t, x} } + \| \phi \|_{ L^{\infty, 2}_{x, t} } ) \| E_{\sim l} \psi \|_{ L^{1, 2}_{t, x} } \text{,} \\
\notag \| E_k ( \phi \cdot \cint^t_0 E_{< 0} \psi ) \|_{ L^{2, 2}_{t, x} } &\lesssim 2^{ - k } ( \| \partial \phi \|_{ L^{2, 2}_{t, x} } + \| \phi \|_{ L^{\infty, 2}_{x, t} } ) \| E_{\lesssim 0} \psi \|_{ L^{1, 2}_{t, x} } \text{,} \\
\notag \| E_{< 0} ( \phi \cdot \cint^t_0 E_l \psi ) \|_{ L^{2, 2}_{t, x} } &\lesssim 2^{ - l } ( \| \partial \phi \|_{ L^{2, 2}_{t, x} } + \| \phi \|_{ L^{\infty, 2}_{x, t} } ) \| E_{\sim l} \psi \|_{ L^{1, 2}_{t, x} } \text{,} \\
\notag \| E_{< 0} ( \phi \cdot \cint^t_0 E_{< 0} \psi ) \|_{ L^{2, 2}_{t, x} } &\lesssim \| \phi \|_{ L^{\infty, 2}_{x, t} } \| E_{\lesssim 0} \psi \|_{ L^{1, 2}_{t, x} } \text{.}
\end{align}
\end{itemize}
\end{lemma}

\begin{proof}
Again, we only prove the first estimates in each set.

Assume first that $l \geq k$.
For \eqref{eqc.intertwining_trace_sh}, we have
\begin{align*}
\| E_k \cint^t_0 ( \phi E_l \psi ) \|_{ L^{\infty, 2}_{t, x} } &\lesssim 2^{-l} [ \| E_k \partial \cint^t_0 ( \phi \tilde{E}_l \psi ) \|_{ L^{\infty, 2}_{t, x} } + \| E_k \cint^t_0 ( \partial \phi \cdot \tilde{E}_l \psi ) \|_{ L^{\infty, 2}_{t, x} } ] \\
&\lesssim 2^{k - l} [ \| \cint^t_0 ( \phi \tilde{E}_l \psi ) \|_{ L^{2, \infty}_{x, t} } + \| \cint^t_0 ( \partial \phi \cdot \tilde{E}_l \psi ) \|_{ L^{1, \infty}_{x, t} } ] \text{,}
\end{align*}
where in the last step, we applied \eqref{eqc.glp_fb} and \eqref{eqc.glp_bernstein}.
By the definition of $\cint^t_0$,
\begin{align*}
\| E_k \cint^t_0 ( \phi E_l \psi ) \|_{ L^{\infty, 2}_{t, x} } &\lesssim 2^{k - l} [ \| \phi \tilde{E}_l \psi \|_{ L^{2, 1}_{x, t} } + \| \partial \phi \cdot \tilde{E}_l \psi \|_{ L^{1, 1}_{t, x} } ] \\
&\lesssim 2^{ - |k - l| } ( \| \phi \|_{ L^{\infty, 2}_{x, t} } + \| \partial \phi \|_{ L^{2, 2}_{t, x} } ) \| E_{\sim l} \psi \|_{ L^{2, 2}_{t, x} } \text{,}
\end{align*}
as desired.
Similarly, in the case of \eqref{eqc.intertwining_trace_shp}, we have
\begin{align*}
\| E_k ( \phi \cint^t_0 E_l \psi ) \|_{ L^{2, 2}_{t, x} } &\lesssim 2^{-l} [ \| E_k \partial ( \phi \cint^t_0 \tilde{E}_l \psi ) \|_{ L^{2, 2}_{t, x} } + \| E_k ( \partial \phi \cint^t_0 \tilde{E}_l \psi ) \|_{ L^{2, 2}_{t, x} } ] \\
&\lesssim 2^{k - l} ( \| \phi \cint^t_0 \tilde{E}_l \psi \|_{ L^{2, 2}_{t, x} } + \| \partial \phi \cint^t_0 \tilde{E}_l \psi \|_{ L^{2, 1}_{t, x} } ) \\
&\lesssim 2^{k - l} ( \| \phi \|_{ L^{\infty, 2}_{x, t} } + \| \partial \phi \|_{ L^{2, 2}_{t, x} } ) \| \cint^t_0 \tilde{E}_l \psi \|_{ L^{2, \infty}_{x, t} } \text{.}
\end{align*}
By the definition of $\cint^t_0$ and the Minkowski integral inequality, we obtain
\[ \| E_k ( \phi \cint^t_0 E_l \psi ) \|_{ L^{2, 2}_{t, x} } \lesssim 2^{ - |k - l| } ( \| \phi \|_{ L^{\infty, 2}_{x, t} } + \| \partial \phi \|_{ L^{2, 2}_{t, x} } ) \| E_{\sim l} \psi \|_{ L^{1, 2}_{t, x} } \text{.} \]

Next, assume $l < k$.
For \eqref{eqc.intertwining_trace_sh}, first
\begin{align*}
\| E_k \cint^t_0 ( \phi E_l \psi ) \|_{ L^{\infty, 2}_{t, x} } &\lesssim 2^{-k} [ \| \cint^t_0 ( \partial \phi \cdot E_l \psi ) \|_{ L^{\infty, 2}_{t, x} } + \| \cint^t_0 ( \phi \partial E_l \psi ) \|_{ L^{\infty, 2}_{t, x} } ] \\
&\lesssim 2^{-k} ( \| \partial \phi \cdot E_l \psi \|_{ L^{2, 1}_{x, t} } + \| \phi \partial E_l \psi \|_{ L^{2, 1}_{x, t} } ) \text{.}
\end{align*}
Applying \eqref{eqc.glp_fb} and \eqref{eqc.glp_bernstein} yet again yields
\begin{align*}
\| E_k \cint^t_0 ( \phi E_l \psi ) \|_{ L^{\infty, 2}_{t, x} } &\lesssim 2^{-k} ( \| \partial \phi \|_{ L^{2, 2}_{t, x} } \| E_l \psi \|_{ L^{2, \infty}_{t, x} } + \| \phi \|_{ L^{\infty, 2}_{x, t} } \| \partial E_l \psi \|_{ L^{2, 2}_{t, x} } ) \\
&\lesssim 2^{- |k - l|} ( \| \partial \phi \|_{ L^{2, 2}_{t, x} } + \| \phi \|_{ L^{\infty, 2}_{x, t} } ) \| E_{\sim l} \psi \|_{ L^{2, 2}_{t, x} } \text{.}
\end{align*}
Similarly, for \eqref{eqc.intertwining_trace_shp}, we have
\begin{align*}
\| E_k ( \phi \cint^t_0 E_l \psi ) \|_{ L^{2, 2}_{t, x} } &\lesssim 2^{-k} [ \| \partial \phi \cdot \cint^t_0 E_l \psi \|_{ L^{2, 2}_{t, x} } + \| \phi \cint^t_0 \partial E_l \psi \|_{ L^{2, 2}_{t, x} } ] \\
&\lesssim 2^{-k} ( \| \partial \phi \|_{ L^{2, 2}_{t, x} } \| \cint^t_0 E_l \psi \|_{ L^{\infty, \infty}_{x, t} } + \| \phi \|_{ L^{\infty, 2}_{x, t} } \| \cint^t_0 \partial E_l \psi \|_{ L^{2, \infty}_{x, t} } ) \\
&\lesssim 2^{-k} ( \| \partial \phi \|_{ L^{2, 2}_{t, x} } \| E_l \psi \|_{ L^{\infty, 1}_{x, t} } + \| \phi \|_{ L^{\infty, 2}_{x, t} } \| \partial E_l \psi \|_{ L^{2, 1}_{x, t} } ) \\
&\lesssim 2^{- |k - l|} ( \| \partial \phi \|_{ L^{2, 2}_{t, x} } + \| \phi \|_{ L^{\infty, 2}_{x, t} } ) \| E_{\sim l} \psi \|_{ L^{1, 2}_{t, x} } \text{.} \qedhere
\end{align*}
\end{proof}

\begin{proposition} \label{thmc.est_trace}
For any $a \in [1, \infty]$, $s \in (-1, 1)$, and $\phi, \psi \in C^\infty_t \mc{S}_x \mc{N}$,
\begin{align}
\label{eqc.est_trace_sh} \| \cint^t_0 ( \phi \cdot \psi ) \|_{ B^{a, \infty, s}_{\ell, t, x} } &\lesssim_s ( \| \partial \phi \|_{ L^{2, 2}_{t, x} } + \| \phi \|_{ L^{\infty, 2}_{x, t} } ) \| \psi \|_{ B^{a, 2, s}_{\ell, t, x} } \text{,} \\
\label{eqc.est_trace_shp} \| \phi \cdot \cint^t_0 \psi \|_{ B^{a, 2, s}_{\ell, t, x} } &\lesssim_s ( \| \partial \phi \|_{ L^{2, 2}_{t, x} } + \| \phi \|_{ L^{\infty, 2}_{x, t} } ) \| \psi \|_{ B^{a, 1, s}_{\ell, t, x} } \text{.}
\end{align}
\end{proposition}

\begin{proof}
For \eqref{eqc.est_trace_sh}, in the case $a = 1$, we decompose and estimate using \eqref{eqc.intertwining_trace_sh}:
\begin{align*}
\| \cint^t_0 ( \phi \psi ) \|_{ B^{1, \infty, s}_{\ell, t, x} } &\lesssim \sum_{k, l \geq 0} 2^{s k} \| E_k \cint^t_0 ( \phi E_l \psi ) \|_{ L^{\infty, 2}_{t, x} } + \sum_{k \geq 0} 2^{s k} \| E_k \cint^t_0 ( \phi E_{< 0} \psi ) \|_{ L^{\infty, 2}_{t, x} } \\
&\qquad + \sum_{l \geq 0} \| E_{< 0} \cint^t_0 ( \phi P_l \psi ) \|_{ L^2_x } + \| E_{< 0} \cint^t_0 ( \phi E_{< 0} \psi ) \|_{ L^{\infty, 2}_{t, x} } \\
&\lesssim ( \| \partial \phi \|_{ L^{2, 2}_{t, x} } + \| \phi \|_{ L^{\infty, 2}_{x, t} } ) \sum_{l \geq 0} \paren{ \sum_{k \geq 0} 2^{s k} 2^{- |k - l|} + 2^{-l} } \| E_{\sim l} \psi \|_{ L^{2, 2}_{t, x} } \\
&\qquad + ( \| \partial \phi \|_{ L^{2, 2}_{t, x} } + \| \phi \|_{ L^{\infty, 2}_{x, t} } ) \paren{ \sum_{k \geq 0} 2^{(s - 1) k} + 1 } \| E_{\lesssim 0} \psi \|_{ L^{2, 2}_{t, x} } \\
&\lesssim ( \| \partial \phi \|_{ L^{2, 2}_{t, x} } + \| \phi \|_{ L^{\infty, 2}_{x, t} } ) \| \psi \|_{ B^{1, 2, s}_{\ell, t, x} } \text{.}
\end{align*}
Note the basic strategy for applying \eqref{eqc.intertwining_trace_sh} is the same as that used in the proof of Proposition \ref{thmc.est_prod}.
The case $a = \infty$ also follows from \eqref{eqc.intertwining_trace_sh}.

Finally, \eqref{eqc.est_trace_shp} is established analogously, using the estimate \eqref{eqc.intertwining_trace_shp}.
\end{proof}

\subsection{Non-Integrated Product Estimates II} \label{sec.eucl_nintex}

Next, we consider product estimates in which one must apply L-P decompositions to both factors.
We begin in this subsection with non-integrated estimates.

The main building block behind the first such estimate is the following:

\begin{lemma} \label{thmc.est_envelope_prod}
Fix integers $k, l, m \geq 0$, and fix $s \in [0, 1)$.
Moreover, let
\[ \alpha = \frac{1 - s}{2} \text{,} \qquad \beta = \frac{1 + s}{2} \text{.} \]
If $f, g \in \mc{S}_x \R^2$, then the following estimates hold:
\begin{align}
\label{eqc.est_envelope_prod} 2^{s k} \| E_k ( E_l f \cdot E_m g ) \|_{ L^2_x } &\lesssim 2^{- \alpha |k - l| } 2^{\beta l} \| E_{\sim l} f \|_{ L^2_x } \cdot 2^{ - \alpha |k - m| } 2^{\beta m} \| E_{\sim m} g \|_{ L^2_x } \text{,} \\
\notag 2^{s k} \| E_k ( E_l f \cdot E_{< 0} g ) \|_{ L^2_x } &\lesssim 2^{- \alpha |k - l|} 2^{\beta l} \| E_{\sim l} f \|_{ L^2_x } \cdot 2^{- \alpha k } \| g \|_{ L^2_x } \text{,} \\
\notag 2^{s k} \| E_k ( E_{< 0} f \cdot E_m g ) \|_{ L^2_x } &\lesssim 2^{- \alpha k} \| f \|_{ L^2_x } \cdot 2^{ - \alpha |k - m| } 2^{\beta m} \| E_{\sim m} g \|_{ L^2_x } \text{,} \\
\notag \| E_{< 0} ( f \cdot g ) \|_{ L^2_x } &\lesssim \| f \|_{ L^2_x } \| g \|_{ L^2_x } \text{.}
\end{align}
\end{lemma}

\begin{proof}
We prove only the first estimate of \eqref{eqc.est_envelope_prod}, as the remaining bounds are similar and easier.
Let $A$ denote the left-hand side of the first estimate in \eqref{eqc.est_envelope_prod}.

Consider first the case $l \geq k$ and $m \geq k$.
By \eqref{eqc.glp_bernstein},
\[ A \lesssim 2^{ (s + 1) k } \| E_l f E_m g \|_{ L^1_x } \lesssim 2^{ (s + 1) k } \| E_l f \|_{ L^2_x } \| E_m g \|_{ L^2_x } \text{.} \]
Since $\alpha + \beta = 1$ and $k \leq l$, then
\begin{align*}
A &\lesssim 2^{s l} 2^{ \alpha (k - l) } 2^{ \alpha l } \| E_{\sim l} f \|_{ L^2_x } 2^{ \beta (k - m) } 2^{ \beta m } \| E_{\sim m} g \|_{ L^2_x } \\
&\lesssim 2^{ \alpha (k - l) } 2^{ \beta l } \| E_{\sim l} f \|_{ L^2_x } 2^{ \alpha (k - m) } 2^{ \beta m } \| E_{\sim m} g \|_{ L^2_x } \text{,}
\end{align*}
and the desired estimate follows in this setting.

Next, consider when $l \leq k \leq m$.
Applying H\"older's inequality and \eqref{eqc.glp_bernstein} yields
\[ A \lesssim 2^{s k} \| E_l f \|_{ L^\infty_x } \| E_m g \|_{ L^2_x } \lesssim 2^{s m} 2^l \| E_{\sim l} f \|_{ L^2_x } \| E_m g \|_{ L^2_x } \text{.} \]
The intended estimate follows from the above, since
\begin{align*}
A &\lesssim 2^{ \alpha l } 2^{ \beta l } \| E_{\sim l} f \|_{ L^2_x } 2^{ -\alpha m } 2^{ (\alpha + s) m } \| E_{\sim m} g \|_{ L^2_x } \\
&\lesssim 2^{ \alpha (l - k) } 2^{ \beta l } \| E_{\sim l} f \|_{ L^2_x } 2^{ \alpha (k - m) } 2^{ \beta m } \| E_{\sim m} g \|_{ L^2_x } \text{.}
\end{align*}
The case $m \leq k \leq l$ can be proved similarly, by switching the roles of $f$ and $g$.

The remaining case, $m \leq k$ and $l \leq k$, is negligible, since due to Fourier supports, these vanish whenever both $m$ and $l$ are less than, say, $k - 5$.
\end{proof}

\begin{proposition} \label{thmc.est_prod_sob}
Let $s \in [0, 1)$, and suppose that $p, p_1, p_2 \in [1, \infty]$ satisfy the relation $p^{-1} = p_1^{-1} + p_2^{-1}$.
If $\phi, \psi \in \mc{C}^\infty_t \mc{S}_x \R^2$, then
\begin{align}
\label{eqc.est_prod_sob} \| \phi \cdot \psi \|_{ B^{1, p, s}_{\ell, t, x} } &\lesssim_s \| \phi \|_{ B^{2, p_1, (1 + s) / 2}_{\ell, t, x} } \| \psi \|_{ B^{2, p_2, (1 + s) / 2}_{\ell, t, x} } \text{.}
\end{align}
\end{proposition}

\begin{proof}
We first decompose the left-hand side and apply \eqref{eqc.est_envelope_prod}:
\begin{align*}
\| \phi \psi \|_{ B^{1, p, s}_{\ell, t, x} } &\lesssim \sum_{k, l, m \geq 0} 2^{s k} \| E_k ( E_l \phi E_m \psi ) \|_{ L^{p, 2}_{t, x} } + \sum_{k, l \geq 0} 2^{s k} \| E_k ( E_l \phi E_{< 0} \psi ) \|_{ L^{p, 2}_{t, x} } \\
&\qquad + \sum_{k, l \geq 0} 2^{s k} \| E_k ( E_{< 0} \phi E_m \psi ) \|_{ L^{p, 2}_{t, x} } + \| E_{< 0} ( \phi \psi ) \|_{ L^{p, 2}_{t, x} } \\
&\lesssim \sum_{k \geq 0} \mc{E}^1_k \phi \mc{E}^2_k \psi + \| \phi \|_{ L^{p_1, 2}_{t, x} } \| \psi \|_{ L^{p_2, 2}_{t, x} } \text{,}
\end{align*}
where
\begin{align*}
\mc{E}^1_k \phi &= \sum_{l \geq 0} 2^{- \frac{1 - s}{2} |k - l|} 2^{ \frac{ 1 + s }{2} l } \| E_l \phi \|_{ L^{p_1, 2}_{t, x} } + 2^{ -\frac{ 1 - s }{2} k } \| \phi \|_{ L^{p_1, 2}_{t, x} } \text{,} \\
\mc{E}^2_k \psi &= \sum_{l \geq 0} 2^{- \frac{1 - s}{2} |k - l|} 2^{ \frac{ 1 + s }{2} l } \| E_l \psi \|_{ L^{p_2, 2}_{t, x} } + 2^{ -\frac{ 1 - s }{2} k } \| \psi \|_{ L^{p_2, 2}_{t, x} } \text{,}
\end{align*}
The proof now follows from the inequality below, which can be computed directly:
\[ \sum_{k \geq 0} ( \mc{E}^1_k \phi \mc{E}^2_k \psi ) \lesssim \| \phi \|_{ B^{2, p_1, (1 + s) / 2}_{\ell, t, x} } \| \psi \|_{ B^{2, p_2, (1 + s) / 2}_{\ell, t, x} } \text{.} \qedhere \]
\end{proof}

Next, we introduce the following trace estimate.

\begin{proposition} \label{thmc.nsob_trace}
If $\phi \in \mc{C}^\infty_t \mc{S}_x \mc{N}$, then
\begin{equation} \label{eqc.nsob_trace} \| \phi \|_{ B^{2, \infty, 1/2}_{\ell, t, x} } \lesssim \| \phi [0] \|_{ H^{1/2}_x } + \| \partial_t \phi \|_{ L^{2, 2}_{t, x} }^\frac{1}{2} ( \| \partial \phi \|_{ L^{2, 2}_{t, x} } + \| \phi \|_{ L^{2, 2}_{t, x} } )^\frac{1}{2} \text{.} \end{equation}
\end{proposition}

\begin{proof}
For any integer $k \geq 0$, we have the calculus inequalities
\begin{align*}
2^k \| E_k \phi \|_{ L^{\infty, 2}_{t, x} }^2 &\lesssim 2^k \| E_k \phi [0] \|_{ L^2_x }^2 + 2^k \int_0^\delta \int_{ \R^2 } \partial_t | E_k \phi |^2 |_{ (\tau, x) } dx d\tau \\
&\lesssim 2^k \| E_k \phi [0] \|_{ L^2_x }^2 + \| E_k \partial_t \phi \|_{ L^{2, 2}_{t, x} } \cdot 2^k \| E_k \phi \|_{ L^{2, 2}_{t, x} } \text{,} \\
\| E_{< 0} \phi \|_{ L^{\infty, 2}_{t, x} }^2 &\lesssim \| E_{< 0} \phi [0] \|_{ L^2_x }^2 + \| E_{< 0} \partial_t \phi \|_{ L^{2, 2}_{t, x} } \cdot \| E_{< 0} \phi \|_{ L^{2, 2}_{t, x} } \text{.}
\end{align*}
Summing the above inequalities results in \eqref{eqc.nsob_trace}.
\end{proof}

Combining Propositions \ref{thmc.est_prod_sob} and \ref{thmc.nsob_trace} yields the following estimate.

\begin{proposition} \label{thmc.est_prod_ex}
For any $\phi, \psi \in \mc{C}^\infty_t \mc{S}_x \mc{N}$,
\begin{align}
\label{eqc.est_prod_ex} \| \phi \cdot \psi \|_{ B^{1, \infty, 0}_{\ell, t, x} } &\lesssim ( \| \phi \|_{ N^1_{t, x} } + \| \phi [0] \|_{ H^{1/2}_x } ) ( \| \psi \|_{ N^1_{t, x} } + \| \psi [0] \|_{ H^{1/2}_x } ) \text{.}
\end{align}
\end{proposition}

\subsection{Integrated Product Estimates II} \label{sec.eucl_intex}

Finally, we prove an integrated analogue of Proposition \ref{thmc.est_prod_ex}.
In the Euclidean case, this estimate was originally proved within \cite[Sect. 3]{kl_rod:stt}; we give another version of the proof here for completeness.

Given an integer $k \geq 0$, we define
\begin{align*}
\mc{N}_k \phi &= \| E_k \partial_t \phi \|_{ L^{2, 2}_{t, x} } + 2^k \| E_k \phi \|_{ L^{2, 2}_{t, x} } + 2^\frac{k}{2} \| E_k \phi \|_{ L^{\infty, 2}_{t, x} } \text{.}
\end{align*}
We will exploit these quantities in order to dyadically recover Sobolev and Besov norms, in the same manner as in the proof of Proposition \ref{thmc.est_prod_sob}.

\begin{lemma} \label{thmc.est_envelope}
Fix integers $k, l, m \geq 0$.
If $\phi, \psi \in C^\infty_t \mc{S}_x \mc{N}$, then
\begin{align}
\label{eqc.est_envelope_trace} \| E_k \cint^t_0 ( E_l \partial_t \phi \cdot E_m \psi ) \|_{ L^{\infty, 2}_{t, x} } &\lesssim 2^{ -\frac{1}{2} |k - l| } \mc{N}_l \phi \cdot 2^{ - \frac{1}{2} |k - m| } \mc{N}_m \psi \text{,} \\
\notag \| E_k \cint^t_0 ( E_l \partial_t \phi \cdot E_{< 0} \psi ) \|_{ L^{\infty, 2}_{t, x} } &\lesssim 2^{ -\frac{1}{2} |k - l| } \mc{N}_l \phi \cdot 2^{ -\frac{1}{2} k } ( \| \psi \|_{ N^1_{t, x} } + \| \psi \|_{ L^{\infty, 2}_{t, x} } ) \text{,} \\
\notag \| E_k \cint^t_0 ( E_{< 0} \partial_t \phi \cdot E_m \psi ) \|_{ L^{\infty, 2}_{t, x} } &\lesssim 2^{ -\frac{1}{2} k } \| \phi \|_{ N^1_{t, x} } \cdot 2^{ -\frac{1}{2} |k - m| } \mc{N}_m \psi \text{,} \\
\notag \| E_{< 0} \cint^t_0 ( \partial_t \phi \cdot \psi ) \|_{ L^{\infty, 2}_{t, x} } &\lesssim \| \partial_t \phi \|_{ L^{2, 2}_{t, x} } \| \psi \|_{ L^{2, 2}_{t, x} } \text{.}
\end{align}
\end{lemma}

\begin{proof}
Again, we prove only the first estimate of \eqref{eqc.est_envelope_trace}, as it is the most difficult.
Let $B$ denote the left-hand side of the first part of \eqref{eqc.est_envelope_trace}.

First, if $k \leq l \leq m$, then by H\"older's inequality and \eqref{eqc.glp_bernstein},
\begin{align*}
B &\lesssim 2^k \| E_l \partial_t \phi \cdot E_m \psi \|_{ L^{1, 1}_{t, x} } \lesssim 2^k \| E_l \partial_t \phi \|_{ L^{2, 2}_{t, x} } \| E_m \psi \|_{ L^{2, 2}_{t, x} } \text{.}
\end{align*}
By our assumption on $k, l, m$, then
\begin{align*}
B &\lesssim 2^{ \frac{1}{2} ( k - l ) } \mc{N}_l \phi \cdot 2^{ \frac{1}{2} ( k + l ) - m } \mc{N}_m \psi \lesssim 2^{- \frac{1}{2} |k - l| - \frac{1}{2} |k - m| } \mc{N}_l \phi \mc{N}_m \psi \text{.}
\end{align*}
On the other hand, if $k \leq m \leq l$, then we can reduce to the previous case by moving the ``$\partial_t$" to the other factor via an integration by parts:
\[ B \lesssim \| E_k \cint^t_0 ( E_l \phi E_m \partial_t \psi ) \|_{ L^{\infty, 2}_{t, x} } + \| E_k ( E_l \phi E_m \psi ) \|_{ L^{\infty, 2}_{t, x} } \text{.} \]
By symmetry, the first term on the right-hand side can be handled in the same way as the preceding case $k \leq l \leq m$.
The second term can bounded using \eqref{eqc.est_envelope_prod}.

Next, consider when $l \leq k \leq m$.
Absorbing the $\cint^t_0$ into the integral norm, then
\[ B \lesssim \| E_l \partial_t \phi \|_{ L^{\infty, 2}_{x, t} } \| E_m \psi \|_{ L^{2, 2}_{t, x} } \lesssim 2^{-k} 2^l \| E_l \partial_t \phi \|_{ L^{2, 2}_{t, x} } 2^k \| E_m \psi \|_{ L^{2, 2}_{t, x} } \text{.} \]
As a result, we obtain, as desired,
\[ B \lesssim 2^{l - k} \mc{N}_l \phi \cdot 2^{k - m} \mc{N}_m \phi \lesssim 2^{- \frac{1}{2} |k - l| - \frac{1}{2} |k - m| } \mc{N}_l \phi \mc{N}_m \psi \text{.} \]
For the opposite case $m \leq k \leq l$, we must once again integrate by parts:
\[ B \lesssim \| E_k \cint^t_0 ( E_l \phi E_m \partial_t \psi ) \|_{ L^{\infty, 2}_{t, x} } + \| E_k ( E_l \phi E_m \psi ) \|_{ L^{\infty, 2}_{t, x} } \text{.} \]
The first term on the right-hand side is now equivalent to the case $l \leq k \leq m$ by symmetry.
The second term on the right-hand side is controlled using \eqref{eqc.est_envelope_prod}.

The remaining case $m \leq k$ and $l \leq k$ is negligible due to Fourier supports.
\end{proof}

\begin{proposition} \label{thmc.est_trace_ex}
For any $\phi, \psi \in C^\infty_t \mc{S}_x \mc{N}$,
\begin{align}
\label{eqc.est_trace_ex} \| \cint^t_0 ( \partial_t \phi \cdot \psi ) \|_{ B^{1, \infty, 0}_{\ell, t, x} } &\lesssim ( \| \phi \|_{ N^1_{t, x} } + \| \phi [0] \|_{ H^{1/2}_x } ) ( \| \psi \|_{ N^1_{t, x} } + \| \psi [0] \|_{ H^{1/2}_x } ) \text{.}
\end{align}
\end{proposition}

\begin{proof}
Let $L$ denote the left-hand side of \eqref{eqc.est_trace_ex}.
We decompose and apply \eqref{eqc.est_envelope_trace}:
\begin{align*}
L &\lesssim \sum_{k, l, m \geq 0} \| E_k \cint^t_0 ( E_l \partial_t \phi \cdot E_m \psi ) \|_{ L^{\infty, 2}_{t, x} } + \sum_{k, l \geq 0} \| E_k \cint^t_0 ( E_l \partial_t \phi \cdot E_{< 0} \psi ) \|_{ L^{\infty, 2}_{t, x} } \\
&\qquad + \sum_{k, l \geq 0} \| E_k \cint^t_0 ( E_{< 0} \partial_t \phi \cdot E_m \psi ) \|_{ L^{\infty, 2}_{t, x} } + \| E_{< 0} \cint^t_0 ( \partial_t \phi \cdot \psi ) \|_{ L^{\infty, 2}_{t, x} } \\
&\lesssim \sum_{k \geq 0} \mc{E}_k \phi \mc{E}_k \psi + \| \partial_t \phi \|_{ L^{2, 2}_{t, x} } \| \psi \|_{ L^{2, 2}_{t, x} } \text{,}
\end{align*}
where
\begin{align*}
\mc{E}_k \phi &= \sum_{l \geq 0} 2^{-\frac{1}{2} |k - l|} \mc{N}_l \phi + 2^{ -\frac{k}{2} } ( \| \phi \|_{ N^1_{t, x} } + \| \phi \|_{ B^{2, \infty, 1/2}_{\ell, t, x} } ) \text{,}
\end{align*}
and similarly for $\mc{E}_k \psi$.
By a direct computation, we see that
\[ \sum_{k \geq 0} \mc{E}_k \phi \mc{E}_k \psi \leq ( \| \phi \|_{ N^1_{t, x} } + \| \phi \|_{ B^{2, \infty, 1/2}_{\ell, t, x} } ) ( \| \psi \|_{ N^1_{t, x} } + \| \psi \|_{ B^{2, \infty, 1/2}_{\ell, t, x} } ) \text{.} \]
The proof is now completed by applying \eqref{eqc.nsob_trace}.
\end{proof}

\section{Besov Norm Comparisons} \label{sec.bcomp}

In the proofs of our main results, we used that the geometric Besov norms were equivalent to certain coordinate-based Besov norms constructed using the regularity conditions of Section \ref{sec.fol_reg}.
The precise comparisons were stated in Proposition \ref{thm.comp_main}.
In this appendix, we will prove this proposition.

Throughout, we will consider both the geometric and the Euclidean L-P operations concurrently.
Moreover, we will refer to integral and Besov norms in both the geometric and the Euclidean settings.
As usual, the norm that is being referenced will depend on the specific context.
We also remark that the methods used here are heavily inspired by a similar estimate found in \cite{wang:cg}.

\subsection{Scalar Reduction Estimates} \label{sec.bcomp_scal}

One technical step in the proof of Proposition \ref{thm.comp_main} is to show that the scalar reduction process described in Section \ref{sec.fol_scal} is sufficiently compatible with our geometric norms.
The main step in this process is the following intermediate estimates, resembling estimates found in Appendix \ref{sec.eucl}.

\begin{lemma} \label{thm.intertwining_prodx_elem}
Assume $(\mc{S}, h)$ satisfies \ass{r0}{C, N}.
Moreover, suppose $F \in \mc{C}^\infty T^{r_1}_{l_1} \mc{S}$ and $G \in \mc{C}^\infty T^{r_2}_{l_2} \mc{S}$.
If $k, l \geq \Z$, with $k \geq 0$, then
\begin{align}
\label{eq.intertwining_prodx_elem} \| P_k ( F \otimes P_l G ) \|_{ L^2_x } &\lesssim_{C, N} 2^{- |k - l|} ( \| \nabla F \|_{ L^4_x } + \| F \|_{ L^\infty_x } ) \| P_{\sim l} G \|_{ L^2_x } \text{,} \\
\notag \| P_k ( F \otimes P_{< 0} G ) \|_{ L^2_x } &\lesssim_{C, N} 2^{-k} ( \| \nabla F \|_{ L^4_x } + \| F \|_{ L^\infty_x } ) \| P_{\lesssim 0} G \|_{ L^2_x } \text{,} \\
\notag \| P_{< 0} ( F \otimes P_l G ) \|_{ L^2_x } &\lesssim_{C, N} 2^{-l} ( \| \nabla F \|_{ L^4_x } + \| F \|_{ L^\infty_x } ) \| P_{\sim l} G \|_{ L^2_x } \text{,} \\
\notag \| P_{< 0} ( F \otimes P_{< 0} G ) \|_{ L^2_x } &\lesssim_{C, N} \| F \|_{ L^\infty_x } \| P_{\lesssim 0} G \|_{ L^2_x } \text{.}
\end{align}
\end{lemma}

\begin{proof}
The proof is analogous to that of \eqref{eqc.intertwining_prod_elem}.
We prove only the first estimate in \eqref{eq.intertwining_prodx_elem} here, as the remaining estimates are similar but easier.
First, if $l \geq k$, we create an instance of $\lapl$ and apply \eqref{eq.glp_fb} and \eqref{eq.glp_bernstein}:
\begin{align*}
\| P_k ( F \otimes P_l G ) \|_{ L^2_x } &\lesssim 2^{-2l} \| P_k ( F \otimes \lapl \tilde{P}_l G ) \|_{ L^2_x } \\
&\lesssim 2^{-2l} [ \| P_k \nabla ( F \otimes \nabla \tilde{P}_l G ) \|_{ L^2_x } + \| P_k ( \nabla F \otimes \nabla \tilde{P}_l G ) \|_{ L^2_x } ] \\
&\lesssim 2^{-2l + k} ( \| F \otimes \nabla \tilde{P}_l G \|_{ L^2_x } + \| \nabla F \otimes \nabla \tilde{P}_l G \|_{ L^{4/3}_x } ) \\
&\lesssim 2^{-2l + k} ( \| F \|_{ L^\infty_x } \| \nabla \tilde{P}_l G \|_{ L^2_x } + \| \nabla F \|_{ L^4_x } \| \nabla \tilde{P}_l G \|_{ L^2_x } ) \text{.}
\end{align*}
Applying \eqref{eq.glp_fb} results in the desired estimate in this case.
On the other hand, if $l < k$, we apply \eqref{eq.glp_fbr} and H\"older's inequality to obtain
\begin{align*}
\| P_k ( F \otimes P_l G ) \|_{ L^2_x } &\lesssim 2^{-k} ( \| \nabla F \otimes P_l G \|_{ L^2_x } + \| F \otimes \nabla P_l G \|_{ L^2_x } ) \\
&\lesssim 2^{-k} ( \| \nabla F \|_{ L^4_x } \| P_l G \|_{ L^4_x } + \| F \|_{ L^\infty_x } \| \nabla P_l G \|_{ L^2_x } ) \text{.}
\end{align*}
Applying \eqref{eq.glp_fb} and \eqref{eq.glp_bernstein} results in the desired estimate.
\end{proof}

We can apply Lemma \ref{thm.intertwining_prodx_elem} to prove a non-sharp variant of \eqref{eq.est_prod_elem}.

\begin{lemma} \label{thm.est_prodx_elem}
Assume $(\mc{N}, \gamma)$ satisfies \ass{R0}{C, N}, and suppose $a \in [1, \infty]$, $p \in [1, \infty]$, and $s \in (-1, 1)$.
If $\Psi \in \mc{C}^\infty \ul{T}^{r_1}_{l_1} \mc{N}$ and $\Phi \in \mc{C}^\infty \ul{T}^{r_2}_{l_2} \mc{N}$, then
\begin{align}
\label{eq.est_prodx_elem} \| \Phi \otimes \Psi \|_{ B^{a, p, s}_{\ell, t, x} } &\lesssim_{C, N, s} ( \| \nabla \Phi \|_{ L^{\infty, 4}_{t, x} } + \| \Phi \|_{ L^{\infty, \infty}_{t, x} } ) \| \Psi \|_{ B^{a, p, s}_{\ell, t, x} } \text{.}
\end{align}
\end{lemma}

\begin{proof}
The proof is analogous to that of \eqref{eqc.est_prod_elem}.
Assume first $a = 1$, and decompose
\begin{align*}
\| \Phi \otimes \Psi \|_{ B^{1, p, s}_{\ell, t, x} } &\lesssim \sum_{k, l \geq 0} 2^{s k} \| P_k ( \Phi \otimes P_l \Psi ) \|_{ L^{p, 2}_{t, x} } + \sum_{k \geq 0} 2^{s k} \| P_k ( \Phi \otimes P_{< 0} \Psi ) \|_{ L^{p, 2}_{t, x} } \\
&\qquad + \sum_{l \geq 0} \| P_{< 0} ( \Phi \otimes P_l \Psi ) \|_{ L^{p, 2}_{t, x} } + \| P_{< 0} ( \Phi \otimes P_{< 0} \Psi ) \|_{ L^{p, 2}_{t, x} } \text{.}
\end{align*}
Since every $(\mc{S}, \gamma [\tau])$ satisfies \ass{r0}{C, N}, each of the above terms can be controlled using \eqref{eq.intertwining_prodx_elem}, essentially in the same manner as in the proof of \eqref{eqc.est_prod_elem}.
The $a = \infty$ case is also handled similarly, again like in the proof of \eqref{eqc.est_prod_elem}.
Finally, for general $a$, we can interpolate between the above two cases.
\end{proof}

We can now use Lemma \ref{thm.est_prodx_elem} to prove our geometric scalar reduction comparison.

\begin{lemma} \label{thm.besov_scalar_red}
Assume $(\mc{N}, \gamma)$ satisfies \ass{R1}{C, N}, with data $\{ U_i, \varphi_i, \eta_i, \tilde{\eta}_i, e_i \}_{i = 1}^N$.
Then, for any $a \in [1, \infty]$, $p \in [1, \infty]$, $s \in (-1, 1)$, and $\Psi \in \mc{C}^\infty \ul{T}^r_l \mc{N}$,
\footnote{Note that $\eta_i X$ extends smoothly to a global tensor field on $\mc{S}$ for each $1 \leq i \leq N$ and $X \in \mc{X}^r_l (i)$.  Therefore, it makes sense to take a geometric norm of $\eta_i \Psi (X)$.}
\begin{align}
\label{eq.besov_scalar_red} \| \Psi \|_{ B^{a, p, s}_{\ell, t, x} } \simeq_{C, N, s, r, l} \sum_{i = 1}^N \sum_{ X \in i \mc{X}^r_l } \| \eta_i \Psi ( X ) \|_{ B^{a, p, s}_{\ell, t, x} } \text{.}
\end{align}
\end{lemma}

\begin{proof}
First of all, we apply \eqref{eq.est_prodx_elem} for each $\eta_i$ and $X$ to obtain
\begin{align*}
\| \eta_i \Psi (X) \|_{ B^{a, p, s}_{\ell, t, x} } &\lesssim ( \| \nabla ( \eta_i X ) \|_{ L^{\infty, 4}_{t, x} } + \| \eta_i X \|_{ L^{\infty, \infty}_{t, x} } ) \| \Psi \|_{ B^{a, p, s}_{\ell, t, x} } \lesssim \| \Psi \|_{ B^{a, p, s}_{\ell, t, x} } \text{.}
\end{align*}
In the last step, we applied \ass{R1}{C, N} and \eqref{eq.frame_basis} to control $\eta_i X$.
Summing the above inequality over both $i$ and $X$ proves one half of \eqref{eq.besov_scalar_red}.

Next, given $1 \leq i \leq N$ and $X \in i \mc{X}^r_l$, we let $X^\ast \in i \mc{X}^l_r$ denote the dual basis element on $U_i$ to $X$.
Since $\eta_i \Psi$ is supported entirely on $U_i$,
\[ \Psi = \sum_{i = 1}^N ( \eta_i \Psi ) = \sum_{i = 1}^N \sum_{ X \in i \mc{X}^r_l } \eta_i \Psi (X) \otimes \tilde{\eta}_i X^\ast \text{.} \]
As a result, we can once again apply \eqref{eq.est_prodx_elem}, \ass{R1}{C, N}, and \eqref{eq.frame_basis}:
\begin{align*}
\| \Psi \|_{ B^{a, p, s}_{\ell, t, x} } &\lesssim \sum_{i = 1}^N \sum_{ X \in i \mc{X}^r_l } ( \| \nabla ( \tilde{\eta}_i X^\ast ) \|_{ L^{\infty, 4}_{t, x} } + \| \tilde{\eta}_i X^\ast \|_{ L^{\infty, \infty}_{t, x} } ) \| \eta_i \Psi (X) \|_{ B^{a, p, s}_{\ell, t, x} } \\
&\lesssim \sum_{i = 1}^N \sum_{ X \in i \mc{X}^r_l } \| \eta_i \Psi (X) \|_{ B^{a, p, s}_{\ell, t, x} } \text{.} \qedhere
\end{align*}
\end{proof}

\subsection{Mixed Estimates} \label{sec.bcomp_mixed}

The other fundamental step in the proof of Proposition \ref{thm.comp_main} is a number of preliminary comparisons in the special case of localized scalar quantities.
The main technical component in this endeavor is a collection of estimates involving both geometric and Euclidean L-P operators.

\begin{lemma} \label{thm.intertwining_mixed}
Assume $(\mc{S}, h)$ satisfies \ass{r1}{C, N}, with data $\{ U_i, \varphi_i, \eta_i, \tilde{\eta}_i, e_i \}_{i = 1}^N$.
In addition, fix integers $1 \leq i \leq N$ and $k, l \geq 0$.
\begin{itemize}
\item If $f \in \mc{C}^\infty \mc{S}$, then
\footnote{In \eqref{eq.intertwining_mixed_good}, multiplying $P_l f$ and $P_{< 0} f$ by $\tilde{\eta}_i$ ensures that the resulting functions are compactly supported in $U_i$, so that the outer (classical) L-P operator $E_k$ is well-defined.}
\begin{align}
\label{eq.intertwining_mixed_good} \| E_k [ ( \tilde{\eta}_i \cdot P_l f ) \circ \varphi_i^{-1} ] \|_{ L^2_x } &\lesssim_C 2^{ - |k - l| } \| P_{\sim l} f \|_{ L^2_x } \text{,} \\
\notag \| E_k [ ( \tilde{\eta}_i \cdot P_{< 0} f ) \circ \varphi_i^{-1} ] \|_{ L^2_x } &\lesssim_C 2^{- k} \| P_{\lesssim 0} f \|_{ L^2_x } \text{,} \\
\notag \| E_{< 0} [ ( \tilde{\eta}_i \cdot P_l f ) \circ \varphi_i^{-1} ] \|_{ L^2_x } &\lesssim_C 2^{- l} \| P_{\sim l} f \|_{ L^2_x } \text{,} \\
\notag \| E_{< 0} [ ( \tilde{\eta}_i \cdot P_{< 0} f ) \circ \varphi_i^{-1} ] \|_{ L^2_x } &\lesssim_C \| P_{\lesssim 0} f \|_{ L^2_x } \text{.}
\end{align}

\item If $g \in \mc{S}_x \R^2$, and if $\bar{\eta}_i = \tilde{\eta}_i \circ \varphi_i^{-1}$, then
\footnote{In \eqref{eq.intertwining_mixed_bad}, multiplying $E_l g$ and $E_{< 0} g$ by $\bar{\eta}_i$ ensures that the resulting functions are compactly supported in $\varphi_i (U_i)$, hence their compositions by $\varphi_i$ can be smoothly extended to $\mc{S}$.}
\begin{align}
\label{eq.intertwining_mixed_bad} \| P_k [ ( \bar{\eta}_i \cdot E_l g ) \circ \varphi_i ] \|_{ L^2_x } &\lesssim_C |k - l | 2^{- |k - l|} \| E_{\sim l} g \|_{ L^2_x } \text{,} \\
\notag \| P_k [ ( \bar{\eta}_i \cdot E_{< 0} g ) \circ \varphi_i ] \|_{ L^2_x } &\lesssim_C 2^{- k} \| E_{\lesssim 0} g \|_{ L^2_x } \text{,} \\
\notag \| P_{< 0} [ ( \bar{\eta}_i \cdot E_l g ) \circ \varphi_i ] \|_{ L^2_x } &\lesssim_C l 2^{- l} \| E_{\sim l} g \|_{ L^2_x } \text{,} \\
\notag \| P_{< 0} [ ( \bar{\eta}_i \cdot E_{< 0} g ) \circ \varphi_i ] \|_{ L^2_x } &\lesssim_C \| E_{\lesssim 0} g \|_{ L^2_x } \text{.}
\end{align}
\end{itemize}
\end{lemma}

\begin{proof}
Once again, we only prove the first estimate in each set.

Consider first the case $l \leq k$.
For this setting, we will prove a more general version of \eqref{eq.intertwining_mixed_bad}.
Let $\nu \in \mc{S}_x \R^2$.
Applying \eqref{eq.glp_fbr}, \eqref{eqr.change_of_coord}, and \ass{r0}{C, N}, then
\begin{align*}
\| P_k ( \bar{\eta}_i \nu E_l g \circ \varphi_i ) \|_{ L^2_x } &\lesssim 2^{-k} \| \nabla ( \bar{\eta}_i \nu E_l g \circ \varphi_i ) \|_{ L^2_x } \\
&\lesssim 2^{-k} ( \| \bar{\eta}_i \nu \|_{ L^\infty_x } \| \partial E_l g \|_{ L^2_x } + \| \partial ( \bar{\eta}_i \nu ) \|_{ L^2_x } \| E_l g \|_{ L^\infty_x } ) \text{.}
\end{align*}
Applying \eqref{eqc.glp_fb}, \eqref{eqc.glp_bernstein}, and \ass{r0}{C, N}, we obtain, as desired,
\begin{equation} \label{eq.intertwining_mixed_pre} \| P_k ( \bar{\eta}_i \nu E_l g \circ \varphi_i ) \|_{ L^2_x } \lesssim 2^{- |k - l|} ( \| \partial \nu \|_{ L^2_x } + \| \nu \|_{ L^\infty_x } ) \| E_{\sim l} g \|_{ L^2_x } \text{,} \qquad l \leq k \text{.} \end{equation}
In particular, by setting $\nu \equiv 1$, we establish \eqref{eq.intertwining_mixed_bad} in the $l \leq k$ case.
Similarly, for \eqref{eq.intertwining_mixed_good}, we apply \eqref{eqr.change_of_coord}, \eqref{eqc.glp_fb}, and \ass{r0}{C, N}:
\begin{align*}
\| E_k ( \tilde{\eta}_i P_l f \circ \varphi_i^{-1} ) \|_{ L^2_x } &\lesssim 2^{-k} ( \| \nabla \tilde{\eta}_i \|_{ L^\infty_x } \| P_l f \|_{ L^2_x } + \| \nabla P_l f \|_{ L^2_x } ) \lesssim 2^{ - | k - l | } \| P_{\sim l} f \|_{ L^2_x } \text{.}
\end{align*}

Next, consider when $l \geq k$.
Fix $u \in \mc{S}_x \R^2$ and $w \in \mc{C}^\infty \mc{S}$, with
\[ \| u \|_{ L^2_x } = 1 \text{,} \qquad \| w \|_{ L^2_x } = 1 \text{.} \]
By standard duality arguments, it suffices to show that
\begin{align}
\label{eql.intertwining_mixed_good} \abs{ \int_{ \R^2 } E_k [ ( \tilde{\eta}_i P_l P_{\sim l} f ) \circ \varphi_i^{-1} ] \cdot u } &\lesssim 2^{- |k - l|} \| P_{\sim l} f \|_{ L^2_x } \text{,} \\
\label{eql.intertwining_mixed_bad} \abs{ \int_\mc{S} P_k [ ( \bar{\eta}_i E_l E_{\sim l} g ) \circ \varphi_i ] \cdot w \cdot d \omega } &\lesssim | k - l | 2^{- |k - l|} \| E_{\sim l} g \|_{ L^2_x } \text{.}
\end{align}
To show this, we first need the following preliminary estimates for $\tilde{\eta}_i \vartheta_i$, which follows immediately from \ass{r1}{C, N}, \eqref{eqr.vd_unif}, and \eqref{eqr.vdd_inv}:
\begin{alignat}{3}
\label{eql.intertwining_mixed_est} \| \tilde{\eta}_i \vartheta_i \|_{ L^\infty_x } &\lesssim 1 \text{,} &\qquad \| \tilde{\eta}_i \vartheta_i^{-1} \|_{ L^\infty_x } &\lesssim 1 \text{,} \\
\notag \| \nabla ( \tilde{\eta}_i \vartheta_i ) \|_{ L^2_x } &\lesssim 1 \text{,} &\qquad \| \nabla ( \tilde{\eta}_i \vartheta_i^{-1} ) \|_{ L^2_x } &\lesssim 1 \text{.}
\end{alignat}

Let $I_1$ and $I_2$ denote the left-hand sides of \eqref{eql.intertwining_mixed_good} and \eqref{eql.intertwining_mixed_bad}, respectively.
From the self-adjointness properties of $P_k$ and $E_l$, we have
\begin{align}
\label{eql.intertwining_mixed_bad_int} I_2 &= \abs{ \int_{ \mc{S} } [ ( \bar{\eta}_i E_l E_{\sim l} g ) \circ \varphi_i ] \cdot P_k w \cdot d \omega } \\
\notag &= \abs{ \int_{ \R^2 } E_l E_{\sim l} g \cdot [ ( \tilde{\eta}_i P_k w \cdot \vartheta_i ) \circ \varphi_i^{-1} ] } \\
\notag &\leq \| E_{\sim l} g \|_{ L^2_x } \| E_l [ ( \tilde{\eta}_i \vartheta_i \cdot P_k w ) \circ \varphi_i^{-1} ] \|_{ L^2_x } \\
\notag &= \| E_{\sim l} g \|_{ L^2_x } I_{11} \text{.}
\end{align}
By a similar computation, we have
\begin{align}
\label{eql.intertwining_mixed_good_int} I_1 &= \abs{ \int_\mc{S} P_l P_{\sim l} f \cdot \tilde{\eta} \vartheta_i^{-1} \cdot ( E_k u \circ \varphi_i ) \cdot d \omega } \\
\notag &\leq \| P_{\sim l} f \|_{ L^2_x } \| P_l [ \tilde{\eta}_i \vartheta_i^{-1} \cdot ( E_k u \circ \varphi_i ) ] \|_{ L^2_x } \text{.}
\end{align}

Applying \eqref{eq.intertwining_mixed_pre} in conjunction with \eqref{eqr.change_of_coord} and \eqref{eql.intertwining_mixed_est}, we obtain
\[ \| P_l [ \tilde{\eta}_i \vartheta_i^{-1} \cdot ( E_k u \circ \varphi_i ) ] \|_{ L^2_x } \lesssim 2^{ - | k - l | } \| E_{\sim k} u \|_{ L^2_x } \lesssim 2^{ - | k - l | } \text{.} \]
Combining the above with \eqref{eql.intertwining_mixed_good_int} yields \eqref{eql.intertwining_mixed_good}, which completes the proof of \eqref{eq.intertwining_mixed_good}.

Finally, to deal with \eqref{eql.intertwining_mixed_bad_int}, we decompose again using classical L-P operators:
\begin{align*}
I_{11} &\lesssim \sum_{m \geq 0} \| E_l \{ ( \vartheta \circ \varphi_i^{-1} ) E_m [ ( \tilde{\eta}_i P_k w ) \circ \varphi_i^{-1} ] \} \|_{ L^2_x } \\
&\qquad + \| E_l \{ ( \vartheta_i \circ \varphi_i^{-1} ) E_{< 0} [ ( \tilde{\eta}_i P_k w ) \circ \varphi_i^{-1} ] \} \|_{ L^2_x } \text{.}
\end{align*}
Applying \eqref{eqc.intertwining_prod_elem} and \eqref{eq.intertwining_mixed_good}, along with \eqref{eqr.vd_unif}, \eqref{eqr.change_of_coord}, and \eqref{eqr.vdd_inv}, then
\begin{align*}
I_{11} &\lesssim \sum_{m \geq 0} 2^{- |l - m| } \| E_{\sim m} [ ( \tilde{\eta}_i P_k w ) \circ \varphi_i^{-1} ] \|_{ L^2_x } + 2^{-l} \| ( \tilde{\eta}_i P_k w ) \circ \varphi_i^{-1} \|_{ L^2_x } \\
&\lesssim \sum_{m \geq 0} 2^{ - |l - m| - |m - k| } \| P_{\sim k} w \|_{ L^2_x } + 2^{-l} \| P_{\sim k} w \|_{ L^2_x } \\
&\lesssim | k - l | 2^{ - | l - k | } \text{.}
\end{align*}
The above, combined with \eqref{eql.intertwining_mixed_bad_int}, yields \eqref{eql.intertwining_mixed_bad}, which proves \eqref{eq.intertwining_mixed_bad}.
\end{proof}

\subsection{Completion of the Proof} \label{sec.bcomp_comp}

With the technical necessities out of the way, we can now finish the proof of Proposition \ref{thm.comp_main}.
The first step is to do this in the localized scalar case.
This argument is a minor variation of a similar result in \cite{wang:cg}.
Its main component is the mixed L-P estimates of Lemma \ref{thm.intertwining_mixed}.

\begin{lemma} \label{thm.besov_comp_sc}
Assume $(\mc{N}, h)$ satisfies \ass{R1}{C, N}, and let $1 \leq i \leq N$, $a \in [1, \infty]$, $p \in [1, \infty]$, and $s \in (-1, 1)$.
If $\phi \in C^\infty \mc{N}$ is supported within the support of $\eta_i$, then
\begin{align}
\label{eq.besov_comp_sc} \| \phi \|_{ B^{a, p, s}_{\ell, t, x} } \simeq_{C, s} \| \phi \circ \varphi_i^{-1} \|_{ B^{a, p, s}_{\ell, t, x} } \text{.}
\end{align}
\end{lemma}

\begin{proof}
For notational convenience, we define $\bar{\phi} = \phi \circ \varphi_i^{-1}$ and $\bar{\eta}_i = \tilde{\eta}_i \circ \varphi_i^{-1}$.
First, consider the case $a = 1$.
Using \eqref{eqc.glp_bdd}, \eqref{eq.glp_bdd}, and \eqref{eqr.change_of_coord}, we can decompose
\begin{align*}
\| \phi \|_{ B^{1, p, s}_{\ell, t, x} } &\lesssim \sum_{k, l \geq 0} 2^{s k} \| P_k ( \bar{\eta}_i E_l \bar{\phi} \circ \varphi_i ) \|_{ L^{p, 2}_{t, x} } + \sum_{k \geq 0} 2^{s k} \| P_k ( \bar{\eta}_i E_{< 0} \bar{\phi} \circ \varphi_i ) \|_{ L^{p, 2}_{t, x} } \\
\notag &\qquad + \sum_{l \geq 0} \| P_{< 0} ( \bar{\eta}_i E_l \bar{\phi} \circ \varphi_i ) \|_{ L^{p, 2}_{t, x} } + \| P_{< 0} ( \bar{\eta}_i E_{< 0} \bar{\phi} \circ \varphi_i ) \|_{ L^{p, 2}_{t, x} } \text{,} \\
\| \bar{\phi} \|_{ B^{1, p, s}_{\ell, t, x} } &\lesssim \sum_{k, l \geq 0} 2^{s k} \| E_k ( \tilde{\eta}_i P_l \phi \circ \varphi_i^{-1} ) \|_{ L^{p, 2}_{t, x} } + \sum_{k \geq 0} 2^{s k} \| E_k ( \tilde{\eta}_i P_{< 0} \phi \circ \varphi_i^{-1} ) \|_{ L^{p, 2}_{t, x} } \\
\notag &\qquad + \sum_{l \geq 0} \| E_{< 0} ( \tilde{\eta}_i P_l \phi \circ \varphi_i^{-1} ) \|_{ L^{p, 2}_{t, x} } + \| E_{< 0} ( \tilde{\eta}_i P_{< 0} \phi \circ \varphi_i^{-1} ) \|_{ L^{p, 2}_{t, x} } \text{.}
\end{align*}
Note in particular that due to the supports of $\phi$ and $\bar{\phi}$, then $\tilde{\eta}_i \phi = \phi$ and $\bar{\eta}_i \bar{\phi} = \bar{\phi}$.

Applying \eqref{eq.intertwining_mixed_good} the second inequality above, then
\begin{align*}
\| \bar{\phi} \|_{ B^{1, p, s}_{\ell, t, x} } &\lesssim \sum_{k, l \geq 0} 2^{s k} 2^{- |k - l|} \| P_{\sim l} \phi \|_{ L^{p, 2}_{t, x} } + \sum_{k \geq 0} 2^{s k} 2^{- k} \| P_{\lesssim 0} \phi \|_{ L^{p, 2}_{t, x} } \\
&\qquad + \sum_{l \geq 0} 2^{-l} \| P_{\sim l} \phi \|_{ L^{p, 2}_{t, x} } + \| P_{\lesssim 0} \phi \|_{ L^{p, 2}_{t, x} } \\
&\lesssim \| \phi \|_{ B^{1, p, s}_{\ell, t, x} } \text{.}
\end{align*}
Similarly, applying \eqref{eq.intertwining_mixed_bad}, we obtain
\begin{align*}
\| \phi \|_{ B^{1, p, s}_{\ell, t, x} } &\lesssim \sum_{k, l \geq 0} 2^{s k} | k - l | 2^{ - |k - l| } \| E_l \bar{\phi} \|_{ L^{p, 2}_{t, x} } + \sum_{k \geq 0} 2^{s k} 2^{-k} \| E_{\lesssim 0} \bar{\phi} \|_{ L^{p, 2}_{t, x} } \\
&\qquad + \sum_{l \geq 0} l 2^{- l} \| E_{\sim l} \bar{\phi} \|_{ L^{p, 2}_{t, x} } + \| E_{\lesssim 0} \bar{\phi} \|_{ L^{p, 2}_{t, x} } \\
&\lesssim \| \bar{\phi} \|_{ B^{1, p, s}_{\ell, t, x} } \text{.}
\end{align*}
In both cases, we also used that every $(\mc{S}, \gamma [\tau])$ satisfies \ass{r1}{C, N}.

The $a = \infty$ case can be proved similarly, again by applying \eqref{eq.intertwining_mixed_good} and \eqref{eq.intertwining_mixed_bad}.
The general case of arbitrary $a$ follows by interpolation.
\end{proof}

We can now complete the proof of Proposition \ref{thm.comp_main}.
For \eqref{eq.comp_main}, we combine the localized scalar estimates of Lemma \ref{thm.besov_comp_sc} with the scalar reduction argument of Lemma \ref{thm.besov_scalar_red}.
Indeed, given $\Psi \in \mc{C}^\infty \ul{T}^r_l \mc{N}$, we have
\begin{align*}
\| \Psi \|_{ B^{a, p, s}_{\ell, t, x} } &\simeq \sum_{i = 1}^N \sum_{ X \in i \mc{X}^r_l } \| \eta_i \Psi (X) \|_{ B^{a, p, s}_{\ell, t, x} } \\
&\simeq \sum_{i = 1}^N \sum_{ X \in i \mc{X}^r_l } \| \eta_i \Psi (X) \circ \varphi_i^{-1} \|_{ B^{a, p, s}_{\ell, t, x} } \\
&= \| \Psi \|_{ \mc{B}^{a, p, s}_{\ell, t, x} } \text{.}
\end{align*}

For the technical estimate \eqref{eq.comp_main_tech}, we first fix $1 \leq i \leq N$ and $X \in i \mc{X}^r_l$.
Partitioning the quantity $\tilde{\eta}_i \Psi (X)$ using the $\eta_j$'s and applying \eqref{eq.comp_main}, then
\begin{align*}
\| \tilde{\eta}_i \Psi (X) \circ \varphi_i^{-1} \|_{ B^{a, p, s}_{\ell, t, x} } &\lesssim \sum_{j = 1}^N \| \eta_j \Psi ( \tilde{\eta}_i X ) \circ \varphi_i^{-1} \|_{ B^{a, p, s}_{\ell, t, x} } \lesssim \| \Psi ( \tilde{\eta}_i X ) \|_{ B^{a, p, s}_{\ell, t, x} } \text{.}
\end{align*}
By \eqref{eq.est_prodx_elem} as well as the \ass{R1}{C, N} condition, we obtain
\begin{align*}
\| \tilde{\eta}_i \Psi (X) \circ \varphi_i^{-1} \|_{ B^{a, p, s}_{\ell, t, x} } &\lesssim ( \| \nabla ( \tilde{\eta}_i X ) \|_{ L^{\infty, 4}_{t, x} } + \| \tilde{\eta}_i X \|_{ L^{\infty, \infty}_{t, x} } ) \| \Psi \|_{ B^{a, p, s}_{\ell, t, x} } \lesssim \| \Psi \|_{ B^{a, p, s}_{\ell, t, x} } \text{.}
\end{align*}
Summing over $i$ and $X$ yields \eqref{eq.comp_main_tech}, as desired.

\section{Conformal Elliptic Estimates} \label{sec.cell}

The goal of this appendix is to complete the proof of Proposition \ref{thm.conf} by establishing the estimates \eqref{eq.besov_embed}, \eqref{eq.besov_nabla}, and \eqref{eq.besov_hodge}.
Throughout, we assume the setting of Proposition \ref{thm.conf}; in particular, we define $u$ and $\bar{\gamma}$ using \eqref{eq.conf_curv} and \eqref{eq.conf_transform}.
Moreover, we adopt all the conventions used throughout Section \ref{sec.curv_conf}.

Recall we have already established that $(\mc{N}, \bar{\gamma})$ satisfies \ass{R1}{C^\prime, N} and \ass{K}{C^\prime, D^\prime} for appropriate constants $C^\prime$, $D^\prime$ depending on $C$ and $N$.
More importantly, though, we also have the estimate \eqref{eq.conf_curv_est}, which provides $\bar{L}^2_x$-control for $\bar{\mc{K}}$.
The main idea is that this $\bar{L}^2_x$-control for $\bar{\mc{K}}$ is far better than the control available for $\mc{K}$ from the \ass{K}{} condition.
Consequently, with respect to $(\mc{N}, \bar{\gamma})$, we can derive stronger versions of the estimates found in Sections \ref{sec.curv_ell}-\ref{sec.curv_wbdh}.

\subsection{Strong Curvature Estimates} \label{sec.cell_ellex}

Here, we demonstrate the improved elliptic estimates that are avaiable due to $\bar{\mc{K}}$ having $\bar{L}^2_x$-control.
Each of these estimates takes place on an arbitrary fixed level set $(\mc{S}, \bar{\gamma} [\tau])$ of our foliation.
The first improvement for $(\mc{N}, \bar{\gamma})$ is the following tensorial Bochner estimate.

\begin{lemma} \label{thm.ell_est}
Assume $(\mc{N}, \gamma)$ satisfies \ass{R1}{C, N} and \ass{K}{C, D}, with $D \ll 1$ sufficiently small.
If $F \in \mc{C}^\infty T^r_l \mc{S}$, then, with respect to $(\mc{S}, \bar{\gamma} [\tau])$, we have
\begin{equation} \label{eq.ell_est} \| \bar{\nabla}^2 F \|_{ \bar{L}^2_x } \lesssim_{C, N, r, l} \| \bar{\lapl} F \|_{ \bar{L}^2_x } + \| \bar{\nabla} F \|_{ \bar{L}^2_x } + \| F \|_{ \bar{L}^2_x } \text{.} \end{equation}
\end{lemma}

\begin{proof}
Applying \eqref{eq.div_curl_est} to $\bar{\nabla} F$ on $(\mc{S}, \bar{\gamma} [\tau])$ yields
\begin{align*}
\| \bar{\nabla}^2 F \|_{ \bar{L}^2_x } &\lesssim \| \bar{\lapl} F \|_{ \bar{L}^2_x } + (r + l) \| | \bar{\mc{K}} | \cdot | F | \|_{ \bar{L}^2_x } + \| \bar{\nabla} F \|_{ \bar{L}^2_x } \\
&\lesssim \| \bar{\lapl} F \|_{ \bar{L}^2_x } + (r + l) D \| F \|_{ \bar{L}^\infty_x } + (r + l) \| F \|_{ \bar{L}^2_x } + \| \bar{\nabla} F \|_{ \bar{L}^2_x } \text{.}
\end{align*}
The $\bar{L}^\infty_x$-norm of $F$ can be treated using \eqref{eq.gns_2}:
\begin{align*}
(r + l) \| F \|_{ \bar{L}^\infty_x } &\lesssim (r + l) \| \bar{\nabla}^2 F \|_{ \bar{L}^2_x }^\frac{1}{2} \| F \|_{ \bar{L}^2_x }^\frac{1}{2} + (r + l) \| F \|_{ \bar{L}^2_x } \\
&\lesssim \| \bar{\nabla}^2 F \|_{ \bar{L}^2_x } + (r + l)^2 \| F \|_{ \bar{L}^2_x } \text{.}
\end{align*}
Combining the above and recalling the smallness of $D$ yields \eqref{eq.ell_est}.
\end{proof}

Using Lemma \ref{thm.ell_est}, we can now extend the weak Bernstein inequalities in Proposition \ref{thm.glp_bernstein} to the sharp cases $q = \infty$ and $q^\prime = 1$.

\begin{lemma} \label{thm.glp_bernstein_sh}
Assume $(\mc{N}, \gamma)$ satisfies \ass{R1}{C, N} and \ass{K}{C, D}, with $D \ll 1$ sufficiently small.
Also, suppose $k \geq 0$ is an integer, and let $F \in \mc{C}^\infty T^r_l \mc{S}$.
\begin{itemize}
\item The following estimates hold with respect to $(\mc{S}, \bar{\gamma} [\tau])$:
\begin{align}
\label{eq.glp_bochner} \| \bar{\nabla}^2 \bar{P}_k F \|_{ \bar{L}^2_x } + \| \bar{P}_k \bar{\nabla}^2 F \|_{ \bar{L}^2_x } &\lesssim_{C, N, r, l} 2^{ 2 k } \| F \|_{ \bar{L}^2_x } \text{,} \\
\notag \| \bar{\nabla}^2 \bar{P}_{< 0} F \|_{ \bar{L}^2_x } + \| \bar{P}_{< 0} \bar{\nabla}^2 F \|_{ \bar{L}^2_x } &\lesssim_{C, N, r, l} \| F \|_{ \bar{L}^2_x } \text{.}
\end{align}

\item The following estimates hold with respect to $(\mc{S}, \bar{\gamma} [\tau])$:
\begin{alignat}{3}
\label{eq.glp_bernstein_sh} \| \bar{P}_k F \|_{ \bar{L}^\infty_x } &\lesssim_{C, N, r, l} 2^k \| F \|_{ \bar{L}^2_x } \text{,} &\qquad \| \bar{P}_{< 0} F \|_{ \bar{L}^\infty_x } &\lesssim_{C, N, r, l} \| F \|_{ \bar{L}^2_x } \text{,} \\
\notag \| \bar{P}_k F \|_{ \bar{L}^2_x } &\lesssim_{C, N, r, l} 2^k \| F \|_{ \bar{L}^1_x } \text{,} &\qquad \| \bar{P}_{< 0} F \|_{ \bar{L}^2_x } &\lesssim_{C, N, r, l} \| F \|_{ \bar{L}^1_x } \text{.}
\end{alignat}
\end{itemize}
\end{lemma}

\begin{proof}
First, \eqref{eq.glp_bochner} follows from \eqref{eq.ell_est} in the same manner that \eqref{eq.glp_bochner_sc} follows from \eqref{eq.ell_est_sc}.
Next, the first part of \eqref{eq.glp_bernstein_sh} follows from Proposition \ref{thm.glp}, \eqref{eq.gns_2}, and \eqref{eq.glp_bochner}:
\[ \| \bar{P}_k F \|_{ \bar{L}^\infty_x } \lesssim \| \bar{\nabla}^2 \bar{P}_k F \|_{ \bar{L}^2_x }^\frac{1}{2} \| \bar{P}_k F \|_{ \bar{L}^2_x }^\frac{1}{2} + \| \bar{P}_k F \|_{ \bar{L}^2_x } \lesssim 2^k \| F \|_{ \bar{L}^2_x } \text{.} \]
The low-frequency analogue follows by a similar proof.
The remaining $L^2$-$L^1$-estimates follow from the corresponding $L^\infty$-$L^2$-estimates by duality.
\end{proof}

As a consequence of the above, we can derive much stronger versions of the intermediate estimates \eqref{eq.intertwining_weak}.
These are proved in the following two lemmas.

\begin{lemma} \label{thm.intertwining_nabla}
Assume $(\mc{N}, \gamma)$ satisfies \ass{R1}{C, N} and \ass{K}{C, D}, with $D \ll 1$ sufficiently small.
Given $k, m \geq 0$ and $F \in \mc{C}^\infty T^r_l \mc{S}$, we have, with respect to $(\mc{S}, \bar{\gamma} [\tau])$,
\begin{align}
\label{eq.intertwining_nabla} \| \bar{P}_k \bar{\nabla} \bar{P}_m F \|_{ \bar{L}^2_x } &\lesssim_{C, N, r, l} 2^{- |k - m|} 2^{ \min (k, m) } \| \bar{P}_{\sim m} F \|_{ \bar{L}^2_x } \text{,} \\
\notag \| \bar{P}_k \bar{\nabla} \bar{P}_{< 0} F \|_{ \bar{L}^2_x } &\lesssim_{C, N, r, l} 2^{-k} \| \bar{P}_{\lesssim 0} F \|_{ \bar{L}^2_x } \text{,} \\
\notag \| \bar{P}_{< 0} \bar{\nabla} \bar{P}_m F \|_{ \bar{L}^2_x } &\lesssim_{C, N, r, l} 2^{-m} \| \bar{P}_{\sim m} F \|_{ \bar{L}^2_x } \text{,} \\
\notag \| \bar{P}_{< 0} \bar{\nabla} \bar{P}_{< 0} F \|_{ \bar{L}^2_x } &\lesssim_{C, N} \| \bar{P}_{\lesssim 0} F \|_{ \bar{L}^2_x } \text{.}
\end{align}
\end{lemma}

\begin{proof}
We prove only the first estimate, as the remaining bounds are similarly proved and are easier.
First, if $k \geq m$, we can apply \eqref{eq.glp_fbr} and \eqref{eq.glp_bochner}:
\begin{align*}
\| \bar{P}_k \bar{\nabla} \bar{P}_m F \|_{ \bar{L}^2_x } &\lesssim 2^{-k} \| \bar{\nabla}^2 \bar{P}_m F \|_{ \bar{L}^2_x } \lesssim 2^{-k + 2m} \| \bar{P}_{\sim m} F \|_{ \bar{L}^2_x } \text{.}
\end{align*}
The $k \geq m$ case follows from the above by duality.
\end{proof}

\begin{lemma} \label{thm.intertwining_hodge}
Assume $(\mc{N}, \gamma)$ satisfies \ass{R1}{C, N} and \ass{K}{C, D}, with $D \ll 1$ sufficiently small.
Let $\bar{\mc{D}}$ denote any one of $\bar{\mc{D}}_1$, $\bar{\mc{D}}_2$, $\bar{\mc{D}}_1^\ast$, $\bar{\mc{D}}_2^\ast$, and let $X$ be a smooth section of the appropriate Hodge bundle on $(\mc{S}, \bar{\gamma} [\tau])$.
Then, for any integers $k, m \geq 0$, we have, with respect to $(\mc{S}, \bar{\gamma} [\tau])$, the estimates
\begin{align}
\label{eq.intertwining_hodge} \| \bar{P}_k \bar{\mc{D}}^{-1} \bar{P}_m X \|_{ \bar{L}^2_x } &\lesssim_{C, N} 2^{- \max(k, m)} 2^{- |k - m|} \| \bar{P}_{\sim m} X \|_{ \bar{L}^2_x } \text{,} \\
\notag \| \bar{P}_k \bar{\mc{D}}^{-1} \bar{P}_{< 0} X \|_{ \bar{L}^2_x } &\lesssim_{C, N} 2^{-2 k} \| \bar{P}_{\lesssim 0} X \|_{ \bar{L}^2_x } \text{,} \\
\notag \| \bar{P}_{< 0} \bar{\mc{D}}^{-1} \bar{P}_m X \|_{ \bar{L}^2_x } &\lesssim_{C, N} 2^{-2 m} \| \bar{P}_{\sim m} X \|_{ \bar{L}^2_x } \text{,} \\
\notag \| \bar{P}_{< 0} \bar{\mc{D}}^{-1} \bar{P}_{< 0} X \|_{ \bar{L}^2_x } &\lesssim_{C, N} \| \bar{P}_{\lesssim 0} X \|_{ \bar{L}^2_x } \text{.}
\end{align}
\end{lemma}

\begin{proof}
Again, like for Lemma \ref{thm.intertwining_nabla}, we prove only the first (and the most difficult) estimate.
Consider first the case $k \geq m$.
Applying \eqref{eq.hodge_sq} and \eqref{eq.glp_fbl}, we obtain
\begin{align*}
\| \bar{P}_k \bar{\mc{D}}^{-1} \bar{P}_m X \|_{ \bar{L}^2_x } &\lesssim 2^{-2 k} \| \bar{\lapl} \bar{\mc{D}}^{-1} \bar{P}_m X \|_{ \bar{L}^2_x } \lesssim 2^{-2 k} ( \| \bar{\mc{D}}^\ast \bar{P}_m X \|_{ \bar{L}^2_x } + \| \bar{\mc{K}} \bar{P}_m X \|_{ \bar{L}^2_x } ) \text{.}
\end{align*}
Applying \eqref{eq.glp_fb}, \eqref{eq.conf_curv_est}, and \eqref{eq.glp_bernstein_sh} yields
\begin{align*}
\| \bar{P}_k \bar{\mc{D}}^{-1} \bar{P}_m X \|_{ \bar{L}^2_x } &\lesssim 2^{-2k} ( 2^m \| \bar{P}_{\sim m} X \|_{ \bar{L}^2_x } + D \| \bar{P}_m X \|_{ \bar{L}^\infty_x } ) \lesssim 2^{-2k + m} \| \bar{P}_{\sim m} X \|_{ \bar{L}^2_x } \text{.}
\end{align*}
The remaining case $k \leq m$ follows from the above by duality.
\end{proof}

\subsection{Proof of the Estimates} \label{sec.cell_est}

We are now prepared to prove \eqref{eq.besov_embed}-\eqref{eq.besov_hodge}.
First of all, \eqref{eq.besov_embed} follows immediately from an L-P decomposition and \eqref{eq.glp_bernstein_sh}:
\[ \| \Psi \|_{ \bar{L}^{p, \infty}_{t, x} } \lesssim \sum_{k \geq 0} \| \bar{P}_k \Psi \|_{ \bar{L}^{p, \infty}_{t, x} } + \| \bar{P}_{< 0} \Psi \|_{ \bar{L}^{p, \infty}_{t, x} } \lesssim \sum_{k \geq 0} 2^k \| \bar{P}_k \Psi \|_{ \bar{L}^{p, 2}_{t, x} } + \| \Psi \|_{ \bar{L}^{p, 2}_{t, x} } \text{.} \qedhere \]
For \eqref{eq.besov_nabla} and \eqref{eq.besov_hodge}, we will prove only the case $a = 1$ here.
The dual case $a = \infty$ can be proved similarly, and the general case follows then from interpolation.

For the first part of \eqref{eq.besov_nabla}, we decompose the left-hand side and apply \eqref{eq.intertwining_nabla}:
\begin{align*}
\| \bar{\nabla} \Psi \|_{ \bar{B}^{1, p, s}_{\ell, t, x} } &\lesssim \sum_{k, m \geq 0} 2^{s k} \| \bar{P}_k \bar{\nabla} \bar{P}_m \Psi \|_{ \bar{L}^{p, 2}_{t, x} } + \sum_{k \geq 0} 2^{s k} \| \bar{P}_k \bar{\nabla} \bar{P}_{< 0} \Psi \|_{ \bar{L}^{p, 2}_{t, x} } \\
&\qquad + \sum_{m \geq 0} \| \bar{P}_{< 0} \bar{\nabla} \bar{P}_m \Psi \|_{ \bar{L}^{p, 2}_{t, x} } + \| \bar{P}_{< 0} \bar{\nabla} \bar{P}_{< 0} \Psi \|_{ \bar{L}^{p, 2}_{t, x} } \\
&\lesssim \sum_{k, m \geq 0} 2^{- |k - m|} 2^{ \min (k, m) } 2^{s k} \| \bar{P}_{\sim m} \Psi \|_{ \bar{L}^{p, 2}_{t, x} } + \sum_{k \geq 0} 2^{-k} 2^{s k} \| \bar{P}_{\lesssim 0} \Psi \|_{ \bar{L}^{p, 2}_{t, x} } \\
&\qquad + \sum_{m \geq 0} 2^{-m} \| \bar{P}_{\sim m} \Psi \|_{ \bar{L}^{p, 2}_{t, x} } + \| \bar{P}_{\lesssim 0} \Psi \|_{ \bar{L}^{p, 2}_{t, x} } \text{.}
\end{align*}
Evaluating the above sums, we see that
\begin{align*}
\| \bar{\nabla} \Psi \|_{ \bar{B}^{1, p, s}_{\ell, t, x} } &\lesssim \sum_{m \geq 0} 2^{(1 + s) m} \| \bar{P}_{\sim m} \Psi \|_{ \bar{L}^{p, 2}_{t, x} } + \| \bar{P}_{\lesssim 0} \Psi \|_{ \bar{L}^{p, 2}_{t, x} } \lesssim \| \Psi \|_{ \bar{B}^{1, p, s + 1}_{\ell, t, x} } \text{.}
\end{align*}

For the second inequality of \eqref{eq.besov_nabla}, we generate a Laplacian and decompose:
\begin{align*}
\| \Psi \|_{ \bar{B}^{1, p, s + 1}_{\ell, t, x} } &\lesssim \sum_{k \geq 0} 2^{- k} 2^{s k} \| \tilde{\bar{P}}_k \bar{\lapl} \Psi \|_{ \bar{L}^{p, 2}_{t, x} } + \| \bar{P}_{< 0} \Psi \|_{ \bar{L}^{p, 2}_{t, x} } \\
&\lesssim \sum_{k, m \geq 0} 2^{- k} 2^{s k} \| \tilde{\bar{P}}_k \bar{\nabla} \bar{P}_m \bar{\nabla} \Psi \|_{ \bar{L}^{p, 2}_{t, x} } + \sum_{k \geq 0} 2^{-k} 2^{s k} \| \tilde{\bar{P}}_k \bar{\nabla} \bar{P}_{< 0} \bar{\nabla} \Psi \|_{ \bar{L}^{p, 2}_{t, x} } \\
&\qquad + \| \bar{P}_{< 0} \Psi \|_{ \bar{L}^{p, 2}_{t, x} } \text{.}
\end{align*}
Applying again \eqref{eq.intertwining_nabla} yields
\begin{align*}
\| \Psi \|_{ \bar{B}^{1, p, s + 1}_{\ell, t, x} } &\lesssim \sum_{k, m \geq 0} 2^{- k} 2^{s k} 2^{ \min(k, m) } 2^{- |k - m|} \| \bar{P}_{\sim m} \bar{\nabla} \Psi \|_{ \bar{L}^{p, 2}_{t, x} } \\
&\qquad + \| \bar{P}_{\lesssim 0} \bar{\nabla} \Psi \|_{ \bar{L}^{p, 2}_{t, x} } + \| \Psi \|_{ \bar{L}^{p, 2}_{t, x} } \\
&\lesssim \sum_{m \geq 0} 2^{s m} \| \bar{P}_{\sim m} \bar{\nabla} \Psi \|_{ \bar{L}^{p, 2}_{t, x} } + \| \bar{P}_{\lesssim 0} \bar{\nabla} \Psi \|_{ \bar{L}^{p, 2}_{t, x} } + \| \Psi \|_{ \bar{L}^{p, 2}_{t, x} } \\
&\lesssim \| \bar{\nabla} \Psi \|_{ B^{1, p, s}_{\ell, t, x} } + \| \Psi \|_{ \bar{L}^{p, 2}_{t, x} } \text{.}
\end{align*}

Finally, for \eqref{eq.besov_hodge}, we apply yet another L-P decomposition:
\begin{align*}
\| \bar{\mc{D}}^{-1} \xi \|_{ \bar{B}^{1, p, s + 1}_{\ell, t, x} } &\lesssim \sum_{k, m \geq 0} 2^k 2^{s k} \| \bar{P}_k \bar{\mc{D}}^{-1} \bar{P}_m \xi \|_{ \bar{L}^{p, 2}_{t, x} } + \sum_{k \geq 0} 2^k 2^{s k} \| \bar{P}_k \bar{\mc{D}}^{-1} \bar{P}_{< 0} \xi \|_{ \bar{L}^{p, 2}_{t, x} } \\
&\qquad + \sum_{m \geq 0} \| \bar{P}_{< 0} \bar{\mc{D}}^{-1} \bar{P}_m \xi \|_{ \bar{L}^{p, 2}_{t, x} } + \| \bar{P}_{< 0} \bar{\mc{D}}^{-1} \bar{P}_{< 0} \xi \|_{ \bar{L}^{p, 2}_{t, x} } \text{.}
\end{align*}
Applying \eqref{eq.glp_bdd}, \eqref{eq.hodge_inv_est}, and \eqref{eq.intertwining_hodge}, and then summing, we have
\begin{align*}
\| \bar{\mc{D}}^{-1} \xi \|_{ \bar{B}^{1, p, s + 1}_{\ell, t, x} } &\lesssim \sum_{k, m \geq 0} 2^k 2^{s k} 2^{- |k - m|} 2^{- \max(k, m)} \| \bar{P}_{\sim m} \xi \|_{ \bar{L}^{p, 2}_{t, x} } \\
&\qquad + \sum_{k \geq 0} 2^{-k} 2^{s k} \| \bar{P}_{\lesssim 0} \xi \|_{ \bar{L}^{p, 2}_{t, x} } + \sum_{m \geq 0} 2^{-2 m} \| \bar{P}_{\sim m} \xi \|_{ \bar{L}^{p, 2}_{t, x} } + \| \bar{P}_{\lesssim 0} \xi \|_{ \bar{L}^{p, 2}_{t, x} } \\
&\lesssim \sum_{m \geq 0} 2^{s m} \| \bar{P}_{\sim m} \xi \|_{ \bar{L}^{p, 2}_{t, x} } + \| \bar{P}_{\lesssim 0} \xi \| _{ \bar{L}^{p, 2}_{t, x} } \\
&\lesssim \| \xi \|_{ \bar{B}^{1, p, s}_{\ell, t, x} } \text{.}
\end{align*}
This completes the proof of Proposition \ref{thm.conf}.

\raggedright
\bibliographystyle{amsplain}
\bibliography{articles,books,misc}

\end{document}